\documentclass[
a4paper,
12pt,
DIV=14,
toc=bibliography, 
headsepline, 
parskip=half-,
abstract=true,
]{scrartcl}

\usepackage{amsmath,amsfonts,amsthm,amssymb,nicefrac}
\usepackage{graphicx,color,url}
\usepackage{xcolor}
\usepackage{accents}
\usepackage{enumerate}
\usepackage{comment}
\usepackage{scrlayer-scrpage} 

\newtheorem{theorem}{Theorem}
\newtheorem{lemma}[theorem]{Lemma}
\newtheorem{corollary}[theorem]{Corollary}
\newtheorem{proposition}[theorem]{Proposition}
\theoremstyle{definition}
\newtheorem{remark}[theorem]{Remark}

\newtheorem{claim}{Claim}
\numberwithin{equation}{section}\numberwithin{theorem}{section}
\allowdisplaybreaks

\newcounter{stepctr}
{\end{list}}

\def\XXint#1#2#3{{\setbox0=\hbox{$#1{#2#3}{\int}$}
     \vcenter{\hbox{$#2#3$}}\kern-.5\wd0}}

\newcommand{\mres}{\mathbin{\vrule height 1.6ex depth 0pt width 0.13ex\vrule height 0.13ex depth 0pt width 1.3ex}}

\newcommand{\e}{\varepsilon}

\DeclareMathOperator{\spt}{spt}

\providecommand{\titlemacro}{{Brakke Regularity for the Allen--Cahn Flow}}
\title{\titlemacro}
\author{Huy The Nguyen and Shengwen Wang}

\ohead{\titlemacro}
\ihead{H.\;T.\;Nguyen, S.\;Wang.}
\setkomafont{pageheadfoot}{\footnotesize}
\newcommand\printaddress{{
\setlength{\parindent}{15pt}
\setlength{\parskip}{2.5ex}
\footnotesize~
\par
{
\par
{\scshape Huy The Nguyen}
\newline 
Queen Mary University of London, 
School of Mathematical Sciences, 
Mile End Road, 
London E1 4NS, UK
\newline
\textit{E-mail address:} 
\texttt{h.nguyen@qmul.ac.uk}
\par
{\scshape Shengwen Wang}
\newline 
Queen Mary University of London, 
School of Mathematical Sciences, 
Mile End Road, 
London E1 4NS, UK
\newline
\textit{E-mail address:} 
\texttt{shengwen.wang@qmul.ac.uk}
\par
}} 
}

\date{}

\begin{document}

\title{Brakke Regularity for the Allen--Cahn Flow}
\maketitle
\begin{abstract}
In this paper we prove an analogue of the Brakke's $\e$-regularity theorem for the parabolic Allen--Cahn equation. In particular, we show uniform $C^{2,\alpha}$ regularity for the transition layers converging to smooth mean curvature flows as $\e\rightarrow0$. The proof utilises Allen--Cahn versions of the monotonicity formula, parabolic Lipschitz approximation and blowups. A corresponding gap theorem for entire eternal solutions of the parabolic Allen--Cahn is also obtained. As an application of the regularity theorem, we give an affirmative answer to a question of Ilmanen \cite[13.4]{Ilmanen1994} that there is no cancellation in $\mathbf {BV}$ convergence in the mean convex setting.
\end{abstract}

\tableofcontents

\section{Introduction}

The $\e$-parabolic Allen--Cahn equation is
\begin{align}
\partial_t u^\e &=\Delta u^\e-\frac{1}{\e^2}W'(u^\e), \quad u(x,0)=u^\e_0(x),  
\end{align}
where $W:\mathbb R\rightarrow\mathbb R$ is a double-well potential function (a typical choice is $W(u)=\frac{(1-u^2)^2}{2}$). This flow was introduced by Allen and Cahn in 1979 \cite{Allen1979} to model the motion of phase boundaries by surface tension. This flow is the gradient flow of the energy
\begin{align*}
\mu^\e&=\int \left(\frac{\e|\nabla u^\e|^2}{2}+ \frac{W(u^\e)}{\e}\right) dx.
\end{align*}  
Heuristically, the term $\frac{1}{\e^2}W'(u^\e)$ forces $u^\e$ to approximate a characteristic function with a transition layer of width $\e$ and slope $ C/\e$. As $\e$ approaches zero, the nodal set approaches a hypersurface interface moving by mean curvature. This was derived formally by \cite{Fife1988} and \cite{Rubinstein1989}. In the radially symmetric case, a rigorous proof was obtained by \cite{Bronsard1991}. Under the assumption that the underlying mean curvature flow exists and is smooth, convergence was shown by  
\cite{Mottoni1989}, \cite{Chen1992} and \cite{Chen1994} among others. If the mean curvature flow is not smooth, there are various notions of weak mean curvature flow, we mention \cite{brakke2015motion}, \cite{Chen1991}, \cite{Evans1991}. For example \cite{Evans1992} showed that the limit of the level sets of the parabolic Allen--Cahn are contained in the viscosity solutions for the mean curvature flow studied by \cite{Evans1991} and \cite{Chen1991}. Most relevant for this paper is the paper of \cite{Ilmanen1993}, where it was shown, using methods from geometric measure theory, that the limit is a mean curvature flow in the sense of Brakke, that is the energy measures $d\mu^\e_t:=\frac{\e|\nabla u^\e|^2}{2}+ \frac{W(u^\e)}{\e}$ converge to the Brakke's weak mean curvature flow \cite{brakke2015motion}. Also of vital importance is Tonegawa's proof that the limit is an integral varifold \cite{ tonegawa2003integrality}. 

Higher regularity for the convergence is a necessary ingredient for geometric applications. In the elliptic setting, Caffarelli--Cordoba \cite{caffarelli2006phase} have shown that the transition layers of stable phase transitions have uniform $C^{1,\theta}$ regularity (independent of $\e$) and Wang--Wei \cite{Wang2019a, Wang2019} have proved that stable transition layers converge in a stronger $C^{2,\alpha}$ sense to the limit minimal surfaces. Using improved convergence in dimension $3$, Chodosh--Mantoulidis \cite{Chodosh2018} showed that the min-max minimal surfaces obtained from the Allen--Cahn construction in a generic 3-manifold has multiplicity $1$ and expected index. 

We will develop here the $\e$-regularity theory of parabolic Allen--Cahn equations, which is the diffused analogue of Brakke's regularity theory \cite{brakke2015motion} for mean curvature flow. Our main theorem is the following
\begin{theorem}\label{MainRegularityTheorem}
For any $b\in(0,1)$, there exists $\e_0,\tau_0>0$ small, and $R_0, K_0>0$ such that the following holds: let $u^\e$ be a solution to \eqref{eqn_ACFe} with non-positive discrepancy (see section 2 for the definition) in the parabolic cylinder $P_{R_0}(0,0)=B^{n+1}_{R_0}(0)\times(-R_0^2,R_0^2)\subset\mathbb R^{n+1}\times\mathbb R$, with $|u^\e(0,0)|\leq 1-b$ such that the following energy estimate holds 
\begin{align}\label{AlmostUnit}
\frac{1}{2}R_0^{-n-2}\iint_{P_{R_0}(0,0)}\left(\frac{\e|\nabla u^\e|^2}{2}+\frac{W(u^\e)}{\e}\right) dxdt\leq (1+\tau_0)\alpha\omega_n,
\end{align}
where $\alpha=\int_{-1}^1\sqrt{2W(g)} dx$ is the total energy of the of the $1$-dimensional solution (see \eqref{eqn_oneD}).

Then there exists a hyperplane (without loss of generality, assumed to be $\{x_{n+1}=0\}$), such that for any $s\in(-1+b,1-b)$, $|\nabla u^\e(x,t)|\neq0$ for $u^\e(x,t)=s$, and the level sets $\{u=s\}\cap P_1$ can be represented by a $C^{2,\alpha}_P$ graph of the form $x_{n+1}=h^{\e,s}(x_1,...,x_n,t)=h^{\e,s}(\hat x,t)$ with
\begin{align*}
\|h^{\e,s}\|_{C^{2,\alpha}(\hat P_1)}\leq K_0,
\end{align*}
where $\hat P_r=B^n_r\times(-r^2,r^2)\subset\mathbb R^n\times\mathbb R$ and $C^{2,\alpha}_P$ is the parabolic $C^{2,\alpha}$ (with respect to the parabolic metric $d\left((x_1,t_1),(x_2,t_2)\right)=\max\{|x_1-x_2|,\sqrt{|t_1-t_2|}\}$ in space-time).
\end{theorem}
\begin{remark}
We will prove uniform Lipschitz estimates in this paper. The improvement in regularity from Lipschitz to uniform $C^{2,\alpha}$ was proved in \cite{Nguyen2020}.
\end{remark}

As in the case of minimal surfaces, mean curvature flow and elliptic Allen--Cahn (see \cite{allard1972, brakke2015motion, wang2014new}), the key to proving uniform regularity is to obtain an excess decay property, Theorem \ref{thm_ExcessDecay}.

One of the ingredients of the proof is the following gap theorem for entire eternal parabolic Allen--Cahn equations. This theorem is of interest on its own.
\begin{theorem}\label{GapTheorem}
There exists $\tau(n)>0$ depending only on the dimension (and independent of $\e$) such that if $u^\e:\mathbb R^{n+1}\times\mathbb R\rightarrow[-1,1]$ is an eternal entire solution to \eqref{eqn_ACFe} with
\begin{align*}
\frac{1}{2}R^{-n-2}\iint_{P_{R}(0,0)}\left(\frac{\e|\nabla u^\e|^2}{2}+\frac{W(u^\e)}{\e}\right) dxdt\leq (1+\tau(n))\alpha\omega_n, \forall R>0,
\end{align*}
then $u^\e(x,t)=\tanh(\frac{e\cdot x+s_0}{\e})$  for some unit vector $e\in\mathbb S^n$ and $s_0\in\mathbb R$.
\end{theorem}
\begin{remark}
Unlike the proof of Allard's regularity minimal surfaces or Brakke's regularity for mean curvature flow, where the gap theorem is obtained as a corollary of the regularity theorem, here the gap theorem is proved first from the excess decay, and then it is used in the proof of the main regularity theorem. This difference is due to that fact that the excess decay assumption requires the technical assumption \eqref{TechnicalAssumption}, and consequently does not directly give uniform regularity.
\end{remark}

As an application to this regularity theory, we give an affirmative answer to a question of Ilmanen (\cite[Question 4 of section 13]{Ilmanen1993}) about the \textbf{BV} convergence of a sequence of Allen--Cahn equations as the parameter $\e\rightarrow0$ in the mean convex case. This is the Allen--Cahn analogue of no cancellation for the Brakke flow \cite{Metzger2008}.
\begin{theorem}[Strong Convergence for $H>0$]\label{NoCancelation}
Suppose $M_0^n\subset\mathbb R^{n+1}$ is a smooth strictly mean convex closed hypersurface. Then there exists a sequence of smooth functions (constructed in \cite[1.4]{Ilmanen1993}) $u_0^\e:\mathbb R^{n+1}\rightarrow[-1,1]$ such that $u_0^\e\rightarrow2\chi_{E_0}-1$ in $\mathbf{BV}_{loc}(\mathbb R^{n+1})$, where $E_0$ is the region enclosed by $M_0$ (see also the beginning of section \ref{PfNoCancelation} for a summary of the properties satisfied by this approximation). Furthermore the sequence of solutions $u^\e:\mathbb R^{n+1}\times\mathbb R^+\rightarrow[-1,1]$ to \eqref{eqn_ACFe} with initial condition $u^\e(\cdot,0)=u^\e_0$ constructed above satisfies
\begin{align*}
u^\e\rightarrow2\chi_{E}-1
\end{align*}
in $\mathbf{BV}_{loc}(\mathbb R^{n+1}\times\mathbb R^+)$ for some open subset $E\subset \mathbb R^{n+1}\times\mathbb R^+$, with $E\cap\left(\mathbb R^{n+1}\times\{0\}\right)=E_0\times\{0\}$.

Then we have the conclusion that there is no cancellation in the $\mathbf{BV}$ convergence, namely
\begin{align*}
|\nabla u^\e|\rightarrow|\nabla u|=\frac{1}{\alpha}d\mu
\end{align*}
in $L^1_{loc}(\mathbb R^{n+1}\times\mathbb R^+)$.
\end{theorem}
The paper is organised as follows. In section \ref{sec_prelim}, we fix our notation and define quantities that will be used throughout the paper. In section \ref{sec_StatementExcessDecay}, we state the main technical part of the paper, the excess decay theorem. The proof will proceed by contradiction, so will assume that the result of the Excess decay theorem does not hold. The majority of this paper is concerned with the proof of this theorem and we give a very detailed overview of the rest of the paper. In section \ref{sec_Brakke}, we further develop some of the geometric measure theory of the $\e$-Allen--Cahn equation. In particular, we prove the $L^2$-$L^\infty$ version of $\e$-Allen--Cahn equation which is modification of the mean curvature flow version of Ecker \cite{Ecker2004} (see also \cite{Kasai2014}). In section \ref{sec_EnergyEstimates}, we prove the key Caccioppoli type inequality and the energy estimate (which can be viewed as the parabolic version of Caccioppoli type inequality). In section \ref{sec_ParabolicLipschitz}, we prove a Lipschitz approximation of the level sets of the solution to the parabolic Allen--Cahn equation. In section \ref{sec_ExcessDecay}, we prove that level sets of $\e$-Allen--Cahn equation are well approximated by solutions of the heat equation and derive our contradiction. In section \ref{sec_Gap}, we prove Theorem \ref{GapTheorem}, which allows us to obtain the the level sets are Lipschitz regular. Since we require the assumption \eqref{TechnicalAssumption}, we can not conclude the level sets are H\"{o}lder regular. Instead in section \ref{sec_Holder}, we show how to appeal to our previous paper \cite{Nguyen2020} to conclude full $C^{2,\alpha}$ regularity. Finally, in section \ref{PfNoCancelation} we show how we can provide an affirmative answer to Ilmanen's question through a proof of Theorem \ref{NoCancelation}.

$\textbf{Acknowledgements.}$
H.T.Nguyen and S.Wang were supported by the EPSRC grant EP/S012907/1.

\section{Preliminaries and notations}\label{sec_prelim}
We say $u^\e:\mathbb R^{n+1}\times\mathbb R\rightarrow\mathbb R$ is a solution to the $\e$-parabolic Allen--Cahn equation if it satisfies
\begin{align}\label{eqn_ACFe}
\partial_t u^\e &=\Delta u^\e-\frac{1}{\e^2}W'(u^\e), \quad u(x,0)=u^\e_0(x).  
\end{align}
We will omit the super-script if $\e=1$ in which case $u$ satisfies 
\begin{align}\label{eqn_ACF1}
\partial_t u &=\Delta u-W'(u), \quad u(x,0)=u_0(x),  
\end{align}
the scaled equation with $\e=1$.

We will use the notation $ \hat x=(x_1, \dots, x_n)$ to denote the first $n$ coordinates and write $(\hat x,x_{n+1})=(x_1,\dots,x_n,x_{n+1})\in\mathbb R^{n+1}$ and denote by $B_r=\{(\hat x, x_{n+1})\in\mathbb R^{n+1}:|\hat x|^2+x_{n+1}^2\leq r^2\}$, $\hat C_r=\{(\hat x,x_{n+1})\in\mathbb R^{n+1}: | \hat x|^2\leq r^2\}$ and $\hat B_r=\{\hat x\in\mathbb R^n: |\hat x|^2\leq r^2\}$.

We also denote the parabolic cylinder of radius $r$ by $P_r=\{(x,t)\in\mathbb R^{n+1}\times\mathbb R: |x|\leq r, |t|\leq r^2\}$.

For any $T\in\mathbf G(n+1,n)$ a hyperplane, the cylinder $C(T,r)$ of radius $r$ with respect to $T$ is defined by $C(T,r)=\{x\in\mathbb R^{n+1}:|T(x)|\leq r,|T^\perp(x)|\leq r\}$. For a solution $u^\e$ to \eqref{eqn_ACFe} in $\mathbb R^{n+1}\times\mathbb R$, we define the following geometric quantities for subsets $B_r\subset\mathbb R^{n+1}$ and  $P_r\subset\mathbb R^{n+1}\times\mathbb R$ that will be used in this paper,
\begin{enumerate}[(i)]
\item The energy
\begin{align*}
\mu^\e(B_r)&=\int_{B_r}\left(\frac{\e|\nabla u^\e|^2}{2}+ \frac{W(u^\e)}{\e}\right) dx\\
\mu^\e(P_r)&=\iint_{P_r}\left(\frac{\e|\nabla u^\e|^2}{2}+ \frac{W(u^\e)}{\e}\right) dxdt,
\end{align*}
\item The tilt-excess
\begin{align*}
E^\e(B_r)&=r^{-n}\int_{B_r}\left(1-\left(\frac{\nabla u^\e}{|\nabla u^\e|}\cdot e_{n+1}\right)^2\right)\e|\nabla u^\e|^2 dx\\
&=r^{-n}\int_{B_r}(1-\nu_{n+1}^2)\e|\nabla u^\e|^2 dx\\
E^\e(P_r)&=r^{-n-2}\iint_{P_r}\left(1-\left(\frac{\nabla u^\e}{|\nabla u^\e|}\cdot e_{n+1}\right)^2\right)\e|\nabla u^\e|^2 dxdt\\
&=r^{-n-2}\iint_{P_r}(1-\nu_{n+1}^2)\e|\nabla u^\e|^2 dxdt,
\end{align*}
where for simplicity we denote by the normal vector to the level sets of $u$ by $\nu=\frac{\nabla u^\e}{|\nabla u^\e|}$,
\item The height-excess
\begin{align*}
\mathbb H^\e(B_r)&=r^{-n-2}\int_{B_r}x_{n+1}^2\e|\nabla u^\e|^2 dx\\
\mathbb H^\e(P_r)&=r^{-n-4}\iint_{P_r}x_{n+1}^2\e|\nabla u^\e|^2 dxdt,
\end{align*}
\item The $L^2$ norm of time derivative ($\e$-Willmore term)
\begin{align*}
\mathcal{W}^\e(B_r)=\int_{B_r}\e\left(\Delta u^\e-\frac{W(u^\e)}{\e^2}\right)^2 dx=\int_{B_1}\e\left(\frac{\partial u}{\partial t}\right)^2 dx\\
\mathcal{W}^\e(P_r)=\iint_{P_r}\e\left(\Delta u^\e-\frac{W(u^\e)}{\e^2}\right)^2 dxdt=\iint_{P_r}\e\left(\frac{\partial u}{\partial t}\right)^2 dxdt,
\end{align*}
\item  The discrepancy
\begin{align*}
\xi^\e=\frac{\e|\nabla u|^2}{2}-\frac{W(u)}{\e}
\end{align*}
and the discrepancy measure
\begin{align*}
d \xi^\e=\left(\frac{\e|\nabla u|^2}{2}-\frac{W(u)}{\e}\right)dx.
\end{align*}
\end{enumerate}
For the double-well potential $W(u)=\frac{(1-u^2)^2}{4}$, there is an explicit solution, the $1$-d standing wave solution $g^\e(x)=\tanh(\frac{x}{\e})$, to the equation $g''-\frac{W'(g)}{\e^2}=g''+g(1-g^2)=0$. For simplicity we denote $g(x)=\tanh(x)$ when $\e=1$. The total energy of this 1-d solution is 
\begin{align}\label{eqn_oneD}
\alpha:=\int_{-\infty}^\infty[(g^\e)']^2 dx=\int_{-\infty}^\infty[g']^2 dx.
\end{align}
Throughout this paper, we assume the discrepancy satisfies $ \xi^\e\leq 0$, that is 
\begin{align}\label{eqn_discrepancy}
\frac{\e|\nabla u(x,t)|^2}{2}\leq \frac{W(u(x,t))}{\e}, \forall t\in\mathbb R
\end{align}
which is the same technical assumption made in \cite{Ilmanen1993} that simplifies our argument. It is shown in \cite{Ilmanen1993}, that if the discrepancy is initially non-positive, then it stays non-positive along the flow. We remark here that this assumption may be removed by a result of Soner \cite{soner1997ginzburg} who showed that even if the discrepancy is initially not non-positive, it exponentially decays to zero. 

\section{Statement of Excess Decay}\label{sec_StatementExcessDecay}
The following Excess decay theorem shows us that if the height excess is small over some hyperplane (here $\mathbb R^n \times\{ 0\}\subset \mathbb R^{n+1}$) then by shrinking the radius of the parabolic ball and possibly tilting the hyperplane a little, the height excess becomes smaller. This theorem parallels Campanato's theorem for functions and will be used to prove the Lipschitz regularity of intermediate layers. 
\begin{theorem}[Excess Decay]\label{thm_ExcessDecay}
Given $b\in(0,1)$, there exists constants $\e_0,\delta_0>0,\tau_0$ small, $\theta\in(0,\frac{1}{4})$, $K_1>0$ large with the following property : let $u^\e$ be a solution to \eqref{eqn_ACFe} with $\e\leq \e_0$ in the parabolic ball $P_1\subset\mathbb R^{n+1}\times\mathbb R$, that satisfies the non-positive discrepancy condition \eqref{eqn_discrepancy}, and 
\begin{align}
\label{NonTrivialEnd}|u^\e(0,0)|&\leq 1-b,\\
\label{SingleLayer}\int_{B_1}\left(\frac{\e|\nabla u^\e|^2}{2}+ \frac{W(u^\e)}{\e}\right) dx&\leq \alpha\omega_n(1+\tau_0),\\
\label{SmallHeight} \iint_{P_1}x_{n+1}^2\e|\nabla u^\e|^2 dxdt&\leq \delta_0.
\end{align}
Furthermore we assume
\begin{align}\label{LayerRepulsion}
\iint_{P_1}x_{n+1}^2\e|\nabla u^\e|^2 dxdt&\geq K_1\e^2.
\end{align}
Then there exists $T\in G(n+1,n), \lambda\in\mathbb R$ such that 
\begin{align}\label{ExcessDecay}
\theta^{-n-4}\iint_{P_\theta}|T^\perp(x)-\lambda|^2\e|\nabla u^\e|^2 dxdt&\leq \frac{\theta}{2} \iint_{P_1}x_{n+1}^2\e|\nabla u^\e|^2 dxdt.
\end{align}
Moreover there exists a universal constant $C$ such that 
\begin{align}\label{PointwiseTilt}
\| e-e_{n+1}\|\leq C \left(\iint_{P_1}x_{n+1}^2\e|\nabla u^\e|^2 dxdt\right)^{1/2}
\end{align}
where $e$ is the unit normal to $T$.
\end{theorem}

\begin{remark} 
\begin{enumerate}[(i)]

\item The condition \eqref{SingleLayer} says that the generalized varifold associated to $u^\e$ has area ratio very close to that of the n-dimensional Euclidean hyperplane. In particular, it consists of a single transition layer. This corresponds to the small area ratio condition in Allard's and Brakke's regularity theorems. 

\item The condition \eqref{NonTrivialEnd} says that measure theoretically, there is a nontrivial part of the varifold at some point along the flow.

\item We can in fact show \eqref{SmallHeight} holds with respect to a suitable hyperplane provided that \eqref{SingleLayer} holds with sufficiently small $\tau_0$ (by Lemma \ref{ExcessConvergence} and Proposition \ref{Comparable}). So in the proof of Theorem \ref{MainRegularityTheorem} when Theorem \ref{thm_ExcessDecay} will be applied, if $\tau_0$ in \eqref{AlmostUnit} is chosen small enough, condition \eqref{SmallHeight} is automatically satisfied.

\item We note that compared to Allard's regularity and Brakke's regularity theorems, the hypothesis \eqref{LayerRepulsion} does not seem so satisfactory. However, such a condition is required because in the Allen--Cahn equation separate layers can interact since, although energy is concentrated in transition region, it is still distributed in a strip of width $\e$. This phenomenon does not occur in minimal surface theory or mean curvature flow.   
\end{enumerate}
\end{remark}

The excess decay theorem will be proved in section \ref{sec_ExcessDecay} by a contradiction argument.

Here we will give an overview of the proof. We will consider here $\phi_T$ to be our test function and we use the integrated form of the Allen--Cahn equation \eqref{eqn_BrakkeAllenCahn1}, which gives us 
\begin{align*}
\frac{d}{dt}\int_{\mathbb R^{n+1}}\phi_T^2 d\mu^\e_t= \int_{\mathbb R^{n+1}}-\e \phi_T^2 \left(-\Delta u+\frac{W'(u)}{\e^2} \right)^2 +\e\langle \nabla \phi_T^2 ,\nabla u \rangle \left(-\Delta u+\frac{W'(u)}{\e^2} \right) dx.
\end{align*}

We note that the term $-\e \varphi \left(-\Delta u+\frac{W'(u)}{\e^2} \right)^2$ roughly corresponds to $-H^2d\mu_t$ of the mean curvature flow \cite{Ilmanen1993}. This term is clearly negative and has a dissipative effect on the left hand side. We will refer to this as a Willmore type term and we will want to control it using our height excess. A straightforward computation \eqref{DerivativeEnergy} gives us 
\begin{align*}
\frac{d}{dt}\int\phi_T^2\left(\frac{\e |\nabla u|^2}{2}+ \frac{W(u)}{\e}\right) dx\leq &-\frac{3}{4}\int\e\phi_T^2\left(\Delta u-\frac{W'(u)}{\e^2}\right)^2dx\\
&+4\int|\nabla\phi_T|^2(1-\nu_{n+1}^2)\e|\nabla u|^2dx
\end{align*}
The term $(1-\nu_{n+1}^2)\e|\nabla u|^2$ corresponds is the Tilt-Excess. Using the divergence structure and the Stress-Energy tensor associated to the Allen--Cahn energy, an integration by parts argument (Corollary \ref{TimeSliceCaccioppoli}) allows us to estimate this term, 
\begin{align*}
\frac{d}{dt}\int_{\mathbb R^{n+1}}\phi_T^2 d\mu^\e_t&\leq -\frac{1}{4}\int_{B_1}\e\left(\Delta u-\frac{W'(u)}{\e^2}\right)^2dx+C_0\int_{B_1}x_{n+1}^2\e|\nabla u|^2dx.
\end{align*}   
In order to utilise this inequality, we need a certain nonlinear inequality. In the varifold setting, this first appeared in 
\cite{brakke2015motion} (see also \cite{Kasai2014}). This estimate tells us that if $\frac{1}{4}\int_{B_1}\e\left(\Delta u-\tfrac{W'(u)}{\e^2}\right)^2dx$ is small and the area excess is small (that is the area is close to that of hyperplane) then we can show the area excess is small and in fact uniformly bounded with a leading order term which is the Willmore type energy ($L^2$ norm of the diffuse mean curvature) raised to the power $\tfrac{2n}{n-2}$. The precise statement can be found in Proposition \ref{InteriorSobolev}. The statement is independent of the Allen--Cahn flow and its proof uses a Lipschitz graphical decomposition of level sets. In order to get a rough idea of the proof, we will drop the lower order terms and if we let
\begin{align*}
E(t)=\int_{T} \phi_T^2 d\mu^\e_t
\end{align*}
then we get that if the energy difference from the flat solution is close to zero and $ W\leq W_1$
\begin{align*}
\frac{d}{dt}E(t) \leq c \mathbb H - c E(t)^{\frac{n-3}{n-1}}.
\end{align*}
 Otherwise, we have
 \begin{align*}
\frac{d}{dt}E(t) \leq c \mathbb H-c W_1.
\end{align*}
If we assume $2 c \mathbb H < c W_1$, we see that $E(t)$ decreases at least at a fixed rate, that is 
\begin{align*}
\frac{d}{dt} E(t) \leq - \frac{c}{2}W_1.
\end{align*}
If we set $ \tilde E(t)=E(t) - c t \mathbb H < E(t)$ we get
\begin{align*}
\frac{d}{dt} \tilde E(t)=\frac{d}{dt}E(t) - c \mathbb H \leq -c E(t)^{\frac{n-3}{n-1}} \leq -c \tilde E(t)^{\frac{n-3}{n-1}}.  
\end{align*}
Therefore if $\tilde E(-T+1)$ is less than $(1-\nu)\alpha$ but larger than $\frac{\alpha}{8}$, $\tilde E(t)$ will decrease at a rate $-\frac{c}{2}W_1$ or at a rate $ - c \tilde E(t)^{\frac{n-3}{n-1}}$. Since $\tfrac{n-3}{n-1}<1$ this implies $\tilde E(t)$ will vanish in finite time, that is $\tilde E(t)\equiv 0$ for $ t\geq -1$ as long as $T$ is sufficiently large. This in turn implies $E(t)\leq c \mathbb H$ after some fixed amount of time has elapsed. This is our first estimate and shows the energy excess is bounded by the height excess. 

In the actual argument, we show
 \begin{align*}
\sup \left | \mu^{\e}_{t}(\phi^2_T) - \alpha \int \phi_T^2d\mu^\e_t \right| \leq c \mathbb H. 
\end{align*}
If we re-examine our $\e$-Brakke Allen--Cahn equation, the change in time of $\int \phi_T^2 d\mu^\e_t$ bounds the Willmore type term $\iint \e \phi_T^2\left (\Delta u^\e- \frac{W(u^\e)}{\e^2}\right)^2 dx$ and hence bounds this term by the height excess, that is  
\begin{align}\label{eqn_WillmoreType}
\iint \e \phi_T^2\left (\Delta u^\e- \frac{W(u^\e)}{\e^2}\right)^2 dx\leq c \mathbb H. 
\end{align}
Then from Corollary \ref{TimeSliceCaccioppoli} allows use to estimate 
\begin{align*}
\int_{\mathbb{R}^{n+1}}\phi_T^2( 1- \nu^2)d \mu^\e_t \leq C_1\mathbb{H}  + C_2\sqrt{\mathbb H \mathcal W}.  
\end{align*}
Then the estimate \eqref{eqn_WillmoreType} allows us to bound
\begin{align*}
\int_{-1}^{1}\int_{\mathbb R^{n+1}} \phi_T^2(1-\nu^2) d\mu_t^\e dt \leq C \mathbb H. 
\end{align*}
This is the tilt-excess bound in terms of the height excess, which is the parabolic version of Caccioppoli inequality in \cite{wang2014new} and a parabolic Allen--Cahn version of \cite{brakke2015motion} and \cite{Kasai2014}. Compared to the elliptic case, we note that the estimates involved are more intricate since the $\e$-Brakke Allen--Cahn equation does not give us direct access to the tilt-excess. 

Once we have the tilt-excess-height excess bounds, we require the key estimate - the height excess decay. Roughly speaking, the height excess decay states that if the height excess with respect to some hyperplane in a ball is sufficiently small then shrinking the radius of the ball and perhaps tilting the hyperplane a little the excess decays. The excess decay will then be used to prove the uniform Lipschitz regularity of the intermediate layers. The proof proceeds by contradiction and involves a number of intricate steps which we summarise as follows
\begin{enumerate}[(i)]
\item We show the level sets $ \{ u^\e=x\}\mid s \in (-1+b, 1-b)$ can be represented by Lipschitz graphs over $\mathbb R^{n}$ of the form
\begin{align*}
x_{n+1}=h^{\e,s}(\hat x)
\end{align*} 
up to bad sets of small measure which are controlled by the excess. Here we use a parabolic Hardy--Littlewood maximal function. 

\item We rewrite the excess using the graphical co-ordinates and we show $ \frac{h^{\e,s}}{\mathbb H_\e}$ are uniformly bounded in $ H^1_{loc}(B_1)$. This shows us that we can assume $\frac{h^{\e,s}}{\mathbb H_\e} \rightharpoonup h^\infty$ weakly in $H^1_{loc}$. Here we will require the layer repulsion hypothesis $\mathbb H_\e \gg \e$ to guarantee that the limit is independent of $s$.

\item By choosing the test function $\phi(\hat x)\psi(x_{n+1})x_{n+1}$ in the $\e$-Brakke Allen--Cahn equation ($\phi\in C_c^\infty(B_1), \psi \in C_c^\infty((-1,1))$) and passing to the limit we show $h_\infty$ is in fact a solution to the heat equation on $B_1$.  

\item Furthermore, by choosing the test function $\phi(\hat x)\psi(x_{n+1})x_{n+1}^2$ in the $\e$-Brakke Allen--Cahn equation, we can show $\tfrac{h^{\e,s}}{\mathbb H_\e}$ converges strongly in $H^1_{loc}(B_1)$, but since heat equations satisfy the height excess decay this provides the contradiction and proves the height excess decay.
\end{enumerate}
Because of the layer repulsion hypothesis, $\mathbb H_\e \gg \e$ we cannot argue as in \cite{brakke2015motion}, \cite{Kasai2014} to conclude $C^{1,\alpha}$ regularity. We are able to only directly conclude uniform Lipschitz regularity of the level sets. Instead, we appeal to our previous paper \cite{Nguyen2020} which shows Lipschitz regularity of the level sets implies full $ C^{2,\alpha}$ regularity.
 
\section{Brakke's Equality and Huisken's Monotonicity}

\subsection{$\e$-Brakke Allen--Cahn Equation}\label{sec_Brakke}
The integral form of the parabolic Allen--Cahn equation (deduced in \cite{Ilmanen1993}) is
\begin{align}\label{eqn_BrakkeAllenCahn1}
&\frac{d}{dt}\int \varphi d\mu^\e_t \\
\nonumber&=\int \varphi \frac{\partial}{\partial t}\left( \frac{\e|\nabla u|^2}{2}+\frac{W(u)}{\e}\right) dx\\
\nonumber &=\int \e \varphi \left(-\Delta u+\frac{W'(u)}{\e^2} \right) \partial_t u-\e\langle \nabla \varphi, \nabla u \rangle \partial_tu dx\\
\nonumber  &=  \int -\e \varphi \left(-\Delta u+\frac{W'(u)}{\e^2} \right)^2 +\e\langle \nabla \varphi ,\nabla u \rangle \left(-\Delta u+\frac{W'(u)}{\e^2} \right) dx.
\end{align}
 If we let $ T_{ij}=\e \nabla_i u\nabla_j u-\left( \frac{\e|\nabla u|^2}{2}+ \frac{W(u)}{\e}\right) \delta_{ij}$ denote the stress-energy tensor, we take its divergence and obtain
 \begin{align*}
\nabla_i T_{ij}=\e \Delta u \nabla_j u-\frac{1}{\e}W'(u) \nabla_j u=\partial_t u \nabla_j u.
\end{align*}

Multiplying this equation by $\nabla_j \varphi$ and integrating by parts, we get 
 \begin{align}\label{eqn_BrakkeAllenCahn2}
\frac{d}{dt}\int \varphi d\mu^\e_t &=\int -\e \varphi \left(-\Delta u+\frac{W'(u)}{\e^2}\right)^2 d x\\
\nonumber&+\int \left(\e\nabla u\otimes\nabla u -\left(\frac{\e|\nabla u|^2}{2}+\frac{W(u)}{\e}\right)I \right):\nabla^2\varphi dx
\end{align}
for any $\phi\in C_c^\infty(\mathbb R^{n+1},\mathbb R^ +)$. 
This is the $\e$-version of Brakke's inequality in \cite{brakke2015motion}. Alternatively, if we let $\varphi=\langle\nabla u, g\rangle$ be a test function, where $g=(g^1,\dots,g^{n+1})$ is any compactly supported test $C^1$ vector field, we compute
\begin{align}\label{eqn_ACDivergence}
&\int \e\varphi \left(\Delta u-\frac{W'(u)}{\e^2}\right)dx\\
\nonumber&=\int \e\langle \nabla u,g\rangle \left(\Delta u-\frac{W'(u)}{\e^2}\right)dx\\
\nonumber&=\int \e\langle \nabla u,g\rangle \Delta u dx-\int \e\langle \nabla u,g\rangle \frac{W'(u)}{\e^2}dx\\
\nonumber&=-\int \e\left\langle \frac{1}{2}\nabla|\nabla u|^2,g\right\rangle dx-\int  \e \nabla u\otimes \nabla u:\nabla g dx-\int \langle \nabla u,g\rangle \frac{W'(u)}{\e}dx\\
\nonumber&=\int \frac{\e|\nabla u|^2}{2}\mathrm{div}g dx-\int \e \nabla u\otimes \nabla u:\nabla gdx+\int \frac{W(u)}{\e}\mathrm{div}g dx\\
\nonumber&=-\int \e \nabla u\otimes \nabla u:\nabla g+\int \left(\frac{\e|\nabla u|^2}{2}+ \frac{W(u)}{\e}\right)\mathrm{div}g dx\\
\nonumber&=-\int \e \sum_{i,j=1}^{n+1}u_{x_i}u_{x_j}g^j_{x_i}dx+\int \left(\frac{\e|\nabla u|^2}{2}+ \frac{W(u)}{\e}\right)\mathrm{div}g dx\\
&=\nonumber \int - \left (\frac{\e}{2} |\nabla u|^2 + \frac{W(u)}{\e} \right)(\delta-\nu \otimes\nu):\nabla gdx + \int \left (\frac{W(u)}{\e} + \frac{\e|\nabla u|^2}{2} \right)\nu\otimes\nu :\nabla g dx \\
&\nonumber \int-\left(\delta - \nu\otimes\nu\right):\nabla g d \mu^\e_t + \int \nu\otimes\nu:\nabla g d\xi^\e_t  
\end{align}
where $\nu=\frac{\nabla u}{|\nabla u|}$ and $d \xi^\e_t=\left(\frac{\e|\nabla u|^2}{2} -\frac{W(u)}{\e}\right)dx$ is the discrepancy measure.  
\subsection{Weighted Monotonicity Formula}
  
Here we will derive a weighted parabolic Allen--Cahn equation $\e$-version of Huisken's monotonicity formula \cite{Huisken1990} for the mean curvature flow. The unweighted $\e$-version may be found in \cite{Ilmanen1993}. Here we will denote by $\Phi_{y,s}$ Huisken's monotonicity kernel
\begin{align*}
\Phi=\Phi_{y,s}(x,t) \equiv \frac{1}{(4 \pi(s-t))^n}e^{-\tfrac{|x-y|^2}{4(s-t)}}, \quad t < s, x \in \mathbb R^{n+1}.
\end{align*}
\begin{theorem}[Weighted Monotonicity Formula] \label{thm_WMF}  Consider a solution to the parabolic Allen--Cahn equation \eqref{eqn_ACFe} and fix a point $(y,s) \in \mathbb R^{n+1}\times \mathbb R$. Let $\rho$ be a sufficiently smooth (possibly time-dependent) function. Then we have 
\begin{align}\label{eqn_WMF}
\frac{d}{dt}\int& \Phi_{y,s}\rho d\mu^\e_t=
\int-\e \rho \Phi_{y,s}\left(-\Delta u+\frac{W'(u)}{\e^2}-\frac{\nabla u \cdot \nabla \Phi_{y,s}}{\Phi_{y,s}}\right)^2 dx+\int  \frac{\rho \Phi_{y,s}}{2(s-t)}d\xi^\e_t\\
\nonumber&+\int \Phi_{y,s}\partial_t \rho d\mu^\e_t+\int\left(\e \nabla_i u \nabla_j u-\frac{\e|\nabla u|^2}{2}\delta_{ij}-\frac{W(u)}{\e}\delta_{ij}\right)\Phi_{y,s}\nabla_i\nabla_j \rho dx
\end{align} 
where $d\xi^\e_t=\left(\frac{\e|\nabla u|^2}{2}-\frac{W(u)}{\e}\right)dx$ is the discrepancy measure. 
\end{theorem}
\begin{proof}
We contract the stress-energy tensor of the Allen--Cahn functional and apply integration by parts to obtain the following equation 
\begin{align*}
&\int\left(\e \nabla_i u \nabla_j u-\frac{\e|\nabla u|^2}{2}\delta_{ij}-\frac{W(u)}{\e}\delta_{ij}\right)\phi \nabla_i\nabla_j \rho dx\\
 &=-\int \nabla_j \left(\e \nabla_i u \nabla_j u-\frac{\e|\nabla u|^2}{2}\delta_{ij}-\frac{W(u)\delta_{ij}}{\e}\right)  \phi \nabla_i \rho dx\\
&-\int \left(\e \nabla_i u \nabla_j u-\frac \e 2|\nabla u|^2 \delta_{ij}-\frac{W(u)}{\e}\delta_{ij}\right) \nabla_i \phi \nabla_j \rho dx\\
&=-\int \nabla_j u  \left(\e \Delta u-\frac{W'(u)}{\e}\right) \phi \nabla_j \rho dx\\
&-\int \left(\e \nabla_i u \nabla_j u-\frac{\e|\nabla u|^2}{2}\delta_{ij}-\frac{W(u)}{\e}\delta_{ij}\right) \nabla_i \phi \nabla_j \rho dx. 
\end{align*}
That is we have 
\begin{align*}
&\int\left(\e \nabla_i u \nabla_j u-\frac{\e|\nabla u|^2}{2}\delta_{ij}-\frac{W(u)}{\e}\delta_{ij}\right)\phi \nabla_i\nabla_j \rho dx+\int \nabla_j u  \left(\e \Delta u-\frac{W'(u)}{\e}\right) \phi \nabla_j \rho dx\\
&=-\int \left(\e \nabla_i u \nabla_j u-\frac{\e|\nabla u|^2}{2}\delta_{ij}-\frac{W(u)}{\e}\delta_{ij}\right) \nabla_i \phi \nabla_j \rho dx\\
&=\int\left(\e \nabla_i u \nabla_j u-\frac{\e|\nabla u|^2}{2}\delta_{ij}-\frac{W(u)}{\e}\delta_{ij}\right)\rho \nabla_i\nabla_j \phi dx+\int \nabla_j u  \left(\e \Delta u-\frac{W'(u)}{\e}\right) \rho \nabla_j \phi dx. 
\end{align*}
Rewriting the above equation, we get 
\begin{align*}
&\int\left(\e \nabla_i u \nabla_j u-\frac{\e|\nabla u|^2}{2}\delta_{ij}-\frac{W(u)}{\e}\delta_{ij}\right)\phi \nabla_i\nabla_j \rho dx+\int \nabla_j u  \left(\e \Delta u-\frac{W'(u)}{\e}\right) \phi \nabla_j \rho dx\\
&-\int\left(\e \nabla_i u \nabla_j u-\frac{\e|\nabla u|^2}{2}\delta_{ij}-\frac{W(u)}{\e}\delta_{ij}\right)\rho \nabla_i\nabla_j \phi dx-\int \nabla_j u  \left(\e \Delta u-\frac{W'(u)}{\e}\right) \rho \nabla_j \phi dx\\
&=0.
\end{align*}
We then compute
\begin{align}\label{eqn_one}
\nonumber\frac{d}{dt}\int \phi \rho d\mu^\e_t  &=\int-\e \phi \rho \left(-\Delta u+\frac{W'(u)}{\e^2}\right)^2 dx+\int(\phi \partial_t \rho+\rho \partial_t \phi) d\mu^\e_t\\
\nonumber&+\e \int \rho \nabla \phi \cdot \nabla  u \left(-\Delta u+\frac{W'(u)}{\e^2}\right) dx +\e \int \phi \nabla \rho \cdot \nabla u \left(-\Delta u+\frac{W'(u)}{\e^2}\right) dx\\
&=\int-\e \phi \rho \left(-\Delta u+\frac{W'(u)}{\e^2}\right)^2 dx+\int(\phi \partial_t \rho+\rho \partial_t \phi) d\mu^\e_t.\\
\nonumber&+2\e \int \rho \nabla \phi \cdot \nabla  u \left(-\Delta u+\frac{W'(u)}{\e^2}\right) dx\\
\nonumber&+\int\left(\e \nabla_i u \nabla_j u-\frac{\e|\nabla u|^2}{2}\delta_{ij}-\frac{W(u)}{\e}\delta_{ij}\right)\phi \nabla_i\nabla_j \rho dx\\
\nonumber&-\int\left(\e \nabla_i u \nabla_j u-\frac{\e|\nabla u|^2}{2}\delta_{ij}-\frac{W(u)}{\e}\delta_{ij}\right)\rho \nabla_i\nabla_j \phi dx.
\end{align}
Consider the term 
\begin{align*}
&\int\left(\e \nabla_i u \nabla_j u-\frac{\e|\nabla u|^2}{2}\delta_{ij}-\frac{W(u)}{\e}\delta_{ij}\right)\rho \nabla_i\nabla_j \phi dx\\
 &=-\int \rho \delta : \nabla^2 \phi d \mu^\e_t+\rho \nu \otimes\nu : \nabla^2 \phi \e|\nabla u|^2 dx\\
&=-\int \rho \left(\delta-\nu \otimes \nu \right)\nabla^2\phi d \mu^\e_t+\int \rho \nu \otimes \nu : \nabla^2 \phi \left(\e|\nabla u|^2-\frac{\e|\nabla u|^2}{2}-\frac{W(u)}{\e}\right) dx\\
&=\int \rho \left(\delta-\nu \otimes \nu \right) : \nabla^2 \phi d \mu^\e_t+\int \rho \nu \otimes \nu : \nabla^2 \phi d\xi^\e_t. 
\end{align*}
Therefore inserting this into the equation \eqref{eqn_one}
\begin{align*}
\frac{d}{dt}\int \phi \rho d\mu^\e_t &=\int-\e \phi \rho \left(-\Delta u+\frac{W'(u)}{\e^2}\right)^2 dx+\int(\phi \partial_t \rho+\rho \partial_t \phi) d\mu^\e_t\\
&+2\e \int \rho \nabla \phi \cdot \nabla  u \left(-\Delta u+\frac{W'(u)}{\e^2}\right) dx\\
&+\int\left(\e \nabla_i u \nabla_j u-\frac{\e|\nabla u|^2}{2}\delta_{ij}-\frac{W(u)}{\e}\delta_{ij}\right)\phi \nabla_i\nabla_j \rho dx\\
&-\int \rho \left(\delta-\nu \otimes \nu \right) : \nabla^2 \phi d \mu^\e_t-\int \rho \nu \otimes \nu : \nabla^2 \phi d\xi^\e_t\\
&=\int-\e \rho \phi \left(-\Delta u+\frac{W'(u)}{\e^2}-\frac{\nabla u \cdot \nabla \phi}{\phi}\right)^2 dx\\
&+\e \int \rho \frac{(\nabla u\cdot \nabla \phi)^2}{\phi}dx-\int \rho \nu \otimes \nu : \nabla^2 \phi d \xi^\e_t\\
&+\int \rho \left(\delta-\nu \otimes \nu  \right) : \nabla^2 \phi d\mu^\e_t+\int \left(\phi \partial_t \rho+\rho \partial_t \phi \right) d \mu^\e_t\\
&+\int\left(\e \nabla_i u \nabla_j u-\frac{\e|\nabla u|^2}{2}\delta_{ij}-\frac{W(u)}{\e}\delta_{ij}\right)\phi \nabla_i\nabla_j \rho dx.
\end{align*}
Since 
\begin{align*}
\e \int \rho  \frac{(\nabla u \cdot \nabla \phi)^2}{\phi}dx &=\int \rho \frac{(\nu \cdot \nabla \phi)^2}{\phi}(d\mu^\e_t+d\xi^\e_t) 
\end{align*}
we get
\begin{align*}
\frac{d}{dt}\int \phi \rho d\mu^\e_t &=\int-\e \rho \phi \left(-\Delta u+\frac{W'(u)}{\e^2}-\frac{\nabla u \cdot \nabla \phi}{\phi}\right)^2 dx\\
&+\int \rho \frac{(\nu \cdot \nabla \phi)^2}{\phi}(d\mu^\e_t+d\xi^\e_t)-\int \rho \nu \otimes \nu : \nabla^2 \phi d \xi^\e_t\\
&+\int \rho \left(\delta-\nu \otimes \nu  \right) : \nabla^2 \phi d\mu^\e_t+\int \left(\phi \partial_t \rho+\rho \partial_t \phi \right) d \mu^\e_t\\
&+\int\left(\e \nabla_i u \nabla_j u-\frac{\e|\nabla u|^2}{2}\delta_{ij}-\frac{W(u)}{\e}\delta_{ij}\right)\phi \nabla_i\nabla_j \rho dx\\
&=\int-\e \rho \phi \left(-\Delta u+\frac{W'(u)}{\e^2}-\frac{\nabla u \cdot \nabla \phi}{\phi}\right)^2 dx\\
&+\int \rho \left(\frac{\nu \cdot \nabla  \phi}{\phi}-(\delta-\nu \otimes \nu) : \nabla^2 \phi+\frac{\partial \phi}{\partial t} \right) d \mu^\e_t\\
&+\int \rho \left(-\nu\otimes \nu : \nabla^2 \phi+\frac{(\nu \cdot \nabla \phi)^2}{\phi}\right) d\xi^\e_t\\
&+\int \phi \partial_t \rho d\mu^\e_t+\int\left(\e \nabla_i u \nabla_j u-\frac{\e|\nabla u|^2}{2}\delta_{ij}-\frac{W(u)}{\e}\delta_{ij}\right)\phi \nabla_i\nabla_j \rho dx.
\end{align*}
If we choose $\phi(x,t)=\Phi_{y,s}(x,t)=\frac{1}{(4\pi(s-t))^{n/2}}e^{-|x-y|^2/4(s-t)}$, we then get 
\begin{align*}
-\nu\otimes\nu : \nabla^2\Phi_{y,s}(x,t) +\frac{(\nu \cdot \nabla \Phi_{y,s}(x,t))^2}{\Phi_{y,s}(x,t)}=\frac{\Phi_{y,s}(x,t)}{2(s-t)}
\end{align*}
 and 
 \begin{align*}
\partial_t \Phi_{y,s}(x,t)+(\delta-\nu\otimes\nu):\nabla^2\Phi_{y,s}(x,t)+\frac{(\nabla \Phi_{y,s}(x,t)\cdot \nu)^2}{\Phi_{y,s}(x,t)}=0. 
\end{align*}
This then gives 
\begin{align*}
\frac{d}{dt}\int \Phi_{y,s}\rho d\mu^\e_t &=
\int-\e \rho \Phi_{y,s}\left(-\Delta u+\frac{W'(u)}{\e^2}-\frac{\nabla u \cdot \nabla \Phi_{y,s}}{\Phi_{y,s}}\right)^2 dx+\int  \frac{\rho \Phi_{y,s}}{2(s-t)}d\xi^\e_t\\
&+\int \Phi_{y,s}\partial_t \rho d\mu^\e_t+\int\left(\e \nabla_i u \nabla_j u-\frac{\e|\nabla u|^2}{2}\delta_{ij}-\frac{W(u)}{\e}\delta_{ij}\right)\Phi_{y,s}\nabla_i\nabla_j \rho dx.
\end{align*}
  and this is the weighted monotonicity formula for the parabolic Allen--Cahn equation.
  \end{proof}
  
\begin{proposition}[{$L^2-L^\infty$ inequality}] \label{eqn_L2Linfty}
Let $\{u^\e_t\}$ be a solution to the Allen--Cahn Flow \eqref{eqn_ACFe} then in the parabolic ball $P_r$ we have 
\begin{align*}
\int_{B_{r/2}}x_{n+1}^2 \Phi_{0,0}d \mu^\e_t \leq \frac{c(n)}{r^{n+2}}\int_{-r^2}^{0}\int_{B_r}x_{n+1}^2 d \mu^\e_t .
\end{align*}
where $\Phi_{0,0}=\frac{1}{(-4\pi t)^\frac{n}{2}}e^{-\frac{|x|^2}{4t}}, t<0$ is the backward heat kernel of dimension $n$.
\end{proposition}

\begin{remark}
This estimate is used later in the blow-up argument to relate the spatial $L^2$ excess and space-time $L^2$ excess.
\end{remark}

\begin{proof}
This is a consequence of the following mean value inequality by substituting $f=x_{n+1}$.
\end{proof}

\begin{theorem}[Mean Value Inequality]
Let $u^\e_t$ be a solution to the Allen--Cahn Flow \eqref{eqn_ACFe} and let $ f : \mathbb R^{n+1} \times \mathbb R  \rightarrow \mathbb R $ be a nonnegative function that satisfies the following inequality 
\begin{align*}
\partial_t f-\Delta f+\frac{\e \nabla_i u^\e \nabla_j u^\e}{\nicefrac {\e |\nabla u^\e|^2}{2}+ \nicefrac{W(u^\e)}{\e}} \nabla_i \nabla_j f  \leq 0 
\end{align*}
for all $ t \in ( t_1, t_0)$. Then we have 
\begin{align*}
\int_{B_{\rho/2}(x_0)}f^2 \Phi_{x_0,t_0}d \mu^\e_t \leq \frac{c(n)}{\rho^{n+2}}\int_{t_0-\rho^2}^{t_0}\int_{B_\rho(x_0)}f^2 d \mu^\e_t.
\end{align*}
\end{theorem}
\begin{proof}
Let us consider firstly 
\begin{align*}
&\partial_t f^2-\Delta f^2+\frac{\e \nabla_i u \nabla_j u}{\nicefrac {\e |\nabla u|^2}{2}+ \nicefrac{W(u)}{\e}} \nabla_i \nabla_j f^2\\
 &\leq 2 f \partial_t f-2 f \Delta f+\frac{\e \nabla_i u \nabla_j u}{\nicefrac {\e |\nabla u|^2}{2}+ \nicefrac{W(u)}{\e}}( 2 f \nabla_i \nabla_j f+2 \nabla_if\nabla_jf)-2 |\nabla f|^2\\
&\leq -2 |\nabla f|^2+2  \frac{\e \nabla_i u \nabla_j u}{\nicefrac {\e |\nabla u|^2}{2}+ \nicefrac{W(u)}{\e}}\nabla_i f \nabla_j f\\
&=-2 |\nabla f|^2+2  \frac{\e |\nabla  u|^2}{\nicefrac {\e |\nabla u|^2}{2}+ \nicefrac{W(u)}{\e}}|\nabla_\nu f|^2\\
&=-2 |\nabla ^Tf|^2+\frac{2 \xi^\e_t | \nabla_\nu f|^2}{\nicefrac {\e |\nabla u|^2}{2}+ \nicefrac{W(u)}{\e}}
\end{align*}
where $\xi^\e_t=\nicefrac{\e|\nabla u|^2}{2}-\nicefrac{W(u)}{\e}\leq 0$ by the discrepancy inequality \eqref{eqn_discrepancy} and $ \nu=\frac{\nabla u}{|\nabla u|}$ and $ X^T=S: X, S=Id-\nu\otimes \nu $. 

Next we compute 
\begin{align*}
&\partial_t (\phi^2f^2) -\Delta (\phi^2f^2)+\frac{\e \nabla_i u \nabla_j u}{\nicefrac {\e |\nabla u|^2}{2}+ \nicefrac{W(u)}{\e}} \nabla_i \nabla_j (\phi^2f^2)\\
&=\phi^2 \left(\partial_t f^2-\Delta f^2+\frac{\e \nabla_i u \nabla_j u}{\nicefrac {\e |\nabla u|^2}{2}+ \nicefrac{W(u)}{\e}}\nabla_i \nabla_j f^2\right)\\
&+f^2 \left(\partial_t \phi^2-\Delta \phi^2+\frac{\e \nabla_i u \nabla_j u}{\nicefrac {\e |\nabla u|^2}{2}+ \nicefrac{W(u)}{\e}}\nabla_i \nabla_j \phi^2\right)\\
&-2 \nabla_i \phi^2\nabla_i f^2+2\frac{\e \nabla_i u \nabla_j u}{\nicefrac {\e |\nabla u|^2}{2}+ \nicefrac{W(u)}{\e}}\nabla_i \phi^2 \nabla_j f^2\\
&\leq \phi^2\left(-2 |\nabla ^Tf|^2+\frac{2 \xi^\e_t | \nabla_\nu f|^2}{\nicefrac {\e |\nabla u|^2}{2}+ \nicefrac{W(u)}{\e}}\right)\\
&+f^2 \left(\partial_t \phi^2-\Delta \phi^2+\frac{\e \nabla_i u \nabla_j u}{\nicefrac {\e |\nabla u|^2}{2}+ \nicefrac{W(u)}{\e}}\nabla_i \nabla_j \phi^2\right)\\
&-2 \nabla_i \phi^2\nabla_i f^2+2\frac{\e \nabla_i u \nabla_j u}{\nicefrac {\e |\nabla u|^2}{2}+ \nicefrac{W(u)}{\e}}\nabla_i \phi^2 \nabla_j f^2.
\end{align*}
By Young's inequality, we also have 
\begin{align*}
&-2 \nabla_i \phi^2\nabla_i f^2 +2\frac{\e \nabla_i u \nabla_j u}{\nicefrac {\e |\nabla u|^2}{2}+ \nicefrac{W(u)}{\e}}\nabla_i \phi^2 \nabla_j f^2\\
&=-8 \phi \nabla_i \phi f \nabla_j f+8  \phi \nabla_\nu \phi f \nabla_\nu f \frac{\e |\nabla u|^2}{\nicefrac {\e |\nabla u|^2}{2}+ \nicefrac{W(u)}{\e}}\\
&=-8 \phi \nabla_i^T\phi f \nabla_j^Tf+\frac{8 \phi \nabla_\nu \phi f \nabla_\nu f \xi^\e_t}{\nicefrac {\e |\nabla u|^2}{2}+ \nicefrac{W(u)}{\e}}\\
&\leq \phi^2 |\nabla ^Tf |^2+16 f^2 |\nabla ^T\phi|^2+\frac{|\xi^\e_t|}{\nicefrac {\e |\nabla u|^2}{2}+ \nicefrac{W(u)}{\e}}( \phi^2 |\nabla_\nu f|^2+16 f^2 |\nabla_\nu \phi|^2).
\end{align*}
So that 
\begin{align*}
&\partial_t (\phi^2f^2) -\Delta (\phi^2f^2)+\frac{\e \nabla_i u \nabla_j u}{\nicefrac {\e |\nabla u|^2}{2}+ \nicefrac{W(u)}{\e}} \nabla_i \nabla_j (\phi^2f^2)\\
\nonumber&\leq \phi^2\left(-2 |\nabla ^Tf|^2+\frac{2 \xi^\e_t | \nabla_\nu f|^2}{\nicefrac {\e |\nabla u|^2}{2}+ \nicefrac{W(u)}{\e}}\right)\\
\nonumber&+f^2 \left(\partial_t \phi^2-\Delta \phi^2+\frac{\e \nabla_i u \nabla_j u}{\nicefrac {\e |\nabla u|^2}{2}+ \nicefrac{W(u)}{\e}}\nabla_i \nabla_j \phi^2\right)\\
\nonumber&+ \phi^2 |\nabla ^Tf |^2+16 f^2 |\nabla ^T\phi|^2+\frac{|\xi^\e_t|}{\nicefrac {\e |\nabla u|^2}{2}+ \nicefrac{W(u)}{\e}}( \phi^2 |\nabla_\nu f|^2+16 f^2 |\nabla_\nu \phi|^2)\\
\nonumber&\leq f^2 \left(\partial_t \phi^2-\Delta \phi^2+\frac{\e \nabla_i u \nabla_j u}{\nicefrac {\e |\nabla u|^2}{2}+ \nicefrac{W(u)}{\e}}\nabla_i \nabla_j \phi^2+16|\nabla^T\phi|^2+16|\nabla_\nu\phi|^2\right)\\
\nonumber&\leq c_\phi f^2
\end{align*}
where we have used the non-positivity of the discrepancy \eqref{eqn_discrepancy} and $ c_\phi=C(n)( \partial_t \phi^2+|\nabla ^2\phi^2|+|\nabla \phi|^2)$. Multiplying both sides by $\left(\frac{\e|\nabla u|^2}{2}+\frac{W(u)}{\e}\right)$, we get
\begin{align}\label{PhifSquare}
&\partial_t (\phi^2f^2)\left(\frac{\e|\nabla u|^2}{2}+\frac{W(u)}{\e}\right) +\left(\e \nabla_i u \nabla_j u- \frac{\e|\nabla u|^2}{2}\delta_{ij}-\frac{W(u)}{\e}\delta_{ij}\right) \nabla_i \nabla_j (\phi^2f^2)\\
\nonumber&\leq c_\phi f^2\left(\frac{\e|\nabla u|^2}{2}+\frac{W(u)}{\e}\right).
\end{align}
Consider the cylinder  $ C_\rho=B_\rho(0) \times ( t_0-\rho^2 , t_0)$ with $\phi$ chosen so that 
\begin{align*}
(\rho |\nabla \phi|+\rho^2 |\nabla ^2 \phi|+|\partial_t \phi| ) \leq C' 
\end{align*}
and $ \spt \phi \subset C_\rho$ and $ \chi_{C_{\rho/2}}\leq \phi \leq \chi_{C_\rho}$. This shows $ c_\phi \leq \frac{c}{\rho^{n+2}}$ and that 
\begin{equation*}
\Phi_{(x_0,t_0)}=\frac{1}{( 4 \pi ( t_0-t))^{n/2}}e^{-|x-y|^2/(4 (t_0-t))}\implies | \Phi_{(x_0,t_0)}| \leq \frac{c}{\rho^{n+2}}\quad \text{on $(t_0-\rho^2, t_0) $}. 
\end{equation*} 
Therefore we can substitute \eqref{PhifSquare} into the Allen--Cahn weighted monotonicity formula \eqref{eqn_WMF} to get 
\begin{align*}
\frac{d}{dt}\int f^2 \phi^2 \Phi_{(x_0,t_0)}d \mu^\e_t \leq \frac{c}{\rho^{n+2}}\int_{B_\rho(x_0) \backslash B_{\rho/2}(x_0)}f^2 d \mu^\e_t 
\end{align*}
where we integrate in time over $ (t_0, t_0-\rho^2/2)$ and note that $ \phi(x_0,t_0)=1$ to get 
\begin{align*}
\int_{B_{\rho/2}(x_0)}f^2 \Phi_{(x_0,t_0)}d \mu^\e_t \leq \frac{c}{\rho^{n+2}}\int_{t_0-\rho^2}^{t_0-\rho^2/4}\int_{B_\rho(x_0) \backslash B_{\rho/2}(x_0)}f^2 d \mu^\e_t. 
\end{align*}
\end{proof}
\subsection{Excess Convergence}
In the following we show if the density ratio is converging to $\alpha$ then the excess converges to zero. In particular, small height excess is a consequence of small tilt excess, which again is a consequence of the area ratio being close to $\alpha$.
\begin{lemma} \label{lem_excess}
Let $ \{u^\e\}$ be a sequence of solutions of \eqref{eqn_ACFe} and suppose that we have the discrepancy inequality \eqref{eqn_discrepancy} in $B_4$. Then the excess of $u^\e$ with respect to $\mathbb R^n$ satisfies 
\begin{align}\label{ExcessConvergence}
\lim_{\e \rightarrow 0} E(C_2\times[t_1,t_2])=0.
\end{align}
\end{lemma} 
\begin{proof} Let us consider the $\e$-Brakke formula \eqref{eqn_BrakkeAllenCahn1}, we then have
\begin{align*}
\frac{d}{dt} \int \varphi d\mu^\e_t &=\int - \e \varphi \left(- \Delta u - \frac{W'(u)}{\e^2}\right )^2 d x \\
&+\int \left(\e \nabla u\otimes \nabla u  - \frac{\e|\nabla u|^2}{2}\delta  - \frac{W(u)}{\e}\delta \right ):\nabla ^2\varphi dx .   
\end{align*}
We choose the test function $\varphi(x)=\eta(x) \frac{1}{2}x_{n+1}^2$. This gives 
\begin{align}\label{eqn_brakketwo}
\nonumber&\left(\e \nabla_i u \nabla_j u  - \frac{\e|\nabla u|^2}{2}\delta_{ij}  - \frac{W(u)}{\e} \delta_{ij} \right )\nabla_i\nabla_j\varphi\\
\nonumber&=\left(\e u_{n+1}^2  - \frac{\e|\nabla u|^2}{2}  - \frac{W(u)}{\e} \right )\eta\\
\nonumber&+ x_{n+1} \left(\e u_{n+1}\nabla_j \eta \nabla_j u - \left(\frac{\e|\nabla u|^2}{2} + \frac{W(u)}{\e} \right ) \nabla_{n+1}\eta\right) \\
\nonumber&+ \e \nabla_i u \nabla_j u \nabla_i\nabla_j \eta \frac{x_{n+1}^2}{2} - \frac{x_{n+1}^2}{2} \left(\frac{\e|\nabla u|^2}{2} + \frac{W(u)}{\e} \right ) \Delta \eta\\
&=\left(\e u_{n+1}^2  - \frac{\e|\nabla u|^2}{2}  - \frac{W(u)}{\e} \right )\eta\\
\nonumber&+ x_{n+1}\left(\nu_{n+1} \sum_{j=1}^n \nabla_j \eta \nu_j \e|\nabla u|^2 - \left(\frac{\e|\nabla u|^2}{2} + \frac{W(u)}{\e} \right )\nabla_{n+1}\eta \right )\\
\nonumber&+ \frac{x_{n+1}^2}{2} \nu_i \nu_j \nabla ^2_{ij} \eta \e|\nabla u|^2 - \frac{x_{n+1}^2}{2} \left(\frac{\e|\nabla u|^2}{2} + \frac{W(u)}{\e} \right ) \Delta \eta. 
\end{align}
We note that the terms 
\begin{align*}
  \nu_{n+1} \nabla_j \eta \nu_j \e|\nabla u|^2, \left(\frac{\e|\nabla u|^2}{2} + \frac{W(u)}{\e} \right )\nabla_{n+1}\eta
\end{align*}
are bounded and hence converge to measures supported on $\mathbb R^n\times \mathbb R $. Hence we have 
\begin{align*}
\lim_{\e\rightarrow 0} \int_{t_1}^{t_2} \int_{C_2} x_{n+1}\left(\nu_{n+1} \sum_{j=1}^n\nabla_j \eta \nu_j \e|\nabla u|^2 - \left(\frac{\e|\nabla u|^2}{2} + \frac{W(u)}{\e} \right )\nabla_{n+1}\eta \right )dxdt=0,
\end{align*}
and
\begin{align*}
\lim_{\e\rightarrow 0}\int_{t_1}^{t_2} \int_{C_2} \frac{x_{n+1}^2}{2} \nu_i \nu_j \nabla ^2_{ij} \eta \e|\nabla u|^2 &- \frac{x_{n+1}^2}{2} \left(\frac{\e|\nabla u|^2}{2} + \frac{W(u)}{\e} \right ) \Delta \eta=0.
\end{align*}
Therefore in the $\e$-Brakke formula \eqref{eqn_BrakkeAllenCahn1} we have the following equation
\begin{align*}
&\lim_{\e \rightarrow 0}  \int_{C_2} \frac{\eta^{2} x_{n+1}^2}{2} d \mu^\e_s \bigg|_{t_1}^{t_2} \\ 
&=-\lim_{\e\rightarrow 0} \int_{t_1}^{t_2} \int_{C_2} \e \phi \left(- \Delta u + \frac{W'(u)}{\e^2}\right )^2 dxdt \\
&+ \lim_{\e\rightarrow 0} \int_{t_1}^{t_2}\int_{C_2}\left(\e u_{n+1}^2 - \left(\frac{\e|\nabla u|^2}{2} + \frac{W(u)}{\e} \right ) \right )\eta dxdt \\
&=-\lim_{\e\rightarrow 0} \int_{t_1}^{t_2}\int_{C_2} \e \phi \left(- \Delta u + \frac{W'(u)}{\e^2}\right )^2 dxdt \\
&- \lim_{\e\rightarrow 0} \int_{t_1}^{t_2} \int_{C_2}\frac{1}{2}\left(1-\nu_{n+1}^2  \right )\e|\nabla u^2|\eta+\lim_{\e\rightarrow 0} \int_{t_1}^{t_2} \int_{C_2}\left(\frac{1}{2}\nu_{n+1}^2 \e|\nabla u^2|-\frac{W(u)}{\e}\right)\eta dxdt.
\end{align*}
Since 
\begin{align*}
\lim_{\e \rightarrow 0}  \int_{C_2} \frac{\eta^{2} x_{n+1}^2}{2} d \mu^\e_s \bigg|_{t_1}^{t_2}=0,
\end{align*}
and all three terms on the right hand side are non-positive, we get
\begin{align*}
0&=\lim_{\e\rightarrow 0} \int_{t_1}^{t_2}\int_{C_2} \e \phi \left(- \Delta u + \frac{W'(u)}{\e^2}\right )^2 dxdt+ \lim_{\e\rightarrow 0} \int_{t_1}^{t_2} \int_{C_2}\frac{1}{2}\left(1-\nu_{n+1}^2  \right )\e|\nabla u^2|\eta dxdt\\
&-\lim_{\e\rightarrow 0} \int_{t_1}^{t_2} \int_{C_2}\left(\frac{1}{2}\nu_{n+1}^2 \e|\nabla u^2|-\frac{W(u)}{\e}\right)\eta dxdt,
\end{align*}
namely
\begin{align*}
\lim_{\e\rightarrow0} \mathcal W(C_2\times\left(t_1,t_2\right))+\lim_{\e\rightarrow0} E(C_2\times\left(t_1,t_2\right))-\lim_{\e\rightarrow0} \xi(C_2\times\left(t_1,t_2\right))=0.
\end{align*}
Since the three terms are non-negative, this proves the theorem.
\end{proof} 
\begin{remark}
As we can see from the proof, it gives an alternative argument of the $L^1$ convergence of discrepancy measure to 0 under the non-positive discrepancy assumption \eqref{eqn_discrepancy}.
\end{remark}
\section{Energy Estimates} \label{sec_EnergyEstimates}
\subsection{Caccioppoli Inequality}

We will prove a Caccioppoli type inequality (which extends 4.7 of \cite{wang2014new} to general functions). It can interpreted as the Allen--Cahn version of Allard's Caccioppoli type inequality for minimal surfaces (8.13 of \cite{allard1972}).

The following inequality does not require $u$ to satisfy the Allen--Cahn equations and it applies to time slices of the parabolic Allen--Cahn equation. When $\Delta u-\frac{W'(u)}{\e^2}=0$, this inequality was obtained in \cite[4.7]{wang2014new}.

\begin{theorem}\label{Caccioppoli}
Assuming the non-positive discrepancy condition \eqref{eqn_discrepancy}, there exists universal constants $C_1,C_2,C_3$ such that
\begin{align*}
&\int_\Omega\phi^2(\hat x)\psi^2(x_{n+1})(1-\nu_{n+1}^2)\e|\nabla u|^2dx+\int_\Omega\phi^2(\hat x)\psi^2(x_{n+1})\left|\frac{\e|\nabla u|^2}{2}-\frac{W(u)}{\e}\right| dx\\
\leq &C_1\int_\Omega \phi^2(\hat x)\psi^2(x_{n+1})x_{n+1}^2\e|\nabla u|^2dx\\
&+C_2\left(\int_\Omega \phi^2(\hat x)\psi(x_{n+1})^2x_{n+1}^2\e|\nabla u|^2dx\cdot\int_\Omega\phi^2(\hat x)\psi^2(x_{n+1}) \left(\Delta u-\frac{W'(u)}{\e^2}\right)^2dx\right)^\frac{1}{2}\\
&+C_3\int_\Omega 2 (1-\nu_{n+1}^2)\phi^2(\hat x) |\psi'(x_{n+1})|x_{n+1}\e |\nabla u|^2dx,
\end{align*}
where $\phi\in C_c^\infty(\mathbb R^n)$ depends only on the first $n$ variables, $\psi\in C_c^\infty(\mathbb R)$ with $0\leq \psi\leq 1, \spt(\psi)\subset\subset(-1,1), \psi(\xi)\equiv1 \text{ for $\xi$ in $(-\frac{1}{2},\frac{1}{2})$}, |\psi'|\leq 3$. 
\end{theorem}
\begin{remark}
Notice that this is stronger than the usual Caccioppoli type inequality in that we bound both the tilt excess and the discrepancy.
\end{remark}

\begin{proof}
We choose the test vector field 
\begin{align*}
g=\left(0,\dots,0,\phi^2(\hat x)\psi^2(x_{n+1})\cdot x_{n+1}\right)=\phi^2(\hat x)\psi^2(x_{n+1})\cdot x_{n+1}e_{n+1}
\end{align*}
and substitute this into \eqref{eqn_ACDivergence} to get
\begin{align*}
&\int_\Omega\frac{\partial u}{\partial x_{n+1}}\phi^2\psi^2 \cdot x_{n+1}\left(\Delta u-\frac{W'(u)}{\e^2}\right)dx\\
&=\int_\Omega\frac{1}{\e}\left(\frac{\e|\nabla u|^2}{2}+ \frac{W(u)}{\e}\right) \left(\phi^2\psi^2+\frac{\partial(\phi^2\psi^2)}{\partial x_{n+1}}\cdot x_{n+1}\right)dx-\int_\Omega\left|\frac{\partial u}{\partial x_{n+1}}\right|^2\phi^2\psi^2dx\\
&-\sum_{i=1}^{n+1}\int_\Omega\frac{\partial u}{\partial x_i}\frac{\partial u}{\partial x_{n+1}}\frac{\partial (\phi^2\psi^2)}{\partial x_i}x_{n+1}dx\\
&=\int_\Omega\frac{1}{\e}\left(\frac{\e|\nabla u|^2}{2}+ \frac{W(u)}{\e}\right) \left(\phi^2\psi^2+2\phi^2\psi\psi'x_{n+1}\right)dx-\int_\Omega\left|\frac{\partial u}{\partial x_{n+1}}\right|^2\phi^2\psi^2dx\\
&-2\sum_{i=1}^{n}\int_\Omega\frac{\partial u}{\partial x_i}\frac{\partial u}{\partial x_{n+1}}\phi\psi^2\frac{\partial \phi}{\partial x_i}x_{n+1}dx-2\int_\Omega \left|\frac{\partial u}{\partial x_{n+1}}\right|^2\phi^2\psi\psi'x_{n+1}dx\\
&=\int_\Omega \frac{1}{\e}\left(\frac{\e|\nabla u|^2}{2}+ \frac{W(u)}{\e}-\e\nu^2_{n+1}|\nabla u|^2\right)\phi^2\psi^2dx-\int_\Omega2\phi\psi^2|\nabla u|^2\left(\sum_{i=1}^{n}\nu_i\nu_{n+1}\frac{\partial\phi}{\partial x_i}\right)x_{n+1}dx\\
&+\int_\Omega\frac{1}{\e}\left(\frac{\e|\nabla u|^2}{2}+ \frac{W(u)}{\e}-\e\nu^2_{n+1}|\nabla u|^2\right) 2 \phi^2\psi'x_{n+1}dx.
\end{align*}
We collect terms and apply Young's inequality to estimate the last equation from below to get 
\begin{align*}
&\int_\Omega\frac{\partial u}{\partial x_{n+1}}\phi^2\psi^2 \cdot x_{n+1}\left(\Delta u-\frac{W'(u)}{\e^2}\right)dx\\
&\geq \frac{1}{2}\int_\Omega(1-\nu_{n+1}^2)\phi^2\psi^2|\nabla u|^2dx+\frac{1}{2}\int_\Omega\left(\frac{W(u)}{\e^2}-\nu^2_{n+1}|\nabla u|^2\right)\phi^2\psi^2dx\\ &-\frac{1}{4}\int_\Omega\sum_{i=1}^{n}\phi^2\psi^2\nu_i^2|\nabla u|^2dx-64\int_\Omega|\nabla\phi|^2\psi^2 x_{n+1}^2|\nabla u|^2dx\\
&+\int_\Omega(1-\nu_{n+1}^2)2\phi^2\psi'x_{n+1}|\nabla u|^2dx\\
&\geq \frac{1}{2}\int_\Omega(1-\nu_{n+1}^2)\phi^2\psi^2|\nabla u|^2dx+\frac{1}{2}\int_\Omega\left(\frac{W(u)}{\e^2}-|\nabla u|^2\right)\phi^2\psi^2dx\\
&-\frac{1}{4}\int_\Omega\sum_{i=1}^{n}\phi^2\psi^2\nu_i^2|\nabla u|^2dx -64\int_\Omega|\nabla\phi|^2\psi^2 x_{n+1}^2|\nabla u|^2dx\\
&+\int_\Omega(1-\nu_{n+1}^2)2\phi^2\psi'x_{n+1}|\nabla u|^2dx\\
&=\frac{3}{4}\int_\Omega(1-\nu_{n+1}^2)\phi^2\psi^2|\nabla u|^2dx-64\int_\Omega|\nabla\phi|^2\psi^2x_{n+1}^2 |\nabla u|^2dx\\
&+\int_\Omega(1-\nu_{n+1}^2)2\phi^2\psi'x_{n+1}|\nabla u|^2dx
\end{align*}
where we again used Young's inequality and the non-positivity of discrepancy \eqref{eqn_discrepancy} in the last line. Next, by H\"older's inequality, the left hand side can be estimated from above by
\begin{align*}
&\int_\Omega \e\frac{\partial u}{\partial x_{n+1}}\phi^2\psi^2 x_{n+1}\left(\Delta u-\frac{W'(u)}{\e^2}\right)dx\\
&\leq \left(\int_\Omega\e \left(\Delta u-\frac{W'(u)}{\e^2}\right)^2\phi^2\psi^2 dx \right)^\frac{1}{2}\left(\int_\Omega\phi^2\psi^2 x_{n+1}^2 \e|\nabla u|^2dx\right)^\frac{1}{2}.
\end{align*}
This in turn gives us 
\begin{align*}
&\int_\Omega\phi^2\psi^2 (1-\nu_{n+1}^2) \e|\nabla u|^2dx+\int_\Omega\phi^2\psi^2\left|\frac{\e|\nabla u|^2}{2}-\frac{W(u)}{\e}\right| dx\\
&\leq \frac{256}{3}\int_\Omega|\nabla\phi|^2\psi^2x_{n+1}^2 \e|\nabla u|^2dx+\frac{4}{3}\left(\int_\Omega\e \left(\Delta u-\frac{W'(u)}{\e^2}\right)^2\phi^2\psi^2dx\cdot\int_\Omega \phi^2\psi^2 x_{n+1}^2 \e|\nabla u|^2dx\right)^\frac{1}{2}\\
&-\frac{4}{3}\int_\Omega2(1-\nu_{n+1}^2)\phi^2\psi'x_{n+1}\e|\nabla u|^2dx\\
\leq &C_1\int_\Omega \phi^2(\hat x)\psi^2 x_{n+1}^2\e|\nabla u|^2dx\\
&+C_2\left(\int_\Omega \phi^2(\hat x)\psi^2x_{n+1}^2\e|\nabla u|^2dx\cdot\int_\Omega\phi^2(\hat x)\psi^2  \left(\Delta u-\frac{W'(u)}{\e^2}\right)^2dx\right)^\frac{1}{2}\\
&+C_3\int_\Omega 2 (1-\nu_{n+1}^2)\phi^2(\hat x) |\psi' |x_{n+1}\e |\nabla u|^2dx,
\end{align*}
and this completes the proof.
\end{proof}
The following estimate shows that away from the nodal set, we have exponential decay as $\e\rightarrow 0$.

\begin{lemma}[{cf. \cite[Proposition 4.4]{wang2014new}}]\label{ExponentialDecay}
For any $h>0$, there exists $C(h)>0$ such that if $u^\e:\mathbb R^{n+1}\times\mathbb R\rightarrow\mathbb R$ is a solution to \eqref{eqn_ACFe} which blows down to the static solution of mean curvature flow supported on the flat plane $\{x_{n+1}=0\}$ as $\e\rightarrow0$, then
\begin{align*}
\sup_{|x_{n+1}|\geq h}(|\nabla u^\e|, 1-|u^\e|^2)\leq e^{-\frac{C(h)}{\e^2}}.
\end{align*}
\end{lemma}
\begin{proof}
We will use some estimates from \cite{trumper2008relaxation}. Let us define
\begin{align*}
v^\e(x,t)=\chi_{\bar\tau}(x_{n+1})g^\e(x_{n+1})+(1-\chi_{\bar\tau})(x_{n+1})\frac{x_{n+1}}{|x_{n+1}|}
\end{align*}
where $g^\e$ is the static 1-d solution to the Allen--Cahn equation with flat level sets defined in section 2 and $x_{\bar\tau}:\mathbb R\rightarrow[-1,1]$ is a cutoff function supported in $[-\bar\tau,\bar\tau]$ such that $\chi_{\bar\tau}(x)\equiv1, x\in[-\frac{\bar\tau}{2},\frac{\bar\tau}{2}]$.

We have by definition (see (44) of \cite{trumper2008relaxation})
\begin{align*}
|(v^\e)^2-1|&\leq 2e^{-\frac{h}{\e}}, \text{for $|x_{n+1}|\geq h$},\\
\left|\frac{\partial}{\partial t}v^\e-\Delta v^\e+\frac{W'(v^\e)}{\e^2}\right|
&\begin{cases}
=0, |x_{n+1}|\leq |\frac{\bar\tau}{2}|,\\
\leq O(\frac{e^{-\frac{\gamma\bar\tau}{2\e}}}{\e^2}), |\frac{\bar\tau}{2}|\leq |x_{n+1}|\geq \bar\tau,\\
=0, |x_{n+1}|\geq \bar\tau,\\
\end{cases}
\end{align*}
where $\gamma>0$ is a universal constant that only depends on the potential function $W$.

Thus (39) of \cite{trumper2008relaxation} gives
\begin{align*}
&\sup_{P_r}|u^\e(x,t)-v^\e(x,t)|\\
&\leq C\sup_{(x,t)\in P_{2r}}\int_0^t\int_{\mathbb R^{n+1}}\mathcal H(x-y,t-s)\left|\left(\frac{\partial}{\partial t}v^\e-\Delta v^\e+\frac{W'(v^\e)}{\e^2}\right)(y,s)\right| dyds\\
&\leq C\frac{e^{-\frac{\gamma\bar\tau}{2\e}}}{\e^2}\iint\mathcal H(x-y,t-s) dyds\\
&=O\left(\frac{e^{-\frac{\gamma\bar\tau}{2\e}}}{\e^2}\right),
\end{align*}
where $\mathcal H(x,t)$ is the heat kernel on $\mathbb R^{n+1}\times\mathbb R$. So for $|x_{n+1}|\geq h$
\begin{align*}
|1-(u^\e)^2|&\leq |1-(v^\e)^2|+O\left(\frac{e^{-\frac{\gamma\bar\tau}{2\e}}}{\e^2}\right)\\
&=O(e^{-\frac{h}{\e}})+O\left(\frac{e^{-\frac{\gamma\bar\tau}{2\e}}}{\e^2}\right)\\
&=O\left(e^{-\frac{h}{\e^2}}\right).
\end{align*}

This gives the desired decay estimate for $|1-(u^\e)^2|$. Once we have the estimates for $|1-(u^\e)^2|$, the estimates for $|\nabla u^\e|$ follows by the non-positive discrepancy assumption \eqref{eqn_discrepancy}.
\end{proof}

For later analysis, we assume the following technical assumption
\begin{align}\label{TechnicalAssumption}
\frac{E^\e(\Omega)}{\e^2}\rightarrow\infty, (\text{as $\e\rightarrow0$})
\end{align}
\begin{remark}
This will be satisfied in our proof by contradiction that where we assume \eqref{LayerRepulsion} does not hold.
\end{remark}

We apply Theorem \ref{Caccioppoli} to time slices of solutions to the parabolic Allen--Cahn equation, we get
\begin{corollary}\label{TimeSliceCaccioppoli}
Let $u^\e:\mathbb R^{n+1}\times\mathbb R\rightarrow\mathbb R$ be a solution to the parabolic Allen--Cahn equation $\frac{\partial}{\partial t}u^\e=\Delta u^\e-\frac{W'(u^\e)}{\e^2}$ which satisfies the technical assumption \eqref{TechnicalAssumption}. Then there exists $\e_0>0$ such that, for any time slice $t$ and $\e\leq \e_0$, we have
\begin{align*}
E(B_{\frac{1}{2}})\leq C_1\mathbb H(B_1)+ C_2\sqrt{\mathcal W(B_1)\mathbb H(B_1)}.
\end{align*}
\end{corollary}

\begin{proof}
Choose a test function $\phi\in C_c^\infty(\mathbb R^n)$ that is supported in $\hat B_1$ and where $\phi\equiv1$, in $\hat B_\frac{1}{2}$.  By Lemma \ref{ExponentialDecay} and the fact $\psi'(x)=0, \text{for $x\in(-\frac{1}{2},\frac{1}{2})$}$, we get for $\e_0$ sufficiently small
\begin{align*}
&E(B_\frac{1}{2}\times\{t\})+|\xi|(B_\frac{1}{2}\times\{t\})\\
&\leq \int_{B_1}\phi^2(\hat x)\psi^2(x_{n+1})(1-\nu_{n+1}^2)\e|\nabla u^\e(x,t)|^2dx+\int_{B_1}\phi^2(\hat x)\psi^2(x_{n+1})\left(\frac{\e|\nabla u^\e(x,t)|^2}{2}+\frac{W(u^\e(x,t))}{\e}\right)dx\\
&\leq C_1\int_{B_1}\phi^2(\hat x)\psi^2(x_{n+1}) x_{n+1}^2\e|\nabla u^\e(x,t)|^2dx\\
&+C_2\left(\int_{B_1}\phi^2(\hat x)\psi^2(x_{n+1})\e \left(\Delta u^\e(x,t)-\frac{W'(u^\e(x,t))}{\e^2}\right)^2 dx\cdot\int_{B_1}\phi^2(\hat x)\psi^2(x_{n+1}) x_{n+1}^2\e|\nabla u^\e(x,t)|^2dx\right)^\frac{1}{2}\\
&+ \tilde Ce^{-\frac{\tilde C}{\e}}\\
&=C_1\int_{B_1}\phi^2(\hat x)\psi^2(x_{n+1}) x_{n+1}^2\e|\nabla u^\e(x,t)|^2dx\\
&+C_2\left(\int_{B_1}\phi^2(\hat x)\psi^2(x_{n+1})\e \left(\Delta u^\e(x,t)-\frac{W'(u^\e(x,t))}{\e^2}\right)^2 dx\cdot\int_{B_1}\phi^2(\hat x)\psi^2(x_{n+1}) x_{n+1}^2\e|\nabla u^\e(x,t)|^2dx\right)^\frac{1}{2}\\
&+ o(\e^2)\\
&=\tilde C_1\mathbb H(B_1\times\{t\})+\tilde C_2\sqrt{\mathbb H(B_1\times\{t\})\mathcal W(B_1\times\{t\})},
\end{align*}
where we used the technical assumption \eqref{TechnicalAssumption} in the last inequality.

\end{proof}

\subsection{Sobolev type inequality}

The following is a quantitative version of Proposition 4.6 in \cite{tonegawa2003integrality}. Geometrically, this measures the the difference of the Allen--Cahn energy from the flat solution in terms of the height excess, the $L^2$ norm of the diffuse mean curvature raised to the power $ \frac{n}{n-2}$ and the discrepancy.
\begin{proposition}\label{InteriorSobolev}
There exists $\e_0,\gamma_0,C_3>0$ such that if $u$ is a time slice of a solution to the parabolic Allen--Cahn equation \eqref{eqn_ACFe} with $\e<\e_0, u(0)=0$, then 
\begin{align*}
&\left|\int_{B_1}\left(\frac{\e|\nabla u|^2}{2}+ \frac{W(u)}{\e}\right) dx-\alpha\omega_{n}\right|\\
\leq &C_3\left(E(B_3)+|\xi|(B_3)+E(B_3)\mathcal W(B_3)+\mathcal W(B_3)^\frac{n}{n-2}\right)\\
\leq &C_3\left(E(B_3)+|\xi|(B_3)+\sqrt{E(B_3)\mathcal W(B_3)}+\mathcal W(B_3)^\frac{n}{n-2}\right),
\end{align*}
where $\xi=\frac{\e|\nabla u|^2}{2}-\frac{W(u)}{\e}$ is the discrepancy and $\omega_k$ is the area of unit $k$-ball.
\end{proposition}
\begin{proof}

For any $\gamma>0$, define the set 
\begin{align*}
A^{\e,\gamma}=\left\{x\in B^{n+1}_1(0)\subset\mathbb R^{n+1}:
\begin{cases}
r^2\mathcal W(B_r(x))\leq \gamma^2\mu^\e(B_r(x))\\
r^nE(B_r(x))+|\xi|(B_r(x))\leq \gamma^2\mu^\e(B_r(x))
\end{cases}\quad \text{for all $r\in(0,2)$}\right\}
\end{align*}
and let $\Pi(A)\subset B^n_1\subset\mathbb R^n$ be the projection of $A$ to the plane $\{x_{n+1}=0\}$.

Recalling $g^\e(x,t)=\tanh(\frac{x}{\e})$ is the 1-d static solution with flat level sets, we define a distance type function 
\begin{align}\label{DistanceTypeFunction}
z(x,t)=z(\hat x,x_{n+1},t)=(g^\e)^{-1}(u^\e(x,t))=\e(g)^{-1}(u^\e(x,t)) 
\end{align}
and we let $\bar g^\e(x)=g^\e(x_{n+1})=g(\frac{x_{n+1}}{\e})$. We compute
\begin{align}\label{Identitiesforz}
\e(1-\nu_{n+1}^2)|\nabla u^\e|^2&=\e((g^\e)'(z))^2(|\nabla z|^2-z^2_{x_{n+1}})\\\nonumber
&=\frac{1}{\e}(g'(\tfrac{z}{\e}))^2(|\nabla z|^2-z^2_{x_{n+1}}),\\\nonumber
\left(\frac{\e|\nabla u^\e|^2}{2}-\frac{W(u^\e)}{\e}\right)&=\frac{1}{2}\e((g^\e)'(z))^2\left(|\nabla z|^2-1\right)\\\nonumber
&=\frac{1}{2\e}(g'(\tfrac{z}{\e}))^2\left(|\nabla z|^2-1\right).
\end{align}
By the non-positivity of discrepancy \eqref{eqn_discrepancy}, we have $|\frac{\partial z}{\partial x_{n+1}}|\leq |\nabla z|\leq 1$.

Before we proceed, we gather some relations between the derivatives of $u$ and derivatives of the graphical function $h^s$ representing the level sets of $u$.
\begin{align}\label{FunctionAndLevelSet}
\frac{\partial u}{\partial x_{n+1}}&=\left(\frac{\partial h^s}{\partial s}\right)^{-1},\\\nonumber
\frac{\partial u}{\partial x_{i}}&=-\left(\frac{\partial h^s}{\partial s}\right)^{-1}\frac{\partial h^s}{\partial x_i} ,i=1,...,n,\\\nonumber
\frac{\partial u}{\partial t}&=-\frac{\partial h^s}{\partial t}\left(\frac{\partial h^s}{\partial s}\right)^{-1}.
\end{align}

\begin{claim}
For any $b>0$, there exists $\gamma_0$ sufficiently small such that if $\gamma<\gamma_0$ the following holds : for $s\in(-1+b,1-b)$, the level sets $A^{\e,\gamma}\cap\{u^\e=s\}\subset B_1=\{(\hat x, x_{n+1})\in\mathbb R^{n+1}|x_{n+1}=h^s(\hat x)\}$ are Lipschitz graphs of functions $h^s$ with Lipschitz constants $\mathrm{Lip}(h^s)\leq \frac{1}{2}$.
\end{claim}
\begin{proof}[Proof of Claim]
First we prove that $\e|\frac{\partial u^\e}{\partial x_{n+1}}|\geq c_0$ for all $\gamma\leq \gamma_0$ when $\gamma_0$ is sufficiently small. 

Suppose not, there exists a sequence $\gamma_i\rightarrow0$, a sequence of functions $u^{\e_i}$ such that $\e|\frac{\partial u^{\e_i}}{\partial x_{n+1}}|\leq \frac{1}{i}\rightarrow0$ and $|u(0,0)|\leq 1-b$. We parabolically rescale the sequence and denote $v^{\e_i}=u^{\e_i}(\e_ix,\e_i^2t)$ so that the new sequence satisfies the same equation \eqref{eqn_ACF1} and also satisfies
\begin{align*}
\left |\frac{\partial v^{\e_i}}{\partial x_{n+1}}\right |\leq \frac{1}{i}\rightarrow0.
\end{align*}
By standard parabolic regularity, the sequence $v^{\e_i}$ converges smoothly to a limit $v^\infty$ which is a static (because $\mathcal W=0$) solution to \eqref{eqn_ACF1} that only depends on the $x_{n+1}$ variable and furthermore it satisfies
\begin{align*}
\left |\frac{\partial v^{\infty}}{\partial x_{n+1}}\right |=0.
\end{align*}
This is a contradiction to the fact that the only 1-d solution must be $v^{\infty}=\tanh(x_{n+1}+\tanh^{-1}(u(0,0)))$. This derivative lower bound shows the level sets in the $x_{n+1}$ direction are Lipschitz graphs.

Next, also by the smooth convergence, the level sets of $v^{\e_i}$ converges smoothly to hyperplanes parallel to $\{x_{n+1}=0\}$. Since $u^{\e_i}$ are rescalings of $v^{\e_i}$ and the Lipschitz constants are scale invariant, we conclude for sufficiently large $i$, the Lipschitz constants satisfy the bounds $\mathrm{Lip}(h^{s,\e_i})\leq \frac{1}{2}$.
\end{proof}
From the claim above we can conclude the level sets of $u^\e$ are Lipschitz graphs and hence we can estimate the difference of the energy of $u$ in a ball with the area of a flat hyperplane. First we consider the level sets that are far from the transition region that is $\{|u|\geq 1-b\}$. We first fix $b$ sufficiently small in order to apply  Lemma \ref{LevelSetCloseTo1} which gives us 
\begin{align}\label{FarFromTransitionEnergySmall1}
\int_{B_1^{n+1}\cap\{|u|\geq 1-b\}}\left(\frac{\e|\nabla u|^2}{2}+ \frac{W(u)}{\e}\right) dx\leq \left(\int_{B_3}(1-\nu_{n+1}^2)\e|\nabla u|^2 dx+\int_{B_3}|\xi| dx\right).
\end{align}
Furthermore since $g'(\frac{z}{\e})=\sqrt{2W(u)}$ by \eqref{Identitiesforz}, by the co-area formula and Young's inequality we have
\begin{align}\label{FarFromTransitionEnergySmall2}
&\left|\int_{-1}^{-1+b}\int_{\Pi(A\cap\{u=s\})\cap\hat B_1}g'(\frac{z}{\e}) d\hat xds+\int_{1-b}^{1}\int_{\Pi(A\cap\{u=s\})\cap\hat B_1}g'(\frac{z}{\e}) d\hat xds\right|\\\nonumber
&=\left|\int_{B_1\cap\{-1<u<-1+b\}}\sqrt{2W(u)}|\nabla u| dx+\int_{B_1\cap\{1-b<u<1\}}\sqrt{2W(u)}|\nabla u| dx\right|\\\nonumber
&\leq \left|\int_{B_1\cap\{-1<u<-1+b\}}\left(\frac{\e|\nabla u|^2}{2}+ \frac{W(u)}{\e}\right) dx+\int_{B_1\cap\{1-b<u<1\}}\left(\frac{\e|\nabla u|^2}{2}+ \frac{W(u)}{\e}\right) dx\right|\\\nonumber
&\leq \left(\int_{B_3}(1-\nu_{n+1}^2)\e|\nabla u|^2 dx+\int_{B_3}|\xi| dx\right), \text{by \eqref{FarFromTransitionEnergySmall1}}.
\end{align}
Using \eqref{FarFromTransitionEnergySmall1} and \eqref{FarFromTransitionEnergySmall2}, we estimate the energy difference
\begin{align}\label{EnergyDifference1}
&\left|\int_{B_1^{n+1}}\left(\frac{\e|\nabla u|^2}{2}+ \frac{W(u)}{\e}\right) dx-\alpha\omega_{n}\right|\\\nonumber
&\leq \left|\int_{B_1^{n+1}\cap\{|u|\leq 1-b\}}\left(\frac{\e|\nabla u|^2}{2}+ \frac{W(u)}{\e}\right) dx-\alpha\omega_{n}\right|+E(B_1^{n+1})+|\xi|(B_1^{n+1})\\\nonumber
&\leq \left|\int_{B_1^{n+1}\cap\{|u|\leq 1-b\}}\e|\nabla u|^2 dx-\alpha\omega_{n}\right|+E(B_1^{n+1})+2|\xi|(B_1^{n+1})\\\nonumber
&\leq \left|\int_{B_1^{n+1}\cap\{|u|\leq 1-b\}}\nu_{n+1}\e|\nabla u|^2 dx-\alpha\omega_{n}\right|+\left|\int_{B_1^{n+1}}(1-\nu_{n+1})\e|\nabla u|^2 dx\right|+E(B_1^{n+1})+2|\xi|(B_1^{n+1})\\\nonumber
&\leq \left|\int_{B_1^{n+1}\cap\{|u|\leq 1-b\}}\nu_{n+1}\e|\nabla u|^2 dx-\alpha\omega_{n}\right|+\left|\int_{B_1^{n+1}}(1-\nu^2_{n+1})\e|\nabla u|^2 dx\right|+E(B_1^{n+1})+2|\xi|(B_1^{n+1})
\end{align}
where we used $(1-\nu_{n+1})\e |\nabla u|^2\leq (1-\nu_{n+1}^2)\e |\nabla u|^2$ in the second last line. Next, in order to use our Lipschitz graphical representation, we will consider the integral over our set $A^{\e,\gamma}=A$. Furthermore since the second term in the last line is the tilt excess, we have the bound 
\begin{align}
&\left|\int_{B_1^{n+1}}\left(\frac{\e|\nabla u|^2}{2}+ \frac{W(u)}{\e}\right) dx-\alpha\omega_{n}\right|\\\nonumber
&\leq \left|\int_{B_1^{n+1}\cap\{|u|\leq 1-b\}}\nu_{n+1}\e|\nabla u|^2 dx-\alpha\omega_{n}\right|+2E(B_1^{n+1})+2|\xi|(B_1^{n+1})\\\nonumber
&=\left|\int_{A\cap\{|u|\leq 1-b\}}\nu_{n+1}\e|\nabla u|^2 dx-\alpha\omega_{n}\right|+\left|\int_{B_1^{n+1}\setminus A}\nu_{n+1}\e |\nabla u|^2\right|+2E(B_1^{n+1})+2|\xi|(B_1^{n+1})\\\nonumber
&\leq \left|\int_{A\cap\{|u|\leq 1-b\}}\nu_{n+1}\e|\nabla u|^2 dx-\alpha\omega_{n}\right|+\left|\int_{B_1^{n+1}\setminus A}\e |\nabla u|^2\right|+3E(B_1^{n+1})+2|\xi|(B_1^{n+1})\\\nonumber
&\leq \left|\int_{A\cap\{|u|\leq 1-b\}}\nu_{n+1}\e|\nabla u|^2 dx-\alpha\omega_{n}\right|+\mu(B_1^{n+1}\setminus A)+3E(B_1^{n+1})+2|\xi|(B_1^{n+1}),
\end{align}

We estimate the first term by using the co-area formula and using the Lipschitz graph representation of the level sets of $u$, \eqref{FunctionAndLevelSet}, 
\begin{align}
&\left|\int_{A\cap\{|u|\leq 1-b\}}\nu_{n+1}\e|\nabla u|^2 dx-\alpha\omega_{n}\right|\\\nonumber
&=\left|\int_{-1+b}^{1-b}\int_{A\cap\{u=s\}}\nu_{n+1}\e |\nabla u| d\mathcal H^nds-\alpha\omega_{n}\right|\\\nonumber
&=\left|\int_{-1+b}^{1-b}\int_{A\cap\{u=s\}}\frac{1}{\sqrt{1+|\hat\nabla h^s|^2}}\e |\nabla u| d\mathcal H^nds-\int_{-1}^1g'(g^{-1}(s)) ds\omega_{n}\right|\\\nonumber
&=\left|\int_{-1+b}^{1-b}\int_{A\cap\{u=s\}}\frac{1}{\sqrt{1+|\hat\nabla h^s|^2}}g'(\frac{z}{\e})|\nabla z| d\mathcal H^nds-\int_{-1}^1g'(g^{-1}(s)) ds\omega_{n}\right|.
\end{align}
Now using the Lipschitz bounds on $h^{s}$, \eqref{FarFromTransitionEnergySmall1}-\eqref{FarFromTransitionEnergySmall2}, as a consequence of the choice of $b$, we get from the last term  
\begin{align}
&\left|\int_{A\cap\{|u|\leq 1-b\}}\nu_{n+1}\e|\nabla u|^2 dx-\alpha\omega_{n}\right|\\\nonumber
&\leq \left|\int_{-1+b}^{1-b}\int_{\Pi(A\cap\{u=s\})\cap\hat B_1}g'(\frac{z}{\e})|\nabla z| d\hat xds-\int_{-1}^{1}\int_{\hat B_1}g'(\frac{z}{\e}) d\hat xds\right|\\\nonumber
&\leq \left|\int_{-1+b}^{1-b}\int_{\Pi(A\cap\{u=s\})\cap\hat B_1}g'(\frac{z}{\e})|\nabla z| d\hat xds-\int_{-1}^{1}\int_{\Pi(A\cap\{u=s\}\cap\hat B_1}g'(\frac{z}{\e}) d\hat xds\right|+\alpha\mathcal H^n(\hat B_1\setminus\pi(A))\\\nonumber
&\leq \left|\int_{-1+b}^{1-b}\int_{\Pi(A\cap\{u=s\})\cap\hat B_1}g'(\frac{z}{\e})|\nabla z| d\hat xds-\int_{-1+b}^{1-b}\int_{\Pi(A\cap\{u=s\})\cap\hat B_1}g'(\frac{z}{\e}) d\hat xds\right|+\alpha\mathcal H^n(\hat B_1\setminus\pi(A))\\\nonumber
&+\left|\int_{-1}^{-1+b}\int_{\Pi(A\cap\{u=s\})\cap\hat B_1}g'(\frac{z}{\e}) d\hat xds+\int_{1-b}^{1}\int_{\Pi(A\cap\{u=s\})\cap\hat B_1}g'(\frac{z}{\e}) d\hat xds\right|\\\nonumber
&\leq \left|\int_{-1+b}^{1-b}\int_{\Pi(A\cap\{u=s\})\cap\hat B_1}g'(\frac{z}{\e})|\nabla z| d\hat xds-\int_{-1+b}^{1-b}\int_{\Pi(A\cap\{u=s\})\cap\hat B_1}g'(\frac{z}{\e}) d\hat xds\right|+\alpha\mathcal H^n(\hat B_1\setminus\pi(A))\\\nonumber
&+E(B_3^{n+1})+|\xi|(B_3^{n+1})\\\nonumber
&\leq \left|\int_{-1+b}^{1-b}\int_{\Pi(A)\cap\hat B_1}g'(\frac{z}{\e})(|\nabla z|-1) d\hat xds\right|+\alpha\mathcal H^n(\hat B_1\setminus\pi(A))+E(B_3^{n+1})+|\xi|(B_3^{n+1}).
\end{align}
The non-positivity of the discrepancy \eqref{eqn_discrepancy} implies $|\nabla z|\leq 1$, and thus $\left||\nabla z|-1\right|\leq \left||\nabla z|^2-1\right|$. Hence the first term above can be estimated by
\begin{align}
&\left|\int_{-1+b}^{1-b}\int_{\Pi(A)\cap\hat B_1}g'(\frac{z}{\e})(|\nabla z|-1) d\hat xds\right|\\\nonumber
&\leq \left|\int_{-1+b}^{1-b}\int_{\Pi(A)\cap\hat B_1}\left(\frac{2\e}{|g'(\frac{z}{\e})|}\right)\left(\frac{1}{2\e}[g'(\frac{z}{\e})]^2(|\nabla z|^2-1)\right) d\hat xds\right|\\\nonumber
&=\left|\int_{-1+b}^{1-b}\int_{\Pi(A)\cap\hat B_1}\left(\frac{2\e}{|g'(\frac{z}{\e})|}\right)\left(\frac{\e|\nabla u^\e|^2}{2}-\frac{W(u^\e)}{\e}\right) d\hat xds\right|\\\nonumber
&=\left|\int_{-\infty}^\infty\int_{\Pi(A\cap\{|u|\leq 1-b\})\cap\hat B_1}\frac{2\e}{|g'(\frac{z}{\e})|}\left(\frac{\e|\nabla u^\e|^2}{2}-\frac{W(u^\e)}{\e}\right)\frac{\partial u}{\partial x_{n+1}} d\hat xdx_{n+1}\right|\\\nonumber
&=\left|\int_{-\infty}^\infty\int_{\Pi(A\cap\{|u|\leq 1-b\})\cap\hat B_1}\frac{2\e}{|g'(\frac{z}{\e})|}\left(\frac{\e|\nabla u^\e|^2}{2}-\frac{W(u^\e)}{\e}\right)\frac{g'(\frac{z}{\e})\frac{\partial z}{\partial x_{n+1}}}{\e} d\hat xdx_{n+1}\right|\\\nonumber
&\leq \left|\int_{-1}^1\int_{\Pi(A\cap\{|u|\leq 1-b\})\cap\hat B_1}2\left|\frac{\partial z}{\partial x_{n+1}}\right|\left(\frac{\e|\nabla u^\e|^2}{2}-\frac{W(u^\e)}{\e}\right) d\hat xdx_{n+1}\right|\\\nonumber
&\leq 2|\xi|(B_1^{n+1}),
\end{align}
where in the last inequality, we used. $\left|\frac{\partial z}{\partial x_{n+1}}\right|\leq |\nabla z|\leq 1$. Inserting this estimate back into \eqref{EnergyDifference1} we get
\begin{align}\label{EnergyDifference}
&\left|\int_{B_1^{n+1}}\left(\frac{\e|\nabla u|^2}{2}+ \frac{W(u)}{\e}\right) dx-\alpha\omega_{n}\right|\\ \nonumber
&\leq O\left(\mathcal H^n(B^n_1\setminus\Pi (A))+\mu(B_1^{n+1}\setminus A)+E(B_3^{n+1})+|\xi|(B_3^{n+1})\right).
\end{align}
Now, we turn to estimating the term $\mathcal H^n(B^n_1\setminus\Pi (A))$. By the definition of $A^{\e,\gamma}$, for any $(\hat x,x_{n+1})\not \in A^{\e,\gamma}$ either 
\begin{equation*}
(\hat x,x_{n+1})\in X_1^{\e,\gamma}=\{r^2 W(B_r)>\gamma^2\mu(B_r), \text{for some $r\in(0,2)$}\}
\end{equation*}
or 
\begin{align*}
(\hat x,x_{n+1})\in X_2^{\e,\gamma}=\{r^nE(B_r)+|\xi|(B_r(x))>\gamma^2\mu(B_r), \text{for some $r\in(0,2)$}\}.
\end{align*}
First we consider $X_1$, for any $y\in X_1$, there exists $0<R_y\leq 2$ such that 
\begin{align*}
R^2_y\mathcal{W}(B_{R_y}(y))>\gamma^2\mu(B_{R_y}(y)).
\end{align*}
Moreover, by Lemma \ref{DensityLowerBound} in the Appendix, there exists $K_0,R_0>0$ such that either 
\begin{align*}
R_y\leq 2\e R_0
\end{align*}
or 
\begin{align*}
\mu(B_{R_y}(y))\geq K_0\omega_nR_y^n
\end{align*}
is satisfied. In the later case,
\begin{align*}
K_0\gamma^2 \omega_nR_y^n<R_y^2\mathcal{W}(B_{R_y}(y))\implies K_0\gamma^2 \omega_nR_y^{n-2}&<\mathcal{W}(B_{R_y}(y)).
\end{align*}
Therefore, we have
\begin{align*}
R_y<\left(\frac{1}{K_0}\gamma^{-2}\omega_n^{-1}\int_{B_{R_y}(y)}\e\left(\Delta u-\frac{W'(u)}{\e^2}\right)^2 dx\right)^{\frac{1}{n-2}}.
\end{align*}
Since $y\in B_1$ and $R_y<2$, we get
\begin{align*}
R_y<\left(\frac{1}{K_0}\gamma^{-2}\omega_n^{-1}\mathcal{W}(B_3)\right)^{\frac{1}{n-2}}.
\end{align*}
Combining the two cases, we have by the technical assumption \eqref{TechnicalAssumption}
\begin{align}\label{eqn_Radius}
R_y^2&\leq \max\left\{4\e^2 R_0^2,\left(\frac{1}{K_0}\gamma^{-2}\omega_n^{-1}\int_{B_{R_y}(y)}\e\left(\Delta u-\frac{W'(u)}{\e^2}\right)^2 dx\right)^{\frac{2}{n-2}}\right\}\\\nonumber
&\leq \max\left\{\tilde CK_0^2E(B_3),\left(\frac{1}{K_0}\gamma^{-2}\omega_n^{-1}\int_{B_{R_y}(y)}\e\left(\Delta u-\frac{W'(u)}{\e^2}\right)^2 dx\right)^{\frac{2}{n-2}}\right\}.
\end{align}
We only consider the dimensions $n>2$ here (the case $n\leq 2$ has already been dealt with by the authors in Corollary 1.5 of \cite{Nguyen2020}). Substituting \eqref{eqn_Radius} into our definition of $X_1^{\e,\gamma}$, $\gamma^{-2}R^2_y\mathcal{W}(B_{R_y}(y))>\mu(B_{R_y}(y))$, we get
\begin{align}\label{CoveringX1}
\mu(B_{R_y}(y))&<\max\left\{\tilde CK_0^2E(B_3)\mathcal{W}(B_{R_y}(y)),\gamma^{-2}\left(\frac{1}{K_0}\gamma^{-2}\omega_n^{-1}\mathcal{W}(B_3)\right)^{\frac{2}{n-2}}\mathcal{W}(B_{R_y}(y))\right\}.
\end{align}
Since for any $y\in X_1$, there is a ball of radius $R_y$ satisfying the above \eqref{CoveringX1}, by the Besicovitch covering theorem, we have
\begin{align*}
\mu(X_1^{\e,\gamma})&\leq C(n)\max\left\{E(B_3)\mathcal W(B_3),\gamma^{-\frac{2n}{n-2}}\mathcal{W}(B_3)^\frac{n}{n-2}\right\}.
\end{align*}
For the set $X_2^{\e,\gamma}$, a similar covering argument gives
\begin{align*}
\mu(X_2^{\e,\gamma})&\leq C(n)\gamma^{-2}(E(B_3)+|\xi|(B_3)).
\end{align*}
Thus
\begin{align*}
\mu(B_3^{n+1}\setminus A)&\leq \mu(X_1^{\e,\gamma})+\mu(X_2^{\e,\gamma})\\
&\leq C(n)\left(E(B_3)+|\xi|(B_3)+E(B_3)\mathcal W(B_3)+\mathcal{W}(B_3)^\frac{n}{n-2}\right).
\end{align*}
Moreover, the projection on to $B_1^n$ gives us
\begin{align*}
\mathcal H^n(B^n_3\setminus\Pi(A))&\leq \mu(B_3^{n+1}\setminus A)+C(n)E(B_3)\\
&\leq \tilde C(n)\left(E(B_3)+|\xi|(B_3)+E(B_3)\mathcal W(B_3)+\mathcal{W}(B_3)^\frac{n}{n-2}\right).
\end{align*}
Substituting the above two estimates on the measure of bad sets $X_1^{\e,\gamma}, X_2^{\e,\gamma}$ into \eqref{EnergyDifference} and applying the technical assumption \eqref{TechnicalAssumption}, we have
\begin{align*}
&\left|\int_{B_1}\left(\frac{\e|\nabla u|^2}{2}+ \frac{W(u)}{\e}\right) dx-\alpha\omega_{n}\right|\\
\leq &C_3\left(E(B_3)+|\xi|(B_3)+E(B_3)\mathcal W(B_3)+\mathcal{W}(B_3)^\frac{n}{n-2}\right).
\end{align*}
\end{proof}

\subsection{Cylindrical growth estimate}
Since the Sobolev-type inequality in Proposition \ref{InteriorSobolev} is an interior estimate, we require a cylindrical growth estimate in order to extend it to balls of larger radius. The lemma below corresponds to \cite[Section 6.4]{brakke2015motion} and \cite[Theorem 11.3]{Kasai2014}. We will consider energy measures of functions with nearly flat level sets and nearly orthogonal cylinders over them. The following lemma estimates the growth of the Allen--Cahn energy as a function of radius. The estimates provide both upper and lower bounds. 
\begin{lemma}\label{CylindricalGrowthLem}
Suppose there exists $0<R_1<R_2<\infty$ and $W_2,E_2>0$ such that for any $R_1<r<R_2$
\begin{align*}
\mathcal W(B_r)&\leq W_2r^n,\\
E(B_r)&\leq E_2,
\end{align*}
then there exists $C(n, R_1,R_2)$ such that
\begin{align*}
\left|\frac{\mu(B_{R_2})}{R_2^n}-\frac{\mu(B_{R_1})}{R_1^n}\right|\leq C\left(E_2+\sqrt{\mathcal{W}_2E_2}+ |\xi|(B_{R_2})\right).
\end{align*}
\end{lemma}
\begin{proof}
Consider $C_r=\{x=(\hat x,x_{n+1})\in\mathbb R^{n+1}:|\hat x|\leq r\}$ and choose a cylindrical vector field $g:\mathbb R^{n+1}\rightarrow\mathbb R^{n+1}: g(x)=(g^1,\dots,g^n,g^{n+1})=\frac{1}{r}\rho(\frac{x}{r})\hat x=\frac{1}{r}\rho(\frac{|\hat x|}{r})\hat x$, where $\rho\in C_c^\infty(\mathbb R^{n+1})$ satisfying $|\rho|\leq 1$, $\spt\rho\subset C_1$ and $ \rho(x)=\rho(|\hat x|)$ is a cylindrical test function which depends only on $|\hat x|$. We compute
\begin{align*}
\mathrm{div}(g)&=n\frac{1}{r}\rho(\frac{|\hat x|}{r})+\left\langle\hat\nabla \left(\frac{1}{r}\rho(\frac{|\hat x|}{r})\right),\hat x\right\rangle=n\frac{1}{r}\rho(\frac{|\hat x|}{r})+\frac{1}{r^2}\rho'(\frac{|\hat x|}{r})|\hat x|,\\
&\hat\nabla\rho(\frac{\hat x}{r})=\rho'(\frac{\hat x}{r})\frac{1}{r}\frac{\hat x}{|\hat x|}=-r\frac{\partial}{\partial r}\rho(\frac{\hat x}{r})\frac{\hat x}{|\hat x|^2},\\
\frac{\partial}{\partial x_i}g^j&=0, \text{$i$ or $j$ equals $n+1$},\\
\frac{\partial}{\partial x_i}g^j&=\frac{1}{r}\rho(\frac{\hat x}{r})\delta_{ij}+ \frac{1}{r^2}\rho'(\frac{\hat x}{r})\frac{x_ix_j}{|\hat x|}=\frac{1}{r}\rho(\frac{\hat x}{r})\delta_{ij}-\frac{\partial}{\partial r}\rho(\frac{\hat x}{r})\frac{x_ix_j}{|\hat x|^2},  \text{$1\leq i,j\leq n$}.
\end{align*}
The vector field $g$ corresponds to the velocity field of the variation $u_s=u+s\langle\nabla u, g\rangle$. Substituting into the first variation formula, we obtain
\begin{align*}
&\int_{C_r}\left(-\e\Delta u+\frac{W'(u)}{\e}\right)\langle\nabla u,g\rangle dx\\
&=\frac{d}{ds}\int_{C_r}\left(\frac{\e|\nabla u_s|^2}{2}+ \frac{W(u_s)}{\e}\right) dx\bigg|_{s=0}\\
&=-\int_{C_r}\left(I-\frac{\nabla u}{|\nabla u|}\otimes\frac{\nabla u}{|\nabla u|}\right):Dg\left(\frac{\e|\nabla u|^2}{2}+ \frac{W(u)}{\e}\right) dx\\
&+\int_{C_r}\frac{\nabla u}{|\nabla u|}\otimes\frac{\nabla u}{|\nabla u|}:Dg\left(\frac{\e|\nabla u|^2}{2}-\frac{W(u)}{\e}\right) dx\\
&=-\int_{C_r}\left(I-\frac{\nabla u}{|\nabla u|}\otimes\frac{\nabla u}{|\nabla u|}\right):\hat\delta_{ij}\frac{1}{r}\rho(\frac{\hat x}{r}) d\mu^\e-\int_{C_r}\left(I-\frac{\nabla u}{|\nabla u|}\otimes\frac{\nabla u}{|\nabla u|}\right):\left(\hat\nabla\rho(\frac{\hat x}{r})\otimes\hat x\right)\frac{1}{r} d\mu^\e\\
&+\int_{C_r}\frac{\nabla u}{|\nabla u|}\otimes\frac{\nabla u}{|\nabla u|}: Dg d\xi_\e\\
&=-\int_{C_r}\left(I-\frac{\nabla u}{|\nabla u|}\otimes\frac{\nabla u}{|\nabla u|}\right):\hat\delta_{ij}\frac{1}{r}\rho(\frac{\hat x}{r}) d\mu^\e+\int_{C_r}\left(I-\frac{\nabla u}{|\nabla u|}\otimes\frac{\nabla u}{|\nabla u|}\right):\left(\frac{\hat x}{|\hat x|}\otimes\frac{\hat x}{|\hat x|}\right)\frac{\partial}{\partial r}\rho(\frac{\hat x}{r}) d\mu^\e\\
&+\int_{C_r}\frac{\nabla u}{|\nabla u|}\otimes\frac{\nabla u}{|\nabla u|}:Dg d\xi_\e.
\end{align*}
Thus, by moving the other two terms on the last line to the left hand side, we have
\begin{align*}
&\frac{d}{dr}\int_{C_r}\left(I-\frac{\nabla u}{|\nabla u|}\otimes\frac{\nabla u}{|\nabla u|}\right):\left(\frac{\hat x}{|\hat x|}\otimes\frac{\hat x}{|\hat x|}\right)\rho(\frac{\hat x}{r}) d\mu^\e\\
&=\int_{C_r}\left(I-\frac{\nabla u}{|\nabla u|}\otimes\frac{\nabla u}{|\nabla u|}\right):\left(\frac{\hat x}{|\hat x|}\otimes\frac{\hat x}{|\hat x|}\right)\frac{\partial}{\partial r}\rho(\frac{\hat x}{r}) d\mu^\e\\
&=-\int_{C_r}\frac{\nabla u}{|\nabla u|}\otimes\frac{\nabla u}{|\nabla u|}:Dg d\xi_\e+\int_{C_r}\left(I-\frac{\nabla u}{|\nabla u|}\otimes\frac{\nabla u}{|\nabla u|}\right):\hat\delta_{ij}\frac{1}{r}\rho(\frac{\hat x}{r}) d\mu^\e\\
&-\int_{C_r}\left(\e\Delta u-\frac{W'(u)}{\e}\right)\langle\nabla u,g\rangle dx.
\end{align*}
Since $ g^{n+1}=0$ and by applying Cauchy-Schwarz (notice that $|g|\leq 1$), we get
\begin{align*}
&\left|\int_{C_r}\left(-\e\Delta u+\frac{W'(u)}{\e}\right)\langle\nabla u,g\rangle dx\right|\\
&\leq \left|\int_{C_r}\left(-\e\Delta u+\frac{W'(u)}{\e}\right)|\hat\nabla u| dx\right|\\
&=\left|\int_{C_r}\left(-\e\Delta u+\frac{W'(u)}{\e}\right)\sqrt{(1-\nu_{n+1}^2)}|\nabla u| dx\right|\\
&\leq \left(\int_{C_r}\e\left(\Delta u-\frac{W'(u)}{\e^2}\right)^2 dx\right)^\frac{1}{2}\cdot \left(\int_{C_r}\e(1-\nu_{n+1}^2)|\nabla u|^2 dx\right)^\frac{1}{2}.
\end{align*}
Combining the above two inequalities, we have
\begin{align}\label{CylindricalGrowth}
&\left|\frac{d}{dr}\int_{C_r}\left(I-\nu\otimes\nu\right):\left(\frac{\hat x}{|\hat x|}\otimes\frac{\hat x}{|\hat x|}\right)\rho(\frac{\hat x}{r}) d\mu^\e-\frac{n}{r}\int_{C_r}\left(I-\nu\otimes\nu\right):\left(\frac{\hat x}{|\hat x|}\otimes\frac{\hat x}{|\hat x|}\right)\rho(\frac{\hat x}{r}) d\mu^\e \right|\\\nonumber
&\leq \left(\int_{C_r}\e\left(\Delta u-\frac{W'(u)}{\e^2}\right)^2 dx\right)\cdot \left(\int_{C_r}(1-\nu_{n+1}^2) \e|\nabla u|^2 dx\right)+\left|\int_{C_r}\nu\otimes\nu:Dg d\xi_\e\right|\\\nonumber
&+\left|\int_{C_r}\left(I-\nu\otimes\nu\right):\hat\delta_{ij}\frac{1}{r}\rho(\frac{\hat x}{r}) d\mu^\e-\frac{n}{r}\int_{C_r}\left(I-\nu\otimes\nu\right):\left(\frac{\hat x}{|\hat x|}\otimes\frac{\hat x}{|\hat x|}\right)\rho(\frac{\hat x}{r}) d\mu^\e\right|.
\end{align}
An application of proof of \cite[(11.18)]{Kasai2014} (with $T=I-\nu\otimes\nu, S=\hat\delta_{ij}$) gives us the estimate on the last term of the last line
\begin{align*}
\left|\int_{C_r}\left(I-\nu\otimes\nu\right):\hat\delta_{ij}\frac{1}{r}\rho(\frac{\hat x}{r}) d\mu^\e-\frac{n}{r}\int_{C_r}\left(I-\nu\otimes\nu\right):\left(\frac{\hat x}{|\hat x|}\otimes\frac{\hat x}{|\hat x|}\right)\rho(\frac{\hat x}{r}) d\mu^\e\right| \\
\leq n\left|\int_{C_r}(1-\nu_{n+1}^2)\frac{1}{r}\rho(\frac{\hat x}{r}) d\mu^\e\right|.
\end{align*}
Hence we have 
\begin{align}
&\left|\frac{d}{dr}\int_{C_r}\left(I-\nu\otimes\nu\right):\left(\frac{\hat x}{|\hat x|}\otimes\frac{\hat x}{|\hat x|}\right)\rho(\frac{\hat x}{r}) d\mu^\e-\frac{n}{r}\int_{C_r}\left(I-\nu\otimes\nu\right):\left(\frac{\hat x}{|\hat x|}\otimes\frac{\hat x}{|\hat x|}\right)\rho(\frac{\hat x}{r}) d\mu^\e \right|\\\nonumber
&\leq \left(\int_{C_r}\e\left(\Delta u-\frac{W'(u)}{\e^2}\right)^2 dx\right)\cdot \left(\int_{C_r}(1-\nu_{n+1}^2)\e |\nabla u|^2 dx\right)+\int_{C_r}\nu\otimes\nu:Dg d\xi_\e\\\nonumber
&+n\left|\int_{C_r}(1-\nu_{n+1}^2)\frac{1}{r}\rho(\frac{\hat x}{r}) d\mu^\e\right|\\\nonumber
&\leq \left(\int_{C_r}\e\left(\Delta u-\frac{W'(u)}{\e^2}\right)^2 dx\right)\cdot \left(\int_{C_r}\e(1-\nu_{n+1}^2)|\nabla u|^2 dx\right)+\int_{C_r}\frac{\nabla u}{|\nabla u|}\otimes\frac{\nabla u}{|\nabla u|}:Dg d\xi_\e\\\nonumber
&+n\left|\int_{C_r}(1-\nu_{n+1}^2)\frac{1}{r}\rho(\frac{\hat x}{r}) \e|\nabla u|^2 dx\right|+\frac{1}{2}\left|\int_{C_r}(1-\nu_{n+1}^2)\frac{1}{r}\rho(\frac{\hat x}{r}) d\xi_\e\right|\\\nonumber
&\leq C_1\int d\xi_\e+C_2\left(\int_{C_r}\e\left(\Delta u-\frac{W'(u)}{\e^2}\right)^2 dx\right)\cdot \left(\int_{C_r}(1-\nu_{n+1}^2)\e |\nabla u|^2 dx\right)\\\nonumber
&+C_3\left(\int_{C_r}(1-\nu_{n+1}^2)\e|\nabla u|^2 dx\right).
\end{align}
Dividing \eqref{CylindricalGrowth} by $r^{n}$ and integrating with respect to $r$ from $R_1$ to $R_2$, we obtain the cylindrical growth estimate
\begin{align}
\left|\frac{\int_{B_{R_2}} d\mu^\e}{R_2^{n}}-\frac{\int_{B_{R_1}} d\mu^\e}{R_1^{n}}\right|\leq C_1 E(B_{R_2})+C_2\sqrt{E(B_{R_2})\mathcal W(B_{R_2})}+ C_3|\xi|(B_{R_2}), \forall r\in[0,2].
\end{align}
\end{proof}
\subsection{Energy inequality}
Fix any $\phi\in C^\infty([0,\infty),[0,1])$ such that 
\begin{align*}
\phi(x)
\begin{cases}
=1, \text{for $0\leq x\leq (\frac{2}{3})^\frac{1}{n}$},\\
>0, \text{for $0\leq x\leq (\frac{5}{6})^\frac{1}{n}$},\\
=0, \text{for $x\geq (\frac{5}{6})^\frac{1}{n}$}.
\end{cases}
\end{align*}
we denote by 
\begin{align*}
\phi_{T,R}(x)=\phi(R^{-1}|T(x)|), \phi_{T}(x)=\phi_{T,1}(x)=\phi(|T(x)|)
\end{align*}
for any $x\in\mathbb R^{n+1}$, $R>0$ and $T\in\mathbf G(n+1,n)$ in the Grassmannian. By this choice of cutoff function, we have
\begin{align*}
\frac{2}{3}\omega_n<\int_T\phi_T^2d\mathcal H^n<\frac{5}{6}\omega_n
\end{align*}
Without loss of generality we suppose $T=\{x_{n+1}=0\}$ here. Note that this means that $\phi_T$ depends only on the first $n$-variables $\hat x=(x_1, \cdots, x_n)$. Substituting in $\varphi=\phi_T^2$ into the Allen--Cahn Brakke identity \eqref{eqn_BrakkeAllenCahn1} and using Young's inequality and Corollary \ref{TimeSliceCaccioppoli}, we have
\begin{align}\label{DerivativeEnergy}
&\frac{d}{dt}\int\phi_T^2\left(\frac{\e |\nabla u|^2}{2}+ \frac{W(u)}{\e}\right) dx\\ \nonumber
&=-\int\e\phi_T^2\left(\Delta u-\frac{W'(u)}{\e^2}\right)^2dx-\int 2\e \phi_T\nabla\phi_T\cdot \nabla u\left(\Delta u-\frac{W'(u)}{\e^2}\right) dx\\ \nonumber
&\leq -\int\e\phi_T^2\left(\Delta u-\frac{W'(u)}{\e^2}\right)^2dx+\frac{1}{4}\int\e\phi_T^2\left(\Delta u-\frac{W'(u)}{\e^2}\right)^2dx\\ \nonumber
&+4\int|\nabla\phi_T|^2(1-\nu_{n+1}^2)\e|\nabla u|^2dx\\ \nonumber
&=-\frac{3}{4}\int\e\phi_T^2\left(\Delta u-\frac{W'(u)}{\e^2}\right)^2dx+4\int|\nabla\phi_T|^2(1-\nu_{n+1}^2)\e|\nabla u|^2dx\\\nonumber
&\leq -\frac{3}{4}\int\e\phi_T^2\left(\Delta u-\frac{W'(u)}{\e^2}\right)^2dx+C_1\int_{B_1}x_{n+1}^2\e|\nabla u|^2dx\\ \nonumber
&+C_2\left(\int_{B_1}\e \left(\Delta u-\frac{W'(u)}{\e^2}\right)^2dx\right)^\frac{1}{2}\left(\int_{B_1}x_{n+1}^2\e|\nabla u|^2dx\right)^\frac{1}{2} \text{(by Corollary \ref{TimeSliceCaccioppoli})}
\\ \nonumber
&\leq -\frac{3}{4}\int\e\phi_T^2\left(\Delta u-\frac{W'(u)}{\e^2}\right)^2dx+C_1\int_{B_1}x_{n+1}^2\e|\nabla u|^2dx\\\nonumber
&+\frac{1}{2}\int\e\phi_T^2\left(\Delta u-\frac{W'(u)}{\e^2}\right)^2dx+2C_2^2\int_{B_1}x_{n+1}^2\e|\nabla u|^2dx
\\\nonumber
&\leq -\frac{1}{4}\int_{B_1}\e\left(\Delta u-\frac{W'(u)}{\e^2}\right)^2dx+C_0\int_{B_1}x_{n+1}^2\e|\nabla u|^2dx.
\end{align}
Next, with the help of Caccioppoli inequality Theorem \ref{Caccioppoli} and Sobolev-type inequality Proposition \ref{InteriorSobolev} from the previous subsection, we have
\begin{lemma}[{cf \cite[Proposition 5.2]{Kasai2014}}]\label{lem_AreaWillmoreHeightBound}
For $n>2$ and under the assumptions of Theorem \ref{thm_ExcessDecay}, there exist constants $\mathbb H_1,E_1,\mathcal W_1,D_1>0$ such that the following holds.
If
\begin{align*}
\left|\mu(\phi_T^2)-\alpha\int_T\phi_T^2d\mathcal H^n\right|&=\left|\int\phi_T^2\left(\frac{\e |\nabla u|^2}{2}+ \frac{W(u)}{\e}\right) dx-\alpha\int_T\phi_T^2d\mathcal H^n\right|\leq \frac{1}{8}\alpha\int_T\phi_T^2d\mathcal H^n,\\
\mathcal W(B_1)&=\int_{B_1}\e\left(\Delta u-\frac{W(u)}{\e^2}\right)^2 dx\leq \mathcal W_1,\\
\mathbb H(B_1)&=\int_{B_1}\left|T^\perp(x)\right|^2\e|\nabla u|^2dx\leq \mathbb H_1,
\end{align*}
then
\begin{align*}
\left|\mu(\phi_T^2)-\alpha\int_T\phi_T^2d\mathcal H^n\right|\leq D_1\left(\mathcal{W}(B_1)^\frac{n}{n-2}+ \mathbb H(B_1)+\mathcal{W}(B_1)^\frac{3}{4}\mathbb H(B_1)^\frac{1}{4}\right).
\end{align*}
\end{lemma}
\begin{remark}
Notice that under the assumption of Theorem \ref{thm_ExcessDecay}, we are only considering situations with almost unit density. 
\end{remark}
\begin{proof}
Using Lemma \ref{CylindricalGrowthLem}, we can extend the interior Sobolev inequality in Proposition \ref{InteriorSobolev} to a global one. Namely, in the ball of radius $2$, we have
\begin{align}\label{EnergyDifference2}
\left|\frac{\mu(B_2(x))}{2^n}-\alpha\omega_n2^n\right|\leq C\left(E(B_2(x))+\sqrt{E(B_2(x))\mathcal{W}(B_2(x))}+ |\xi|(B_2(x))+\mathcal{W}(B_2(x))^\frac{n}{n-2}\right).
\end{align}
Next we use the Caccioppoli inequality Theorem \ref{Caccioppoli} and the technical assumption \ref{TechnicalAssumption} to bound the tilt excess by the height excess.
\begin{align*}
E(B_2(x))+|\xi|(B_2(x))&\leq O\left( \mathbb H(B_2(x))+ \sqrt{ \mathcal{W}(B_2(x)) \mathbb H(B_2(x))}\right)+o(\e^2)\\
&\leq O\left(\mathbb H(B_2(x))+\sqrt{ \mathcal{W}(B_2(x)) \mathbb H(B_2(x))}\right).
\end{align*}
Substituting this into \eqref{EnergyDifference2}, we have
\begin{align*}
&\left|\frac{\mu(B_2(x))}{2^n}-\alpha\omega_n2^n\right|\\
&\leq O\left(E(B_2(x))+\sqrt{E(B_2(x))\mathcal{W}(B_2(x))}+ |\xi|(B_2(x))+\mathcal{W}(B_2(x))^\frac{n}{n-2}\right)\\
&\leq O\left(\mathbb H(B_2(x))+\sqrt{\mathbb H(B_2(x))\mathcal{W}(B_2(x))}+\sqrt{\mathcal{W}(B_2(x))\cdot (\mathbb H(B_2(x))+\sqrt{ \mathcal{W}(B_2(x)) \mathbb H(B_2(x))})}\right)\\
&+O\left(\mathcal{W}(B_2(x))^\frac{n}{n-2}\right)\\
&\leq O\left(\mathbb H(B_2(x))+ \mathbb H(B_2(x))^\frac{1}{2}\mathcal W(B_2(x))^\frac{1}{2}+\mathbb H(B_2(x))^\frac{1}{4}\mathcal W(B_2(x))^\frac{3}{4}+\mathcal{W}(B_2(x))^\frac{n}{n-2}\right)\\
&\leq O\left(\mathbb H(B_2(x))+\mathbb H(B_2(x))^\frac{1}{4}\mathcal W(B_2(x))^\frac{3}{4}+\mathcal{W}(B_2(x))^\frac{n}{n-2}\right)
\end{align*}
 where the second last line follows since both the $\mathbb H$ and $\mathcal W$ terms are positive and the last line follows by Young's inequality.
\end{proof}
As a corollary, we see that the height excess gives a lower bound for the Willmore type term.
\begin{corollary}\label{BoundOnW}
For the same $\mathbb H_1, W_1, D_1$ as in the previous lemma, Lemma \ref{lem_AreaWillmoreHeightBound}, if
\begin{align*}
2D_1\mathbb H^2(B_1)&\leq \left|\mu(\phi_T^2)-\alpha\int_T\phi_T^2d\mathcal H^n\right|\leq \alpha\omega_n,
\end{align*}
then
\begin{align*}
\mathcal{W}(B_1)\geq \min\left\{(4D_1)^{-\frac{n-2}{n}}\hat\mu(B_1)^{\frac{n-2}{n}}, (4D_1)^{-\frac{4}{3}}\mathbb H^{-\frac{3}{2}}\hat\mu(B_1)^{\frac{4}{3}}\right\}
\end{align*}
where $\hat\mu(B_1):=\left|\mu(B_1)-\alpha\int_T\phi_T^2d\mathcal H^n\right|$.
\end{corollary}
\begin{proof}
By the assumption $2D_1\mathbb H^2(B_1)\leq \left|\mu(\phi_T^2)-\alpha\int_T\phi_T^2d\mathcal H^n\right|$ and the Lemma \ref{lem_AreaWillmoreHeightBound}, we have
\begin{align*}
\hat\mu(B_1)&\leq D_1(\mathcal W(B_1)^\frac{n}{n-2}+\mathbb H(B_1)+\mathcal W(B_1)^\frac{3}{4}\mathbb H^\frac{1}{4})\\
&\leq 2D_1(\mathcal W(B_1)^\frac{n}{n-2}+\mathcal W(B_1)^\frac{3}{4}\mathbb H^\frac{1}{4})\\
&\leq 4D_1\max\{\mathcal W(B_1)^\frac{n}{n-2},\mathcal W(B_1)^\frac{3}{4}\mathbb H^\frac{1}{4}\}.
\end{align*}
And thus the conclusion follows by taking appropriate powers.
\end{proof}
Now we can substitute Corollary \ref{BoundOnW} into the energy inequality \eqref{DerivativeEnergy} and have the following space-time Caccioppoli inequality. This theorem bounds the energy difference between $u$ and the flat solution by the height excess. Compared to Allard's regularity, this corresponds to the control of the tilt excess in terms of the height excess which is a reverse Poincar\'e (or Caccioppoli inequality). This theorem shows the deviation of the Allen--Cahn energy (or diffuse area) of the moving generalized varifold from a flat plane may be controlled by the Allen--Cahn tilt excess. In terms of Allard regularity, this is a Caccioppoli type inequality where the tilt excess is bounded by the height excess. Compared to the Brakke flow, this is an energy type inequality and corresponds to \cite[Section 6.6 Popping Soap Films]{brakke2015motion}. Here we only have the $\e$-Brakke Allen--Cahn equation which controls the rate of change of the Allen--Cahn energy in time. However, there is a Willmore type term that corresponds to the square of the diffuse mean curvature. This means that if we are far from a $k$-plane then the diffuse mean curvature is large, then we have additional decay of the energy difference. 

\begin{theorem}[c.f. {\cite[Theorem 5.7]{Kasai2014}}]\label{ParabolicCaccioppoli}
Given a $\nu>0$ then there exists $\e_0,\e_1, K, \Lambda>0$ with the following properties: if $u^\e$ is a solution to the \eqref{eqn_ACFe} with $\e\leq \e_0$
\begin{align}\label{ParabolicCaccioppoliHeightCondition}
u^\e(0,0)&=0,\\\nonumber
\sup_{0\leq t\leq 2\Lambda+3}\int_{C_1}x_{n+1}^2\e|\nabla u^\e(\cdot,t)|^2 dx&\leq \e_1,
\end{align}
and there exists a $\hat t_1 \in [0,1)$ such that
\begin{align}\label{eqn_SingleLayer}
\mu_{\hat t_1}(\phi_T^2) \leq (2-\nu) \alpha\int_T\phi_T^2d\mathcal H^n.
\end{align}
and there exists a $\hat t_2\in [2\Lambda+2,2\Lambda+3]$ such that
\begin{align}\label{eqn_FinalMass}
\mu_{\hat t_2}(\phi_T^2)\geq \nu \alpha\int_T\phi_T^2d\mathcal H^n
\end{align}
then we have
\begin{align}\label{eqn_EnergyBound}
\sup_{1+\Lambda\leq t\leq 2+\Lambda}\left|\mu(\phi_T^2)-\alpha\int_T\phi_T^2d\mathcal H^n\right|\leq K\sup_{0\leq t\leq 2\Lambda+3}\int_{C_1}x_{n+1}^2\e|\nabla u^\e(\cdot,t)|^2 dx,
\end{align}
and 
\begin{align}\label{eqn_WillmoreBound}
\int_{1+\Lambda}^{2+\Lambda}\int_{C_1}\e\phi_T^2\left(\Delta u^\e-\frac{W'(u^\e)}{\e^2}\right)^2 dxdt\leq 12K\sup_{0\leq t\leq 2\Lambda+3}\int_{C_1}x_{n+1}^2\e|\nabla u^\e(\cdot,t)|^2 dx.
\end{align}
\end{theorem}
\begin{remark}
This energy estimate can be viewed as a parabolic version of the Caccioppoli type inequality, which will be used in the proof of excess decay Theorem \ref{thm_ExcessDecay} in Section \ref{sec_ExcessDecay}. The energy bound conditions \eqref{eqn_SingleLayer} and \eqref{eqn_FinalMass} will be satisfied for almost unit density solutions of Allen-Cahn which is close enough in measure to a unit density Brakke flow when $\e$ is small enough.
\end{remark}
\begin{proof}
For this proof, we will drop the superscript $\e$ and write $ u=u^\e$. For $K_2$ a constant to be determined later, we define 
\begin{align*}
e(t):&=\int_{\mathbb R^{n+1}}\phi_T^2\left(\frac{\e|\nabla u(\cdot,t)|^2}{2}+ \frac{W(u(\cdot,t))}{\e}\right) dx-\alpha\int_T\phi_T^2d\mathcal H^n\\
&-K_2\cdot\left(\sup_{0\leq t\leq 2\Lambda+3}\int_{C_1}x_{n+1}^2\e|\nabla u(\cdot,t) |^2 dx\right)\cdot t.
\end{align*}
We will first prove the following inequality,  
\begin{claim}
\begin{align}\label{PushingEnergyDown}
e(t_2)-e(t_1)\leq -\frac{1}{4}\int_{t_1}^{t_2}\int_{\mathbb R^{n+1}}\phi_T^2\e\left(\Delta u-\frac{W'(u)}{\e^2}\right)^2 dxdt.
\end{align}
\end{claim}
\begin{proof}[Proof of Claim]
We compute using the integrated form of the $\e$-Brakke identity \eqref{eqn_BrakkeAllenCahn1} for $0\leq t_1<t_2\leq 2\Lambda+3$,
\begin{align}\label{eqn_Difference}
&e(t_2)-e(t_1)\\\nonumber
&=\int_{\mathbb R^{n+1}}\phi_T^2\left(\frac{\e|\nabla u(\cdot,t_2)|^2}{2}+ \frac{W(u(\cdot,t_2))}{\e}\right) dx
-\int_{\mathbb R^{n+1}}\phi_T^2\left(\frac{\e|\nabla u(\cdot,t_1)|^2}{2}+ \frac{W(u(\cdot,t_1))}{\e}\right) dx\\\nonumber
&-K_2\cdot\left(\sup_{0\leq t\leq 2\Lambda+3}\int_{C_1}x_{n+1}^2\e|\nabla u(\cdot,t)|^2 dx\right)(t_2-t_1)\\\nonumber
&=-\int_{t_1}^{t_2}\int_{\mathbb R^{n+1}}\e\left(\Delta u-\frac{W'(u)}{\e^2}\right)^2\phi_T^2 dxdt+2\int_{t_1}^{t_2}\int_{\mathbb R^{n+1}}\e\phi_T\langle\nabla\phi_T,\nabla u\rangle\left(\Delta u-\frac{W'(u)}{\e^2}\right) dxdt\\\nonumber
&-K_2\cdot\left(\sup_{0\leq t\leq 2\Lambda+3}\int_{C_1}x_{n+1}^2\e|\nabla u(\cdot,t)|^2 dx\right)(t_2-t_1)\\\nonumber
&\leq -\int_{t_1}^{t_2}\int_{\mathbb R^{n+1}}\e \phi_T^2 \left(\Delta u-\frac{W'(u)}{\e^2}\right)^2 dxdt+\frac{1}{4}\int_{t_1}^{t_2}\int_{\mathbb R^{n+1}}\e\phi_T^2\left(\Delta u-\frac{W'(u)}{\e^2}\right)^2 dxdt\\\nonumber
&+4\int_{t_1}^{t_2}\int_{\mathbb R^{n+1}}|\nabla\phi_T|^2(1-\nu_{n+1}^2)\e |\nabla u|^2 dxdt-K_2\cdot\left(\sup_{0\leq t\leq 2\Lambda+3}\int_{C_1}x_{n+1}^2\e|\nabla u(\cdot,t)|^2 dx\right)(t_2-t_1)\\\nonumber
&\leq -\frac{3}{4}\int_{t_1}^{t_2}\int_{\mathbb R^{n+1}}\e \phi_T^2 \left(\Delta u-\frac{W'(u)}{\e^2}\right)^2 dxdt
+4\int_{t_1}^{t_2}\int_{\mathbb R^{n+1}}|\nabla\phi_T|^2(1-\nu_{n+1}^2)\e |\nabla u|^2 dxdt\\ \nonumber
&-K_2\cdot\left(\sup_{0\leq t\leq 2\Lambda+3}\int_{C_1}x_{n+1}^2\e|\nabla u(\cdot,t)|^2 dx\right)(t_2-t_1).
\end{align}
We then apply our a slight modification of our Caccioppoli inequality, Theorem \ref{Caccioppoli} to the second term of the last inequality to get
 \begin{align*}
 4\int_{t_1}^{t_2}&\int_{\mathbb R^{n+1}}|\nabla\phi_T|^2(1-\nu_{n+1}^2)\e |\nabla u|^2 dxdt\\
 &\leq 4C_1 \int_{t_1}^{t_2}\int_{\mathbb R^{n+1}}| \nabla |\nabla \phi_T||^2 x_{n+1}^2 \e |\nabla u|^2 dx dt\\
&+4 C_2 \int_{t_1}^{t_2}\left( \int_{\mathbb R^{n+1}}\e \phi_T^2 \left(\Delta u(\cdot,t)-\frac{W'(u(\cdot,t))}{\e^2}\right)^2  dx \cdot \int_{\mathbb R^{n+1}}\frac{|\nabla \phi_T|^4}{\phi_T^2}x_{n+1}^2 \e |\nabla u(\cdot,t)|^2 dx  \right)^{1/2}dt\\
 &\leq 4C_1 \int_{t_1}^{t_2}\int_{\mathbb R^{n+1}}| \nabla |\nabla \phi_T||^2 x_{n+1}^2 \e |\nabla u|^2 dx dt+\frac{1}{2}\int_{t_1}^{t_2}\int_{\mathbb R^{n+1}}\e \phi_T^2 \left(\Delta u-\frac{W'(u)}{\e^2}\right)^2  dxdt\\
&+4 C_2^2 \int_{t_1}^{t_2}\int_{\mathbb R^{n+1}}\frac{|\nabla \phi_T|^4}{\phi_T^2}x_{n+1}^2 \e |\nabla u|^2 dx dt\\
&\leq \left(4C_1 \sup |\nabla^2 \phi_T|^2+4C_2^2 \sup \frac{|\nabla \phi_T|^4}{\phi_T^2}\right)  \int_{C_1}x_{n+1}^2 \e |\nabla u|^2 dx dt\\
&+\frac{1}{2}\int_{t_1}^{t_2}\int_{\mathbb R^{n+1}}\e \phi_T^2 \left(\Delta u-\frac{W'(u)}{\e^2}\right)^2  dxdt.
\end{align*}
Inserting this inequality into \eqref{eqn_Difference}, we get 
\begin{align*}
&e(t_2)-e(t_1)\\ 
&\leq \left(4C_1 \sup |\nabla^2 \phi_T|^2+4C_2^2 \sup \frac{|\nabla \phi_T|^4}{\phi_T^2}\right)  \int_{C_1}x_{n+1}^2 \e |\nabla u|^2 dx dt\\
&-K_2\cdot\left(\sup_{0\leq t\leq 2\Lambda+3}\int_{C_1}x_{n+1}^2\e|\nabla u(\cdot,t)|^2 dx\right)(t_2-t_1)\\
&+\frac{1}{2}\int_{t_1}^{t_2}\int_{\mathbb R^{n+1}}\e \phi_T^2 \left(\Delta u-\frac{W'(u)}{\e^2}\right)^2  dxdt\\
&-\frac{3}{4}\int_{t_1}^{t_2}\int_{\mathbb R^{n+1}}\e \phi_T^2 \left(\Delta u-\frac{W'(u)}{\e^2}\right)^2 dxdt\\
 &\leq \left(4C_1 \sup |\nabla^2 \phi_T|^2+4C_2^2 \sup \frac{|\nabla \phi_T|^4}{\phi_T^2}-K_2\right)\left(\sup_{0\leq t\leq 2\Lambda+3}\int_{C_1}x_{n+1}^2\e|\nabla u(\cdot,t)|^2 dx\right)(t_2-t_1)\\  
&-\frac{1}{4}\int_{t_1}^{t_2}\int_{\mathbb R^{n+1}}\e \phi_T^2 \left(\Delta u-\frac{W'(u)}{\e^2}\right)^2  dxdt.
\end{align*}
Therefore, if we choose $K_2=4C_1 \sup |\nabla^2 \phi_T|^2+4C_2^2 \sup \frac{|\nabla \phi_T|^4}{\phi_T^2}$, we have \eqref{PushingEnergyDown}.
\end{proof}
By this claim and the single sheet hypothesis \eqref{eqn_SingleLayer}, we have for any $ t \in [\hat t_1, 2\Lambda+3]$,
\begin{align}\label{eqn_eUpperBound}
e(t)\leq e(\hat t_1) &=\mu_{\hat t_1}(\phi_T^2)-\alpha\int_T\phi_T^2d\mathcal H^n-K_2\hat t_1 \sup_{0\leq t\leq 2\Lambda+3}\int_{C_1}x_{n+1}^2\e|\nabla u(\cdot,t)|^2 dx\\
\nonumber&\leq (1-\nu)\mu_{\hat t_1}(\phi_T^2). 
\end{align} 
 This in turn gives us for all $t\in[\hat t_1,2 \Lambda+3]$
\begin{align*}
\mu_{t}(\phi_T^2)-\alpha\int_T\phi_T^2d\mathcal H^n \leq \left(1-\frac{\nu}{2}\right) \alpha\int_T\phi_T^2d\mathcal H^n.
\end{align*}
Furthermore, we can also conclude for $ t\leq \hat t_2$, we have
\begin{align}\label{eqn_eLowerBound}
e(t)&\geq e(\hat t_2)\\
\nonumber&=\mu_{\hat t_2}(\phi_T^2)-\alpha\int_T\phi_T^2d\mathcal H^n-K_2\hat t_2 \sup_{0\leq t\leq 2\Lambda+3}\int_{C_1}x_{n+1}^2\e|\nabla u(\cdot,t)|^2 dx\\
\nonumber&\geq \left(-1+\frac{\nu}{2}\right)\alpha\int_T\phi_T^2d\mathcal H^n,
\end{align}
where we chose $\e_1=\min \left\{\frac{1}{K_2(3+2\Lambda)}\frac{\nu}{4}\alpha\int_T\phi_T^2d\mathcal H^n, \frac{\nu}{4}\alpha\int_T\phi_T^2d\mathcal H^n\right\}$ in the conditions. This give us that for all $ t \in [0, \hat t_2]$ we get 
\begin{align*}
\mu_{t}(\phi_T^2)-\alpha\int_T\phi_T^2d\mathcal H^n\geq \left(-1+\frac{\nu}{2}\right) \alpha\int_T\phi_T^2d\mathcal H^n
\end{align*}
Now for the sake of contradiction, let us suppose the conclusion \eqref{eqn_EnergyBound} or \eqref{eqn_WillmoreBound} are false. If we assume \eqref{eqn_EnergyBound} is false, then for any $K >0$, there exists a $t_*\in[\Lambda+1,\Lambda+2]$ such that 
\begin{align*}
\left|\mu_{t_*}(\phi_T^2)-\alpha\int_T\phi_T^2d\mathcal H^n\right|> K\sup_{0\leq t\leq 2\Lambda+3}\int_{C_1}x_{n+1}^2\e|\nabla u(\cdot,t)|^2 dx
\end{align*}
This implies either
\begin{align}\label{eqn_first}
\mu_{t_*}(\phi_T^2)-\alpha\int_T\phi_T^2d\mathcal H^n> K\sup_{0\leq t\leq 2\Lambda+3}\int_{C_1}x_{n+1}^2\e|\nabla u(\cdot,t)|^2 dx
\end{align}
or
\begin{align}\label{eqn_second}
\mu_{t_*}(\phi_T^2)-\alpha\int_T\phi_T^2d\mathcal H^n<-K\sup_{0\leq t\leq 2\Lambda+3}\int_{C_1}x_{n+1}^2\e|\nabla u(\cdot,t)|^2 dx
\end{align}
Consider the case \eqref{eqn_first}. Then for any $t\in[0,t_*]$, we have
\begin{align*}
\mu_t&(\phi_T^2)-\alpha\int_T\phi_T^2d\mathcal H^n\\
 &\geq e(t) \geq e(t_*)\\
&=\int_{\mathbb R^{n+1}}\phi_T^2\left(\frac{\e|\nabla u(\cdot,t)|^2}{2}+ \frac{W(u(\cdot,t))}{\e}\right) dx-\alpha\int_T\phi_T^2d\mathcal H^n\\
&-K_2\cdot\left(\sup_{0\leq t\leq 2\Lambda+3}\int_{C_1}x_{n+1}^2\e|\nabla u(\cdot,t)|^2 dx\right)\cdot t_*\\
&>K\sup_{0\leq t\leq 2\Lambda+3}\int_{C_1}x_{n+1}^2\e|\nabla u(\cdot,t)|^2 dx-K_2\cdot\left(\sup_{0\leq t\leq 2\Lambda+3}\int_{C_1}x_{n+1}^2\e|\nabla u(\cdot,t)|^2 dx\right)(2\Lambda+3).
\end{align*}
So by choosing $K:=\max\{2,4D_1,2K_2(2\Lambda+3)\}$, we have
\begin{align}\label{FalseAssumption}
\mu(\phi_T^2)-\alpha\int_T\phi_T^2d\mathcal H^n&>\frac{K}{2}\sup_{0\leq t\leq 2\Lambda+3}\int_{C_1}x_{n+1}^2\e|\nabla u(\cdot,t)|^2 dx\\\nonumber
&>2D_1\sup_{0\leq t\leq 2\Lambda+3}\int_{C_1}x_{n+1}^2\e|\nabla u(\cdot,t)|^2 dx.
\end{align}
Thus the conditions in Corollary \ref{BoundOnW} are satisfied and we have
\begin{align*}
&\int_{C_1}\e\phi_T^2\left(\Delta u(\cdot,t)-\frac{W'(u(\cdot,t))}{\e^2}\right)^2 dx\\
&\geq 4P\min\left\{1,e(t)^\frac{n-2}{n},\left(\sup_{0\leq t\leq 2\Lambda+3}\int_{C_1}x_{n+1}^2\e|\nabla u(\cdot,t)|^2 dx\right)^{-\frac{2}{3}}e(t)^{\frac{4}{3}}\right\}.
\end{align*}
Here we chose $P:=\frac{1}{4\cdot2^\frac{4}{3}}\min\left\{\mathcal W_1,(4D_1)^{-\frac{(n-2)}{n}},(4D_1)^{-\frac{4}{3}}\right\}$, where $\mathcal W_1$ is from Lemma \ref{lem_AreaWillmoreHeightBound}. We will use this lower bound of the $L^2$ time derivative to show the energy mass will decrease at least a fixed amount after time $\Lambda$ to contradict the assumption \eqref{FalseAssumption}. This follows by applying Lemma 5.5 of \cite{Kasai2014} with $\Psi(t):=e(t)$ and $\Lambda$ chosen from this lemma. We have for $\Lambda\in[0,t_*]$
\begin{align*}
e(\Lambda)\leq 2D_1\sup_{0\leq t\leq 2\Lambda+3}\int_{C_1}x_{n+1}^2\e|\nabla u(\cdot,t)|^2 dx,
\end{align*}
which contradicts \eqref{FalseAssumption}. 

Now let us assume \eqref{eqn_second}. We then find
\begin{align*}
\mu_t(\phi_T^2)-\alpha\int_T\phi_T^2d\mathcal H^n& \leq e(t) \leq e(t_*)\\
&\leq \mu_{t^*}(\phi_T^2)-\alpha\int_T\phi_T^2d\mathcal H^n+K_2\cdot\left(\sup_{0\leq t\leq 2\Lambda+3}\int_{C_1}x_{n+1}^2\e|\nabla u(\cdot,t)|^2 dx\right)(2\Lambda+3)\\
&\leq \frac{K_2}{2}\left(\sup_{0\leq t\leq 2\Lambda+3}\int_{C_1}x_{n+1}^2\e|\nabla u(\cdot,t)|^2 dx\right)\\
&\leq -2 D_1 \left(\sup_{0\leq t\leq 2\Lambda+3}\int_{C_1}x_{n+1}^2\e|\nabla u(\cdot,t)|^2 dx\right).
\end{align*}
Therefore, we can apply Corollary \ref{BoundOnW}, and get for all $t\in [ t_*, 2 \Lambda+3]$, 
\begin{align*}
&\int_{C_1}\e\phi_T^2\left(\Delta u(\cdot,t)-\frac{W'(u(\cdot,t))}{\e^2}\right)^2 dx\\
& \geq 4 P \min\left\{1, \left(\mu_t(\phi_T^2)-\alpha\int_T\phi_T^2d\mathcal H^n\right)^{\frac{n-2}{n}}, \left(\sup_{0\leq t\leq 2\Lambda+3}\mathbb H_{u_{\cdot,t}}(C_1)\right)^{-\frac{2}{3}}\left(\mu_t(\phi_T^2)-\alpha\int_T\phi_T^2d\mathcal H^n\right)^{\frac{4}{3}} \right\}.
\end{align*}
Since we have 
\begin{align*} 
\left| \mu_{t^*}(\phi_T^2)-\alpha\int_T\phi_T^2d\mathcal H^n\right|\geq \sup_{0\leq t\leq 2\Lambda+3}\mathbb H_{u_{\cdot,t}}(C_1)=K\left(\sup_{0\leq t\leq 2\Lambda+3}\int_{C_1}x_{n+1}^2\e|\nabla u(\cdot,t)|^2 dx\right),
\end{align*}
we get 
\begin{align*}
e(t)&\leq \left| \mu_{t^*}(\phi_T^2)-\alpha\int_T\phi_T^2d\mathcal H^n\right|+K_2(2\Lambda+3) \left(\sup_{0\leq t\leq 2\Lambda+3}\int_{C_1}x_{n+1}^2\e|\nabla u(\cdot,t)|^2 dx\right)\\
&\leq 2 \left| \mu_{t^*}(\phi_T^2)-\alpha\int_T\phi_T^2d\mathcal H^n\right|. 
\end{align*}
This allows us to estimate
\begin{align*}
\int_{C_1}\e&\phi_T^2\left(\Delta u(\cdot,t)-\frac{W'(u(\cdot,t))}{\e^2}\right)dx\\ &\geq 4 P \min\left\{ 1, |e(t)|^{\frac{n-2}{n}}, \left(\sup_{0\leq t\leq 2\Lambda+3}\int_{C_1}x_{n+1}^2\e|\nabla u(\cdot,t)|^2 dx\right)^{-\frac{2}{3}}|e(t)|^{\frac{4}{3}}\right\}
\end{align*}
and therefore to apply Lemma 5.5 of \cite{Kasai2014} to get 
\begin{align*}
e(t_*+\Lambda) <-\alpha\omega_n<-\alpha\int_T\phi_T^2 d\mathcal H^n.
\end{align*}
But this contradicts the lower bound \eqref{eqn_eLowerBound}, which says that for $t^*+\Lambda\leq (\Lambda+2)+\Lambda\leq 2\Lambda+3$
\begin{align*}
e(t^*+\Lambda) \geq \left(-1+\frac{\nu}{2}\right)\alpha\int_T\phi_T^2d\mathcal H^n\geq -\alpha\int_T\phi_T^2 d\mathcal H^n.
\end{align*}
The contradiction argument for the second case \eqref{eqn_WillmoreBound} is similar to the first case.

The second conclusion on the Willmore term bound follows directly from \eqref{PushingEnergyDown} once we have the first conclusion on the energy difference.
\end{proof}
\section{Parabolic Lipschitz approximation}\label{sec_ParabolicLipschitz}
Through out this section, we assume all the conditions (\eqref{NonTrivialEnd}-\eqref{LayerRepulsion}) in the excess decay theorem, Theorem \ref{thm_ExcessDecay} hold. The argument here is the parabolic analogue of the Lipschitz approximation in Allard's regularity theory, the parabolic Lipschitz (with respect to the parabolic metric) refers to Lipschitz in space and $C^{0,\frac{1}{2}}$ in time. We use a maximal function argument for the space-time Lipschitz approximation. Consider the function $f^\e:\mathbb R^n\times\mathbb R\rightarrow\mathbb R$, given by
\begin{align*}
f^\e(\hat x,t)=\int_{-1}^1(1-\nu_{n+1}^2)\e|\nabla u^\e(\hat x,x_{n+1},t)|^2 dx_{n+1}.
\end{align*}
We define the Hardy-Littlewood parabolic maximal function 
\begin{align*}
Mf^\e(\hat x,t)&=\sup_{r\in(0,1)}r^{-n-2}\int_{t-r^2}^{t+r^2}\int_{B_r(\hat x)\cap\mathbb R^{n}}f^\e(\hat y,s) d\hat yds\\
&=\sup_{r\in(0,1)}r^{-n-2}\iint_{P_r(\hat x,t)\cap\mathbb R^{n}}f^\e(\hat y,s) d\hat yds.
\end{align*}
Then for any $l>0$, by the Hardy--Littlewood weak $L^1$ maximal theorem, there exists a $C$ such that
\begin{align*}
\mathcal {H}^{n+1}(\{Mf^\e\geq l\})&\leq \frac{C}{l}\int_{-1}^1\int_{B_1\cap\mathbb R^{n}}f^\e d\hat xdt\\
&\leq \int_{-1}^1\int_{-1}^1\int_{B_1\cap\mathbb R^n}(1-\nu_{n+1}^2)\e|\nabla u|^2 d\hat xdx_{n+1}dt\rightarrow0
\end{align*}
as $l\rightarrow\infty$ by Lemma \ref{lem_excess}, where $\mathcal{H}^k$ denotes the $k$-dimensional Lebesgue measure on the space-time $\mathbb R^n\times\mathbb R$.

For any fixed $l>0$ and $b\in(0,1)$, we partition the transition parts $\{|u|\geq 1-b\}$ into a set of good points (whose level sets are essentially Lipschitz graphs) and a set of bad points (of measure controlled by the tilt excess): 

We define the set of good points as
\begin{align*}
&A_{\e,l}
=\left \{(\hat x,x_{n+1},t)| \sup_{0<r<1}\frac{1}{r^{n+2}}\iint_{P_r(\hat x,x_{n+1},t)\cap\mathbb R^{n+1}}(1-\nu_{\e,n+1}^2)\e|\nabla u^\e(y,\tau)|^2 dyd\tau <l\right \}\cap\{|u|<1-b \}
\end{align*}
and the set of bad points as
\begin{align*}
B_{\e,l}=\left(P_1\cap\{|u|<1-b\}\right)\setminus A_{\e,l}.
\end{align*}

Note that $A_{\e,l}$ is an open set and $B_{\e,l}$ is relatively closed. Define the projections, $\Pi$, by
\begin{align*}
\Pi(A_{\e,l})&=\left\{(\hat x,t)\in\mathbb R^n\times\mathbb R|(\hat x,x_{n+1},t)\in A_{\e,l} \text{for some $x_{n+1}$}\right\},\\
\Pi(B_{\e,l})&=\left\{(\hat x,t)\in\mathbb R^n\times\mathbb R|(\hat x,x_{n+1},t)\in B_{\e,l} \text{for some $x_{n+1}$}\right\}.
\end{align*}
Notice the projection of the good and bad sets might not be disjoint. We will use a covering argument to estimate the measure of the bad set and its projections.
\begin{lemma}\label{lem_SetWeakL1} 
There exists a constant $C$ such that 
\begin{align*}
\iint_{B_{\e,l}}\e|\nabla u|^2 dxdt\leq \frac{C}{l}\iint_{P_2}(1-\nu_{n+1}^2)\e|\nabla u|^2 dxdt,
\end{align*}
and 
\begin{align*}
\mathcal {H}^{n+1}(\Pi (B_{\e,l}))\leq \frac{C}{l}\sup_{-1\leq t\leq 1}\iint_{P_2}(1-\nu_{n+1}^2)\e|\nabla u|^2 dxdt.
\end{align*}
\end{lemma}

\begin{proof}
For any $X=(\hat x,x_{n+1},t)\in B_{\e,l}$, there exists $r_{X}\in(0,1)$ such that
\begin{align}
r_{X}^{n+2}\leq \frac{1}{l}\iint_{P_r(\hat x,x_{n+1},t)}(1-(\nu_{\e}\cdot e_{n+1})^2)\e|\nabla u^\e(y,s)|^2 dyds.
\end{align}
The collection $\{P_{r(X)}(X)\}_{X\in B_{\e,l}}$ forms a covering of $B_{\e,l}$. By the Besicovitch covering lemma on space-time equipped with the parabolic metric, we have a finite family of countable subsets $\mathcal B_1, ... \mathcal B_{C_n}$, such that for each family $\mathcal B_{i}=\{P_{r(X_{i,j})}, i=1,...,C_n\}$ the parabolic cylinders $P_{r(X_{i,j})}$ within each of the families are mutually disjoint and the union of the families covers $B_{\e,l}$, that is 
\begin{align}
B_{\e,l}\subset\bigcup_{i}^{C_n}\bigcup_{j=1}^\infty P_{r(X_{i,j})}.
\end{align}
Using this cover, and using the estimate $\int_{B_{r}(x)}\left(\frac{\e|\nabla u^\e|^2}{2}+ \frac{W(u^\e)}{\e}\right) dx\leq Cr^n$ for any $(x,r)$)
\begin{equation}
\begin{split}
\iint_{(x,t)\in B_{\e,l}}&\left(\frac{\e|\nabla u^\e|^2}{2}+ \frac{W(u^\e)}{\e}\right) dxdt
\leq \sum_{i=1}^{C_n}\sum_{j}\iint_{P_{r(X_{i,j})}}\left(\frac{\e|\nabla u^\e|^2}{2}+ \frac{W(u^\e)}{\e}\right) dxdt\\
&\leq C\sum_{i=1}^{C_n}\sum_j r_{X_{i,j}}^{n+2}\\
&\leq C\sum_{i=1}^{C_n}\sum_j\frac{1}{l}\iint_{P_r(X_{i,j})}(1-(\nu_{\e}\cdot e_{n+1})^2)\e|\nabla u^\e(y,s)|^2 dyds\\
&\leq \frac{C_n}{l}\sup_{-1\leq t\leq 1}E^\e(u(\cdot,t))(B_2).
\end{split}
\end{equation}
To estimate the measure of the projection $\Pi(B_{\e,l})$, we have the following similar argument
\begin{align*}
\mathcal {H}^{n+1}(\Pi (B_{\e,l}))&=\mathcal {H}^{n+1}(\{(\hat x, t)|(\hat x,x_{n+1},t)\in B_{\e,l}  \text{for some $x_{n+1}$}\})\\
&\leq \frac{C_n}{l}\sup_{-1\leq t\leq 1}E^\e(u(\cdot,t))(B_2).
\end{align*}
\end{proof}

Next, for any $0<b<1$, we show in the set of good points (namely away from $B_{\e,l}$), the level sets $\{u^\e=s\}, s\in[-1+b,1-b]$ are essentially Lipschitz graphs with Lipschitz constants $C(b,\e,l)\rightarrow0$ as $l\rightarrow0$.
\begin{lemma}
For any $b\in(0,1)$ fixed and $l>0$ sufficiently small, the level sets $\{u^\e=s\}\cap A_{\e,l}$ can be locally represented by a Lipschitz graph $x_{n+1}=h^{\e,s}(\hat x,t)$ for $s\in(-1+b,1-b)$ where the Lipschitz constant of $h^b$ is bounded by $c(b,l)$ and tends to zero as $l\rightarrow0$, that is $ \lim_{l\rightarrow 0} c(b,l)=0$.
\end{lemma}
\begin{proof}

For $X_0=(x_0,t_0)$ such that $u^\e(X_0)=s$, we define a rescaled function 
\begin{align*}
v^\e(X)&=v^\e(x,t)=u^\e(x_0+\e x,t_0+\e^2t)
\end{align*}
where $v^\e$ satisfies the equation \eqref{eqn_ACF1}.
\begin{claim}
There exists a $c_2(b)>0$ independent of $\e$ such that
\begin{align*}
\left|\frac{\partial v^\e}{\partial x_{n+1}}\right|\geq c_2.
\end{align*}
\end{claim}
\begin{proof}[Proof of Claim]
We will give a proof by contradiction. Therefore suppose the conclusion is false, that is, there exists a sequence of solutions $v_i$ to the equation \eqref{eqn_ACF1} satisfying
\begin{align*}
|v_i(0,0)|&\leq 1-b,\\
\iint_{P_1(0,0)}\left(1-\left(\frac{\nabla v_i}{|\nabla v_i|}\cdot e_{n+1}\right)^2\right)|\nabla v_i|^2 dxdt&\leq l_i\rightarrow0,\\
\left|\frac{\partial v_i}{\partial x_{n+1}}\right|&\leq \frac{1}{i}\rightarrow0.
\end{align*}
However, by interior regularity of parabolic equations, $v_i\rightarrow v_\infty$ in $C^2_{loc}(\mathbb R^{n+1}\times\mathbb R)$ to the 1-d static solution $v_\infty(\hat x,x_{n+1},t)=g(x_{n+1})=\tanh(x_{n+1})$. So for $i$ sufficiently large, we must have $|\frac{\partial v_i}{\partial x_{n+1}}|\geq \frac{1}{2}g'(g^{-1}(b)):=c_2(b)$, which is a contradiction. \end{proof}
The lower bound of $|\frac{\partial v^\e}{\partial x_{n+1}}|$ gives an 
 upper bound for $|\frac{\partial h^s}{\partial s}|$. And combining with the relations between derivatives of $u$ and derivatives of the graphing functions $h^s$ \eqref{FunctionAndLevelSet}, we have the Lipschitz bound of $h^s$ 
\begin{align*}
\mathrm{Lip}(h^{\e,s})\leq \frac{1}{c_2(b)}\mathrm{Lip}(u^\e)\rightarrow0
\end{align*}
as $l\rightarrow0$.
\end{proof}

To show the level sets of $h^{\e,l}$ are globally Lipschitz graphs, we establish an $L^\infty$ bound. This estimate shows, when the tilt excess is small, the level sets are restricted in a narrow neighbourhood of the hyperplane.
\begin{lemma}\label{LInftyBound}
For any $b\in(0,1)$ and $\delta>0$, there exists $R_0>0$ large and $\tau_0,l_0$ sufficiently small such that the following holds: suppose $u^\e$ is a solution to \eqref{eqn_ACFe} in $P_{R_0}$, with $\e\leq 1, |u^\e(0,0)|\leq 1-b$  and satisfies the bounds 
\begin{align*}
R_0^{-n-2}\int_{B_{R_0}}\left(\frac{\e|\nabla u^\e|^2}{2}+ \frac{W(u^\e)}{\e}\right) dx&\leq (1+\tau_0)\alpha\omega_n,\\
R_0^{-n-2}\iint_{P_{R_0}}\left(1-\nu_{n+1}^2\right)\e |\nabla u^\e|^2 dxdt&\leq l_0.
\end{align*}
Then the level sets $\{u^\e=u^\e(0,0)\}\cap P_1$ are contained in a $\delta$ neighbourhood of the static plane $(\mathbb R^n\times\mathbb R)\cap P_1\subset\mathbb R^{n+1}\times\mathbb R$.
\end{lemma}
\begin{proof}
The proof is by contradiction. Therefore let us suppose the conclusion is false, that is for any $R_0>0$, there exists $b_0\in(0,1),\delta_0>0,\tau_i,l_i\rightarrow0,\e_i>0$ such that $u^{\e_i}$ is a solution of \eqref{eqn_ACFe} with $\e_i$ in place of $\e$, $|u^{\e_i}(0,0)|\leq 1-b_0$ and which satisfies
\begin{align*}
R_0^{-n-2}\int_{B_{R_0}}\left(\frac{\e_i|\nabla u^{\e_i}|^2}{2}+ \frac{W(u^{\e_i})}{\e_i}\right) dx&\leq (1+\tau_i)\alpha\omega_n,\\
R_0^{-n-2}\iint_{P_{R_0}}\left(1-\nu_{n+1}^2\right)\e_i |\nabla u^{\e_i}|^2 dxdt&\leq l_i\rightarrow0,
\end{align*}
but there exists $(\hat y_i,y_{i,n+1},t_i)$ with $|y_{i,n+1}|\geq \delta_0$ such that $u^{\e_i}(\hat y_i,y_{i,n+1},t_i)=u^{\e_i}(0,0,0)$.

There are 2 cases: \newline
\textbf{Case 1):}If $\e_i\rightarrow\e_\infty>0$, then up to subsequence $u^{\e_i}$ converges in $C^2_{loc}(P_{R_0-1})$ to a limit $u^\infty$ by standard parabolic estimates, and $u^\infty$ satisfies the equation \eqref{eqn_ACFe} with $\e_\infty$ in place of $\e$. And $(\hat y_i,y_{i,n+1},t_i)\rightarrow(\hat y_\infty,y_{\infty,n+1},t_\infty)$ after passing to a subsequence. In this case, we have
\begin{align}
\label{Linfty1}|u^\infty(0,0)|&\leq 1-b,\\
\label{Linfty2}u^\infty(\hat y_\infty,y_{\infty,n+1},t_{\infty})&=u^\infty(0,0), \quad \text{where $|y_{\infty,n+1}|\geq \delta_0$},\\
\label{Linfty3}R_0^{-n-2}\iint_{P_{R_0}}\left(\frac{\e_\infty|\nabla u^{\infty}|^2}{2}+ \frac{W(u^{\infty})}{\e_\infty}\right) dxdt&\leq 2\alpha\omega_n,\\
\label{Linfty4}R_0^{-n-2}\iint_{P_{R_0}}\left(1-\nu_{n+1}^2\right)\e_\infty |\nabla u^{\infty}|^2 dxdt&=0.
\end{align}
By equation \eqref{Linfty4}, the limit $u^\infty(\hat x,x_{n+1},t)$ only depends on $(x_{n+1},t)$, and by equations \eqref{Linfty1} and \eqref{Linfty3}, it satisfies $u^\infty(\hat x,x_{n+1},t)=g(\frac{x_{n+1}}{\e_\infty})$ where $g$ is the $1$-d solution defined in section \ref{sec_prelim} (this follows since in dimension $1$ the only possibilities are $g$ or constant solutions). For constant solutions we can choose $R_0$ sufficiently large to violate the energy bound condition \eqref{Linfty3}). 
\newline
\textbf{Case 2):} If $\e_i\rightarrow0$, then by the assumptions $\tau_i,l_i\rightarrow0$ and the main theorem of \cite{tonegawa2003integrality} the limit is a multiplicity $1$ static plane. Then by \cite{trumper2008relaxation}, $u^{\e_i}(\hat y_i,y_{i,n+1},t_i)\rightarrow\pm1$ because $|y_{i,n+1}|\geq \delta_0$. This is a contradiction to the assumption that $u^{\e_i}(\hat y_i,y_{i,n+1},t_i)=u^{\e_i}(0,0,0)\in(-1+b_0,1-b_0)$.
\end{proof}

\begin{remark}
Obtaining an $L^\infty$ bound for the level set graphing function is simpler in the Allen--Cahn setting than in the mean curvature flow situation, since we can use parabolic regularity for the underlying function $u$ and combining this with a lower bound on $e_{n+1}$ directional derivatives, \eqref{FunctionAndLevelSet} gives a bound on the graphing function $h$.
\end{remark}

Next we recall a Lemma from \cite{tonegawa2003integrality}, we only need part of the conclusions and will state the part we need.
\begin{lemma}[Lemma 4.6 of \cite{tonegawa2003integrality}]\label{OnlyOnePoint}
For any $b\in(0,1)$, there exists $\eta\in(0,1)$ and $L\in(0,\infty)$ with the following property:

If $u^\e$ ($\e\leq 1$) is a solution of \eqref{eqn_ACFe} in the parabolic ball $P_1\subset\mathbb R^{n+1}\times\mathbb R$ with $|u(0,0)|\leq 1-b$, and 
\begin{align*}
(4\e L)^{-n}\int_{B_{4\e L}(0)\times\{0\}}\left|\frac{\e|\nabla u^\e|^2}{2}-\frac{W(u^\e)}{\e}\right| dx&\leq \eta,\\
(4\e L)^{-n}\int_{B_{4\e L}(0)\times\{0\}}(1-\nu_{n+1}^2)\e|\nabla u^\e|^2 dx&\leq \eta.
\end{align*}
Then we have 
\begin{align*}
\{x_{n+1}| (\hat x,x_{n+1})\in B_{3\e L}(0), u(\hat{0},x_{n+1},0)=u(\hat{0},0,0)\}=\{0\}.
\end{align*}
\end{lemma}

Combining the Lemma \ref{OnlyOnePoint} and Lemma \ref{LInftyBound}, we see that the nodal sets are in fact global graphs by showing that there is essentially one sheet.

\begin{lemma}\label{LipGraphical}
Given any $b\in(0,1)$ and $\in(-1+b,1-b)$, we can choose $l$ sufficiently small such that, for any $(\hat x,t)\in\Pi(A_{\e,l})$, there exists exactly 1 point $(\hat x,h^{\e,s}(\hat x),t)$ in $\Pi^{-1}(x)\cap\{u^\e=s\}$ and where $h^{\e,l}$ is Lipschitz on $\Pi(A_{\e,l})$.
\end{lemma}
\begin{proof}
By Lemma \ref{LInftyBound}, for any $\delta$, the level sets of $u^\e$ with value $u^\e(\hat x,0,t)$ are restricted to a $\delta$ neighbourhood of the static plane $\mathbb R^{n+1}$ by choosing $l$ sufficiently small in $A_{\e,l}$. 

By choosing $\delta=\e L<<1$ where $L$ is given by Lemma \ref{OnlyOnePoint}, we have for any $s\in(-1+b,1-b)$ and $(\hat x,t)\in \Pi(A_{\e,l})$, there is at most one point in $\Pi^{-1}(\hat x,t)\cap\{u^\e=s\}$.

On the other hand, by the multiplicity 1 convergence and \cite{trumper2008relaxation}, for any $(\hat x,t)\in A_{\e,l}$, $\lim_{x_{n+1}\rightarrow+\infty}u^\e(\hat x,x_{n+1},t)=1$ and $\lim_{x_{n+1}\rightarrow-\infty}u^\e(\hat x,x_{n+1},t)=-1$. By the intermediate value theorem, there must be at least one $x_{n+1}$ such $u^\e(\hat x, x_{n+1},t)=s$. 

And thus the level $\{u=s\}$ is a global Lipschitz graph on the projection $\Pi(A_{\e,l})$.
\end{proof}

Finally, we can extend the Lipschitz graph from $\Pi(A_{\e,l})$ to all of $\Pi(P_1)$ preserving the Lipschitz constant by a standard Lipschitz extension lemma.

Before we proceed, we also gather some $L^2$ estimates of this Lipschitz graphical approximation, which will be used in the blow-up argument in the next section.

The first estimate is to bound the space-time $L^2$ norm of the gradient of graphing functions $h^{\e,s}$ by the tilt excess.
\begin{proposition}\label{TiltAndGradient}
For any fixed $b\in(0,1)$, there exists a $C(b)$ such that 
\begin{align*}
\int_{-1+b}^{1-b}\int_{-1}^1\int_{\hat B_1^n}|\hat\nabla h^{\e,s}|^2 d\hat xdtds\leq C(b)\int_{-4}^4\int_{B_2^{n+1}}(1-\nu_{n+1}^2)\e|\nabla u^\e|^2 dxdt.
\end{align*}
\end{proposition}
\begin{proof}
This is essentially the same argument using coarea formula as in the proof of the Interior Sobolev inequality Proposition \ref{InteriorSobolev}.

By Lemma \ref{LipGraphical}, we have extended the Lipschitz graph to all of $\hat B_1^n$ with Lipschitz constant bounded by  $C(b,l)$. Moreover, the measure of the set of bad points is bounded by $\mathcal H^{n+1}(\hat B_1^n\setminus\Pi(A_{\e,l}))\leq C(\int_{-4}^4\int_{B_2^{n+1}}(1-\nu_{n+1}^2)\e|\nabla u^\e|^2 dxdt)$. 

Next consider the bounds on the good sets $A_{\e,l}$:
\begin{align*}
&\int_{-4}^4\int_{A_{\e,l}\cap B_2^{n+1}}(1-\nu_{n+1}^2)\e|\nabla u^\e|^2 dxdt\\
&=\int_{-1}^1\int_{-4}^4\int_{B_2^{n+1}\cap\{u=s\}}(1-\nu_{n+1}^2)\e|\nabla u^\e| d\mathcal H^n\mres\{u=s\}dtds\\
&=\int_{-1}^1\int_{-4}^4\int_{\Pi(A_{\e,l})\cap\hat B_2^n}\left(1-\frac{1}{1+|\hat\nabla h^{\e,s}|^2}\right)\e|\nabla u^\e|\sqrt{1+|\hat\nabla h^{\e,s}|^2} d\hat xdtds\\
&\geq C(b)\int_{-1}^1\int_{-4}^4\int_{\Pi(A_{\e,l})\cap\hat B_2^n}\left(1-\frac{1}{1+|\hat\nabla h^{\e,s}|^2}\right)\sqrt{1+|\hat\nabla h^{\e,s}|^2} d\hat xdtds\\
&\geq \tilde C(b,l)\int_{-1}^1\int_{-4}^4\int_{\hat B_2^n}|\hat\nabla h^{\e,s}|^2 d\hat xdtds,
\end{align*}
where we used that in the good set $A_{\e,l}$, we have gradient lower bounds depending only on $b$ by parabolic regularity and compactness as argued before in the proof of interior Sobolev inequality Proposition \ref{InteriorSobolev}.
\end{proof}

We also have the following estimate as a corollary of Proposition 4.1 of \cite{tonegawa2003integrality}. This shows the excess is arbitrarily small on the set where the values of $u^\e$ are close to $\pm1$.
\begin{lemma}\label{LevelSetCloseTo1}
For any $\sigma>0$, there exists $b\in(0,1),\e_0>0$ sufficiently small such that
\begin{align*}
\iint_{\{|u^\e|\geq 1-b\}\cap P_1}\e|\nabla u|^2dxdt&\leq \sigma,
\end{align*}
for $\e\leq \e_0$.
\end{lemma}
\begin{proof}
The hypotheses allow us to apply \cite[Proposition 4.1]{tonegawa2003integrality}, and obtain the estimate
\begin{align*}
\int_{B_1(0)\times\{t\}\cap \{|u|\geq 1-b\}} \frac{W(u)}{\e}dx \leq \sigma.
\end{align*}
Applying the discrepancy inequality, \eqref{eqn_discrepancy} $ \frac{\e |\nabla u|^2}{2} \leq \frac{W(u)}{\e}$ and integrating the resulting estimate in time, we get
\begin{align*}
\iint_{\{|u^\e|\geq 1-b\}\cap P_1}\e|\nabla u|^2dxdt&\leq \sigma \leq \iint_{ \{|u|\geq 1-b\}\cap P_1} \frac{W(u)}{\e}dxdt \leq \sigma
\end{align*}
as desired.

\end{proof}

We will also need uniform estimates of $\frac{\partial h^{\e,s}}{\partial s}$ in a good set $E^\e$, together with control on the measure of the complement of this set.
\begin{lemma}\label{AUniformEstimate}
For a fixed $b\in(0,1)$, there exists a subset $E^\e\subset\Pi(A^\e)$ with 
\begin{align*}
\mathcal H^{n+1}(\Pi(A^\e)\setminus E^\e)\leq C\iint_{P_1}(1-\nu_{n+1}^2)\e|\nabla u|^2 dxdt
\end{align*}
such that the following holds : for any $s\in(-1+b,1-b)$ and $(x,t)\in\Pi^{-1}(E^\e)\cap A^\e\subset\mathbb R^{n+1}\times\mathbb R$ with $u^\e(x_0,t_0)=s$, we have
\begin{align*}
\e\left(\frac{\partial h^s}{\partial s}(x,t)\right)^{-1}=g'(g^{-1}(s))+o_\e(1).
\end{align*}
\end{lemma}
\begin{proof}
Let $E^\e=\Pi(A^\e)\cap\{Mf^\e\leq |\iint_{P_1}(1-\nu_{n+1}^2)\e|\nabla u|^2 dxdt|\}$. By properties of the Hardy-Littlewood maximal function, we have
\begin{align*}
\mathcal H^{n+1}(\Pi(A^\e)\setminus E^\e)&\leq C\iint_{\Pi(A^\e)}f^\e dxdt\\
&\leq C\iint_{P_1}(1-\nu_{n+1}^2)\e|\nabla u|^2 dxdt.
\end{align*}
In the set $E^\e$, we scale $u^\e$ parabolically and get $v^\e(x,t)=u^\e(x_0+\e x,t_0+\e^2t)$ where $v^\e$ satisfies the equation \eqref{eqn_ACF1}. By standard parabolic regularity, we have $v^\e\rightarrow v^\infty$ where $v^\e(x,t)=g(x_{n+1}+g^{-1}(s))$ in $C^1_{loc}$ as $\e\rightarrow0$ since the tilt excess converges to 0. Therefore
\begin{align*}
\e\left(\frac{\partial h^{\e,s}(\hat x_0,t_0)}{\partial s}\right)^{-1}=\e\frac{\partial u^\e(x_0,t_0)}{\partial x_{n+1}}=\frac{\partial v^\e(0,0)}{\partial x_{n+1}}=g'(g^{-1}(s))+o_\e(1).
\end{align*}
\end{proof}

\section{Approximation by heat equation and excess decay}\label{sec_ExcessDecay}

In this section we prove Theorem \ref{thm_ExcessDecay}, and thus we assume all the conditions \eqref{NonTrivialEnd}-\eqref{LayerRepulsion} hold in this section. To prove the theorem, we show the graphical functions $h^s$ (where $|s|\leq 1-b$) of the intermediate level sets are very close to heat equation as $\e\rightarrow0$ when rescaled by the space-time $L^2$ excess. The excess decay then follows from the decay property of linear heat equations. 

This technique is known as tilt-excess decay to prove H\"older regularity and goes back to De Giorgi \cite{DeGiorgi1961}. It has been used in various settings of elliptic and parabolic problems including minimal surfaces, mean curvature flow, elliptic Allen--Cahn equation, etc. \cite{allard1972, Kasai2014, wang2014new}. However, in the Allen--Cahn situation, the excess decay holds only up to $\e$ scale (due to the technical assumption \eqref{TechnicalAssumption}) and as a consequence extra steps are required to obtain H\"older regularity.

The proof of the excess decay theorem will be conducted at the end of this section, where we will assume on the contrary that:

After passing to a subsequence $\e_m$ of $\e$, there exists a sequence of solutions $u^{\e_m}(x,t)$ to the parabolic Allen--Cahn equation \eqref{eqn_ACFe} with $\e_m\rightarrow0$ in the parabolic cylinder $P_1\subset\mathbb R^{n+1}\times\mathbb R$. The energy density of this sequence in the parabolic cylinder satisfies
\begin{align}\label{SequenceDensityGoingTo1}
\mu^{\e_m}(P_1)\leq \alpha\omega_n(1+\tau_m), \tau_m\rightarrow0.
\end{align}
The height excess of this sequence in the parabolic cylinder satisfies
\begin{align}\label{SequenceHeightGoingTo0}
\iint_{P_1}x_{n+1}^2\e_m|\nabla u^{\e_m}|^2dxdt&\leq \frac{1}{m}.
\end{align}
And that for any $T\in G(n+1,n)$ with normal vector $e\in\mathbb S^{n}$ satisfying 
\begin{align}\label{NormalVectorConverging}
\|e-e_{n+1}\|\leq C\left (\iint_{P_1}x_{n+1}^2\e|\nabla u^{\e_m}|^2dxdt\right )^\frac{1}{2}
\end{align}
(here the uniform $C$ is determined by \eqref{PointwiseGradient}), the following inequality holds
\begin{align}\label{ExcessDecayNotHold}
\theta^{-n-4}\iint_{P_{\theta}}\mathrm{dist}(x,A)^2\e|\nabla u^{\e_m}|^2dxdt&>\frac{\theta}{2}\iint_{P_1}x_{n+1}^2\e|\nabla u^{\e_m}|^2dxdt,
\end{align}
where $A\in A(n,n-1)$ is parallel to $T$ and that $\|A-\{x_{n+1}=0\}\|\leq \frac{1}{m}$.

A contradiction will then be derived assuming the above conditions \eqref{SequenceDensityGoingTo1}-\eqref{SequenceHeightGoingTo0}.

\subsection{$L^2-L^\infty$ estimate}

Before considering blow-up analysis, we state the relations between the space-time $L^2$ norm of tilt excess, space-time $L^2$ norm of height excess and the energy difference of time slices to flat solutions and show they are comparable. 

Under the conditions \eqref{SequenceDensityGoingTo1}-\eqref{SequenceHeightGoingTo0}, we already know that the sequence of energy measures satisfies $d\mu^\e\rightarrow\alpha d\mu$ where $d\mu$ is a static mean curvature flow of flat planes, with unit density. Thus for small enough $\e$ in the sequence (for simplicity we dropped the subscript $m$ here), the conditions \eqref{ParabolicCaccioppoliHeightCondition}-\eqref{eqn_FinalMass} in Theorem \ref{ParabolicCaccioppoli} are all satisfied, and thus
\begin{align}\label{L2LInfty}
\sup_{-\frac{1}{2}\leq t\leq \frac{1}{2}}\left|\mu(\phi_T^2)-\alpha\int_T\phi_T^2d\mathcal H^n\right|&\leq K\sup_{-1\leq t\leq 1}\int_{C_1}x_{n+1}^2\e|\nabla u^\e(\cdot,t)|^2 dx,\\\label{L2LInfty2}
\int_{-\frac{1}{2}}^{\frac{1}{2}}\int_{C_1}\phi_T^2\e\left(\Delta u^\e-\frac{W'(u^\e)}{\e^2}\right)^2 dxdt&\leq 12K\sup_{-1\leq t\leq 1}\int_{C_1}x_{n+1}^2\e|\nabla u^\e(\cdot,t)|^2 dx,\\\label{L2LInfty3}
\int_{-\frac{1}{2}}^{\frac{1}{2}}\left|\mu(\phi_{T,\frac{1}{2}}^2)-\alpha\int_T\phi_{T,\frac{1}{2}}^2d\mathcal H^n\right| dt&\leq K\sup_{-1\leq t\leq 1}\int_{C_1}x_{n+1}^2\e|\nabla u^\e(\cdot,t)|^2 dx,\\\label{L2LInfty4}
\sup_{-1\leq t\leq 0} \int_{C_1}x_{n+1}^2\e|\nabla u^\e(\cdot,t)|^2 dx&\leq C\int_{-2}^0\int_{C_2}x_{n+1}^2\e|\nabla u^\e|^2 dxdt.
\end{align}

The inequalities \eqref{L2LInfty},\eqref{L2LInfty2} and \eqref{L2LInfty3} were proved in Theorem \ref{ParabolicCaccioppoli}.  The inequality \eqref{L2LInfty4} is \eqref{eqn_L2Linfty}, which is a $L^2$-$L^\infty$ type estimate (c.f. 6.4 of \cite{Kasai2014} and 4.26 of \cite{Ecker2004}). 

\begin{proposition}[{\cite[6.3]{wang2014new}}]\label{PoincareTypeInequality}
Let us assume the technical assumption \eqref{TechnicalAssumption}.Then there exists a $\lambda\in\mathbb R$ (to be chosen later in the subsection carrying out excess decay estimates), so that
\begin{align*}
\int_{C_1}(x_{n+1}-\lambda)^2\e|\nabla u^\e(\cdot,t)|^2 dx\leq C\int_{C_1}(1-\nu_{n+1}^2)\e|\nabla u|^2 dx.
\end{align*}
\end{proposition}

Combining \eqref{L2LInfty}-\eqref{L2LInfty4} and Proposition \ref{PoincareTypeInequality}, we have
\begin{proposition}\label{Comparable}
There exists $C_1,C_2,C_3,C_4>0$ depending only on then dimension $n$ such that
\begin{align*}
\mathbb H(P_1)\leq C_1 E(P_1)\leq C_2 \mathcal W(P_1)\leq C_3 \left|\mu(P_1)-2\alpha\omega_n\right|\leq C_4 \mathbb H(P_1).
\end{align*}
\end{proposition}
By this proposition, the height excess, tilt excess, Willmore term and the energy difference with the flat solution are comparable quantities and we do not distinguish them up to multiplication by universal constant.
\subsection{The blow-up analysis}

In this section we prove the excess decay up to scale $\e$. The idea is to show after appropriate scaling, the intermediate level sets 
\begin{align*}
\{u(\hat x,x_{n+1},t)=s\}=\{(\hat x,h^{s}(\hat x,t),t)\in\mathbb R^{n+1}\times\mathbb R\}
\end{align*}
for $|s|\leq 1-b$ are graphs of $h^s$ which are very well approximated by solutions to the heat equation.

We normalize the graphing function $h^{\e,s}$ and define
\begin{align}
\bar h^{\e,s}:=\frac{h^{\e,s}}{(\iint_{P_1}x_{n+1}^2\e|\nabla u^\e(\cdot,t)|^2 dxdt)^\frac{1}{2}}-\lambda^{\e,s}
\end{align}
where $\lambda^{\e,s}$ is chosen so that $\iint_{P_1}\bar h^{\e,s}=0$. Notice that in the Caccioppoli inequality Theorem \ref{Caccioppoli} in section 5, we can replace $x_{n+1}$ by $x_{n+1}-\lambda$ for arbitrary fixed constant $\lambda$.
\begin{theorem}\label{thm_L2Bound}
As $\e\rightarrow0$, up to passing to subsequences, the function defined above converges in $L^2$ as follows 
\begin{align*}
\bar h^{\e,s}\rightarrow\bar h^\infty \text{in $L^2\left(\hat B_\frac{1}{2}^n\times\left[-\tfrac{1}{2},\tfrac{1}{2}\right]\right)$}
\end{align*}
and $\bar h^\infty$ satisfies the heat equation $\frac{\partial \bar h^\infty}{\partial t}-\hat\Delta\bar h^\infty=0$.
\end{theorem}
\begin{proof}
First, by the estimates in Proposition \ref{TiltAndGradient}, the graphing functions $\bar h^{\e,s}$ are bounded in $W^{1,2}$. Thus by Rellich compactness theorem, we can extract a subsequence so that $\bar h^{\e,s}$ converges strongly in $L^2$ to $\bar h^\infty$. Moreover, by the construction of the graphical function $h$ (see the extension of definition of $h$ from $\Pi(A)$ to $\Pi(P_1)$ after Lemma \ref{LipGraphical}), we have
\begin{align*}
c(b)^{-1}\e\leq \frac{\partial h^{\e,s}}{\partial s}=\left(\frac{\partial u^\e}{\partial x_{n+1}}\right)^{-1}\leq c(b)\e.
\end{align*}
Thus
\begin{align*}
|h^{\e,s_1}-h^{\e,s_2}|\leq 2c(b)\e,
\end{align*}
and
\begin{align*}
\iint_{P_1}|h^{\e,s_1}-h^{\e,s_2}|^2 dxdt\leq C(b)\e^2=o\left(\iint_{P_1}(1-\nu_{n+1}^2)\e|\nabla u^\e|^2 dxdt\right),
\end{align*}
for $s_1,s_2\in(-1+b,1-b)$, where the last inequality follows because of the technical assumption \eqref{TechnicalAssumption}.

Since we can compare the height excess and the tilt excess by Proposition \ref{Comparable}, we have
\begin{align*}
\iint_{P_1}|\bar h^{\e,s_1}-\bar h^{\e,s_2}|^2 dxdt&=\frac{\iint_{P_1}|h^{\e,s_1}-h^{\e,s_2}|^2 dxdt}{(\iint_{P_1}x_{n+1}^2\e|\nabla u^\e(\cdot,t)|^2 dxdt)}\\
&=O\left(\frac{\iint_{P_1}|h^{\e,s_1}-h^{\e,s_2}|^2 dxdt}{\iint_{P_1}(1-\nu_{n+1}^2)\e|\nabla u^\e|^2 dxdt}\right)=o(1)\rightarrow0,
\end{align*}
and the limit $\bar h^\infty$ is independent of $s\in(-1+b,1-b)$. Next, let us choose the following cylindrical test function
\begin{equation*}
\eta(\hat x,x_{n+1},t)=\phi(\hat x,t)\psi(x_{n+1})(x_{n+1}-\lambda),
\end{equation*}
where $\phi\in C_0^\infty(\mathbb R^{n}\times\mathbb R)$ is a space--time test function that depends only on the first $n$ spatial variables $(x_1,...,x_n)$ and time $t$, and $\psi\in C_0^\infty(-1,1)$, with $\psi(\xi)\equiv1 \text{ for $\xi\in(-\frac{1}{2},\frac{1}{2})$}$. We calculate
\begin{align}\label{DerivativeEta}
\nabla\eta&=\hat\nabla\phi\psi\cdot(x_{n+1}-\lambda)+\phi\psi'\cdot(x_{n+1}-\lambda)\vec{e}_{n+1}+ \phi\psi\vec{e}_{n+1},\\\nonumber
\nabla^2\eta&=\hat\nabla^2\phi\psi\cdot(x_{n+1}-\lambda)+2\psi'\cdot(x_{n+1}-\lambda)\hat\nabla\phi\otimes\vec{e}_{n+1}+ 2\psi\hat\nabla\phi\otimes\vec{e}_{n+1},\\\nonumber
&+\phi\psi''(x_{n+1}-\lambda)\vec{e}_{n+1}\otimes\vec{e}_{n+1}+ 2\phi\psi'\vec{e}_{n+1}\otimes\vec{e}_{n+1},\\\nonumber
\Delta\eta&=\hat\Delta\phi\psi(x_{n+1}-\lambda)+\phi\psi''+2\phi\psi'\\\nonumber
\partial_t\eta&=\partial_t\phi\psi\cdot(x_{n+1}-\lambda).
\end{align}
Substituting our test function $\eta$ into the the time dependent integral form of the $\e$-Brakke identity \eqref{eqn_BrakkeAllenCahn2}, we get
\begin{align*}
&\int\eta d\mu^\e_{t_2}-\int\eta d\mu^\e_{t_1}\\
&=\int_{t_1}^{t_2}\int-\e\eta\left(\Delta u-\frac{W'(u)}{\e^2}\right)^2 dxdt-\int_{t_1}^{t_2}\int\e\langle\nabla\eta,\nabla u\rangle\left(\Delta u^\e-\frac{W'(u^\e)}{\e^2}\right) dxdt+\int_{t_1}^{t_2}\int\partial_t\eta d\mu^\e_t\\
&=\int_{t_1}^{t_2}\int-\e\eta\left(\Delta u-\frac{W'(u)}{\e^2}\right)^2 dxdt+\int_{t_1}^{t_2}\int\left(\e\nabla u\otimes\nabla u-\left(\frac{\e|\nabla u|^2}{2}+ \frac{\e W(u)}{2}\right) I\right):\nabla^2\eta dxdt\\
&+\int_{t_1}^{t_2}\int\partial_t\eta d\mu^\e_t,\\
&=\int_{t_1}^{t_2}\int-\e\eta\left(\Delta u-\frac{W'(u)}{\e^2}\right)^2 dxdt+\int_{t_1}^{t_2}\int\e(\nabla u\otimes\nabla u):\nabla^2\eta  dxdt\\
&-\int_{t_1}^{t_2}\int\left(\frac{\e|\nabla u|^2}{2}+ \frac{W(u)}{\e}\right)\Delta\eta dxdt
+\int_{t_1}^{t_2}\int\partial_t\eta d\mu^\e_t.
\end{align*}
And thus by moving the Willmore term $\left(\int_{t_1}^{t_2}\int\e\eta\left(\Delta u-\frac{W'(u)}{\e^2}\right)^2 dxdt\right)$ to the left hand side and expand the derivative terms using computations in \eqref{DerivativeEta}, we have
\begin{align}\label{EpsilonHeat}
&\int\eta d\mu^\e_{t_2}-\int\eta d\mu^\e_{t_1}+\int_{t_1}^{t_2}\int\e\eta\left(\Delta u-\frac{W'(u)}{\e^2}\right)^2 dxdt\\\nonumber
&=\int_{t_1}^{t_2}\int\e\sum_{i,j=1}^n\frac{\partial^2\phi}{\partial x_{i}\partial x_{j}}\frac{\partial u}{\partial x_i}\frac{\partial u}{\partial x_j}\psi(x_{n+1}-\lambda) dxdt+2\int_{t_1}^{t_2}\int\e\sum_{i=1}^n\frac{\partial\phi}{\partial x_i}\frac{\partial u}{\partial x_i}\frac{\partial u}{\partial x_{n+1}}\psi'(x_{n+1}-\lambda) dxdt\\\nonumber
&+2\int_{t_1}^{t_2}\int\e\sum_{i=1}^n\frac{\partial\phi}{\partial x_i}\frac{\partial u}{\partial x_i}\frac{\partial u}{\partial x_{n+1}}\psi dxdt+\int_{t_1}^{t_2}\int\phi\psi''(x_{n+1}-\lambda)\e\left|\frac{\partial u}{\partial x_{n+1}}\right|^2 dxdt\\\nonumber
&+2\int_{t_1}^{t_2}\int\phi\psi'\e\left|\frac{\partial u}{\partial x_{n+1}}\right|^2 dxdt-\int_{t_1}^{t_2}\int(\phi\psi''(x_{n+1}-\lambda)+2\phi\psi')\left(\frac{\e|\nabla u|^2}{2}+ \frac{W(u)}{\e}\right) dxdt\\\nonumber
&-\int_{t_1}^{t_2}\int\hat\Delta\phi\psi\cdot(x_{n+1}-\lambda)\left(\frac{\e|\nabla u|^2}{2}+ \frac{W(u)}{\e}\right) dxdt\\\nonumber
&+\int_{t_1}^{t_2}\int\partial_t\phi\psi\cdot(x_{n+1}-\lambda)\left(\frac{\e|\nabla u|^2}{2}+ \frac{W(u)}{\e}\right) dxdt\\\nonumber
&=\int_{t_1}^{t_2}\int\e\sum_{i,j=1}^n\frac{\partial^2\phi}{\partial x_{i}\partial x_{j}}\frac{\partial u}{\partial x_i}\frac{\partial u}{\partial x_j}\psi(x_{n+1}-\lambda) dxdt+2\int_{t_1}^{t_2}\int\e\sum_{i=1}^n\frac{\partial\phi}{\partial x_i}\frac{\partial u}{\partial x_i}\frac{\partial u}{\partial x_{n+1}}\psi'(x_{n+1}-\lambda) dxdt\\\nonumber
&+\int_{t_1}^{t_2}\int\phi\psi''(x_{n+1}-\lambda)\e\left|\frac{\partial u}{\partial x_{n+1}}\right|^2 dxdt+2\int_{t_1}^{t_2}\int\phi\psi'\e\left|\frac{\partial u}{\partial x_{n+1}}\right|^2 dxdt\\\nonumber
&-\int_{t_1}^{t_2}\int(\phi\psi''(x_{n+1}-\lambda)+2\phi\psi')\left(\frac{\e|\nabla u|^2}{2}+ \frac{W(u)}{\e}\right) dxdt\\\nonumber
&+\int_{t_1}^{t_2}\int2\e\sum_{i=1}^n\frac{\partial\phi}{\partial x_i}\frac{\partial u}{\partial x_i}\frac{\partial u}{\partial x_{n+1}}\psi dxdt-\int_{t_1}^{t_2}\int\hat\Delta\phi\psi\cdot(x_{n+1}-\lambda)\left(\frac{\e|\nabla u|^2}{2}+ \frac{W(u)}{\e}\right) dxdt\\\nonumber
&+\int_{t_1}^{t_2}\int\partial_t\phi\psi\cdot(x_{n+1}-\lambda)\left(\frac{\e|\nabla u|^2}{2}+ \frac{W(u)}{\e}\right) dxdt.
\end{align}
Now we estimate the terms on the right hand side of \eqref{EpsilonHeat}, first using the estimates in the previous subsection (tilt excess, height excess, Willmore term are essentially of the same order). We show the the first five terms  are small compared to the excess, which will still go to zero after dividing by the height excess. (as we will see in the following claims). 

\begin{claim}
\begin{align*}
&\left|\int_{t_1}^{t_2}\int\e\sum_{i,j=1}^n\frac{\partial^2\phi}{\partial x_{i}\partial x_{j}}\frac{\partial u}{\partial x_i}\frac{\partial u}{\partial x_j}\psi(x_{n+1}-\lambda) dxdt\right|\\
&+2\left|\int_{t_1}^{t_2}\int\e\sum_{i=1}^n\frac{\partial\phi}{\partial x_i}\frac{\partial u}{\partial x_i}\frac{\partial u}{\partial x_{n+1}}\psi'(x_{n+1}-\lambda) dxdt\right|\\
&+\left|\int_{t_1}^{t_2}\int\phi\psi''(x_{n+1}-\lambda)\e\left|\frac{\partial u}{\partial x_{n+1}}\right|^2 dxdt+2\int_{t_1}^{t_2}\int\phi\psi'\e\left|\frac{\partial u}{\partial x_{n+1}}\right|^2 dxdt\right|\\
&+\left|\int_{t_1}^{t_2}\int(\phi\psi''(x_{n+1}-\lambda)+2\phi\psi')\left(\frac{\e|\nabla u|^2}{2}+ \frac{W(u)}{\e}\right) dxdt\right|\\
&\leq O\left(\iint_{P_2}x_{n+1}^2\e|\nabla u^\e|^2dxdt \right).
\end{align*}
\end{claim}
\begin{proof}[Proof of Claim]
Since we know that the space-time $L^2$ norm of height excess and tilt excess are comparable and thus we will not distinguish them up to $O\left(\left(\iint_{P_2}x_{n+1}^2\e|\nabla u^\e|^2dx\right)^\frac{1}{2}\right)$. We have
\begin{align*}
&\left|\int_{t_1}^{t_2}\int\e\sum_{i,j=1}^n\frac{\partial^2\phi}{\partial x_{i}\partial x_{j}}\frac{\partial u}{\partial x_i}\frac{\partial u}{\partial x_j}\psi(x_{n+1}-\lambda) dxdt\right|\\
&\leq \left|\int_{t_1}^{t_2}\int C(n)|\nabla^2\phi|(1-\nu_{n+1}^2)\psi(x_{n+1}-\lambda)\e|\nabla u|^2 dxdt\right|\\
&\leq C(n,\phi)\left(\iint(1-\nu_{n+1}^2)^2\e|\nabla u|^2 dxdt\right)^\frac{1}{2}\left(\iint(x_{n+1}-\lambda)^2\e|\nabla u|^2 dxdt\right)^\frac{1}{2}\\
&\leq C(n,\phi)\left(\iint(1-\nu_{n+1}^2)\e|\nabla u|^2 dxdt\right)^\frac{1}{2}\left(\iint(x_{n+1}-\lambda)^2\e|\nabla u|^2 dxdt\right)^\frac{1}{2}\\
&=O\left(\iint_{P_2}x_{n+1}^2\e|\nabla u^\e|^2dxdt \right).
\end{align*}
For the other 3 terms, since $\psi\equiv1$ in $(-\frac{1}{2},\frac{1}{2})$, we have
\begin{align*}
&2\left|\int_{t_1}^{t_2}\int\e\sum_{i=1}^n\frac{\partial\phi}{\partial x_i}\frac{\partial u}{\partial x_i}\frac{\partial u}{\partial x_{n+1}}\psi'(x_{n+1}-\lambda) dxdt\right|\\
&+\left|\int_{t_1}^{t_2}\int\phi\psi''(x_{n+1}-\lambda)\e\left|\frac{\partial u}{\partial x_{n+1}}\right|^2 dxdt+2\int_{t_1}^{t_2}\int\phi\psi'\e\left|\frac{\partial u}{\partial x_{n+1}}\right|^2 dxdt\right|\\
&+\left|\int_{t_1}^{t_2}\int(\phi\psi''(x_{n+1}-\lambda)+2\phi\psi')\left(\frac{\e|\nabla u|^2}{2}+ \frac{W(u)}{\e}\right) dxdt\right|\\
&=2\left|\int_{t_1}^{t_2}\int_{|x_{n+1}|\geq \frac{1}{2}}\e\sum_{i=1}^n\frac{\partial\phi}{\partial x_i}\frac{\partial u}{\partial x_i}\frac{\partial u}{\partial x_{n+1}}\psi'(x_{n+1}-\lambda) dxdt\right|\\
&+\left|\int_{t_1}^{t_2}\int_{|x_{n+1}|\geq \frac{1}{2}}\phi\psi''(x_{n+1}-\lambda)\e\left|\frac{\partial u}{\partial x_{n+1}}\right|^2 dxdt+2\int_{t_1}^{t_2}\int\phi\psi'\e\left|\frac{\partial u}{\partial x_{n+1}}\right|^2 dxdt\right|\\
&+\left|\int_{t_1}^{t_2}\int_{|x_{n+1}|\geq \frac{1}{2}}(\phi\psi''(x_{n+1}-\lambda)+2\phi\psi')\left(\frac{\e|\nabla u|^2}{2}+ \frac{W(u)}{\e}\right) dxdt\right|\\
&\leq e^{-\frac{C}{\e^2}}=o(\e^2)\\
&=o\left(\iint_{P_2}x_{n+1}^2\e|\nabla u^\e|^2dxdt \right),
\end{align*}
where the second last inequality is due the exponential decay away from the transition region from Lemma \ref{ExponentialDecay} and the last bound by the height excess follows from \eqref{TechnicalAssumption}.
\end{proof}

For the last three terms of the right hand side of \eqref{EpsilonHeat}, we have
\begin{claim}
For a fixed $b\in(0,1)$, we have
\begin{align*}
&\iint_{\{|u|\leq 1-b\}\cap P_1}2\e\sum_{i=1}^n\frac{\partial\phi}{\partial x_i}\frac{\partial u}{\partial x_i}\frac{\partial u}{\partial x_{n+1}}\psi dxdt \\
&+ \iint_{\{|u|\leq 1-b\}\cap P_1}(\partial_t\phi-\hat\Delta\phi)(x_{n+1}-\lambda)\psi\left(\frac{\e|\nabla u|^2}{2}+ \frac{W(u)}{\e}\right) dxdt\\
&=\int_{-1+b}^{1-b}\iint_{\Pi(\hat P_1)}\left((\partial_t\phi-\hat\Delta\phi)(h^{\e,s}-\lambda)-2\langle\hat\nabla\phi,\hat\nabla (h^s-\lambda)\rangle \right)\psi g'(g^{-1}(s)) d\hat xdtds\\
&+O\left(\iint_{P_1}(x_{n+1}-\lambda)^2\e|\nabla u^\e|^2dx\right).
\end{align*}

Furthermore there exists $b\in(0,1)$ such that
\begin{align*}
&\iint_{\{|u|\geq 1-b\}\cap P_1}2\e\sum_{i=1}^n\frac{\partial\phi}{\partial x_i}\frac{\partial u}{\partial x_i}\frac{\partial u}{\partial x_{n+1}}\psi-\hat\Delta\phi\psi(h^s-\lambda) d\mu^\e+\partial_t\phi\psi(h^s-\lambda) d\mu^\e\\
&\leq o_b(1)O\left (\left(\iint_{P_2}(x_{n+1}-\lambda)^2\e|\nabla u^\e|^2dx\right)^\frac{1}{2}\right ),
\end{align*}
where $o_b(1)\rightarrow 0$ as $b\rightarrow0$.
\end{claim}
\begin{proof}[Proof of Claim]
In order to prove these estimates, we write our integrals using the graphical representations of the level sets given by $h$. Hence for a fixed $b\in(0,1)$, by \eqref{FunctionAndLevelSet}, the coarea formula and integration by parts,
\begin{align*}
&\iint_{\{|u|\leq 1-b\}\cap P_1}2\e\sum_{i=1}^n\frac{\partial\phi}{\partial x_i}\frac{\partial u}{\partial x_i}\frac{\partial u}{\partial x_{n+1}}\psi dxdt\\
&=\iint_{\{|u|\leq 1-b\}\cap P_1}2\e\langle\hat\nabla\phi,\hat\nabla u\rangle\frac{\partial u}{\partial x_{n+1}}\psi dxdt\\
&=\int_{-1+b}^{1-b}\iint_{P_1\cap\{u=s\}}2\e\langle\hat\nabla\phi,\hat\nabla u\rangle\frac{\partial u}{\partial x_{n+1}}\frac{1}{|\nabla u|}\psi d\mathcal H^n\llcorner\{u=s\}dtds\\
&=-\int_{-1+b}^{1-b}\iint_{P_1\cap\{u=s\}}2\e\langle\hat\nabla\phi,\hat\nabla h\rangle\left(\frac{\partial h^s}{\partial s}\right)^{-1}\left(\frac{\partial h^s}{\partial s}\right)^{-1}\frac{1}{\sqrt{1+|\hat\nabla h|^2}}\left(\frac{\partial h^s}{\partial s}\right)\psi d\mathcal H^n\llcorner\{u=s\}dtds\\
&\text{(By \eqref{FunctionAndLevelSet})}\\
&=-\int_{-1+b}^{1-b}\iint_{P_1\cap\{u=s\}}2\e\langle\hat\nabla\phi,\hat\nabla h^s\rangle\left(\frac{\partial h^s}{\partial s}\right)^{-1}\frac{1}{\sqrt{1+|\hat\nabla h|^2}}\psi d\mathcal H^n\llcorner\{u=s\}dtds\\
&=-\int_{-1+b}^{1-b}\iint_{\Pi(P_1\cap\{u=s\})}2\e\langle\hat\nabla\phi,\hat\nabla h^s\rangle\left(\frac{\partial h^s}{\partial s}\right)^{-1}\psi d\hat xdtds\\
&=-\int_{-1+b}^{1-b}\iint_{E^\e\subset \Pi(\hat P_1\cap\{u=s\})}2\e\langle\hat\nabla\phi,\hat\nabla h^s\rangle\left(\frac{\partial h^s}{\partial s}\right)^{-1}\psi d\hat xdtds+o\left(\iint_{P_1}(1-\nu_{n+1}^2)\e|\nabla u|^2 dxdt\right)\\
&\text{(By Lemma \ref{AUniformEstimate})}\\
&=-\int_{-1+b}^{1-b}\iint_{E^\e\subset \Pi(\hat P_1\cap\{u=s\})}2\langle\hat\nabla\phi,\hat\nabla h^s\rangle g'(g^{-1}(s))\psi d\hat xdtds+o\left(\iint_{P_1}(1-\nu_{n+1}^2)\e|\nabla u|^2 dxdt\right)\\
&=-\int_{-1+b}^{1-b}\iint_{\Pi(\hat P_1)}2\langle\hat\nabla\phi,\hat\nabla h^s\rangle g'(g^{-1}(s))\psi d\hat xdtds+o\left(\iint_{P_1}(1-\nu_{n+1}^2)\e|\nabla u|^2 dxdt\right).
\end{align*}
Using Proposition \ref{Comparable}, this then gives the following expansion
\begin{align}
\label{eqn_TermOne}&\iint_{\{|u|\leq 1-b\}\cap P_1}2\e\sum_{i=1}^n\frac{\partial\phi}{\partial x_i}\frac{\partial u}{\partial x_i}\frac{\partial u}{\partial x_{n+1}}\psi dxdt \\\nonumber
&=-\int_{-1+b}^{1-b}\iint_{\Pi(\hat P_1)}2\langle\hat\nabla\phi,\hat\nabla h^s\rangle g'(g^{-1}(s))\psi d\hat xdtds+O\left(\iint_{P_1}(x_{n+1}-\lambda)^2\e|\nabla u|^2 dxdt\right).
\end{align}
For the second term in the claim, by the Caccioppoli inequality Theorem \ref{Caccioppoli} and Proposition \ref{Comparable}, we can estimate the potential term 
\begin{align*}
&\iint_{\{|u|\leq 1-b\}\cap P_1}(\partial_t\phi-\hat\Delta\phi)(x_{n+1}-\lambda)\psi\left(\frac{\e|\nabla u|^2}{2}+ \frac{W(u)}{\e}\right) dxdt\\
&=\iint_{\{|u|\leq 1-b\}\cap P_1}(\partial_t\phi-\hat\Delta\phi)(x_{n+1}-\lambda)\psi\e|\nabla u|^2 dxdt+o(|\xi|(P_1))\\
&=\iint_{\{|u|\leq 1-b\}\cap P_1}(\partial_t\phi-\hat\Delta\phi)(x_{n+1}-\lambda)\psi\e|\nabla u|^2 dxdt+o(E(P_1)).
\end{align*}
We then apply the co-area formula and write the resulting integral above in terms of the graphing functions $h^s$
\begin{align*}
&\iint_{\{|u|\leq 1-b\}\cap P_1}(\partial_t\phi-\hat\Delta\phi)(x_{n+1}-\lambda)\psi\e|\nabla u|^2 dxdt \\
&=\int_{-1+b}^{1-b}\iint_{\{u=s\}\cap P_1}(\partial_t\phi-\hat\Delta\phi)(x_{n+1}-\lambda)\psi\e|\nabla u| d\mathcal H^n\llcorner\{u=s\}dtds \\
&=\int_{-1+b}^{1-b}\iint_{\{u=s\}\cap P_1}(\partial_t\phi-\hat\Delta\phi)(x_{n+1}-\lambda)\psi\e\sqrt{1+|\hat\nabla h|^2}\left(\frac{\partial h^s}{\partial s}\right)^{-1} d\mathcal H^n\llcorner\{u=s\}dtds   +o(E(P_1))\\
&=\int_{-1+b}^{1-b}\iint_{\Pi(P_1)}(\partial_t\phi-\hat\Delta\phi)(x_{n+1}-\lambda)\psi\e(1+|\hat\nabla h|^2)\left(\frac{\partial h^s}{\partial s}\right)^{-1} d\hat{x}dtds   +o(E(P_1))\\
&=\int_{-1+b}^{1-b}\iint_{\Pi(P_1)}(\partial_t\phi-\hat\Delta\phi)(x_{n+1}-\lambda)\psi\e\left(\frac{\partial h^s}{\partial s}\right)^{-1} d\hat xdtds   +O(E(P_1))
\end{align*}
where in the last line we used the bound in Lemma \ref{TiltAndGradient}. By Lemma \ref{AUniformEstimate}
\begin{align*}
&\int_{-1+b}^{1-b}\iint_{\Pi(P_1)}(\partial_t\phi-\hat\Delta\phi)(x_{n+1}-\lambda)\psi\e\left(\frac{\partial h^s}{\partial s}\right)^{-1} d\hat xdtds\\
&=\int_{-1+b}^{1-b}\iint_{E^\e}(\partial_t\phi-\hat\Delta\phi)(x_{n+1}-\lambda)\psi g'(g^{-1}(s)) d\hat xdtds +O(\mathcal H^{n+1}(\Pi(P_1)\setminus E^\e))\\
&=\int_{-1+b}^{1-b}\iint_{E^\e}(\partial_t\phi-\hat\Delta\phi)(x_{n+1}-\lambda)\psi g'(g^{-1}(s)) d\hat xdtds+O(E(P_1))\\
&=\int_{-1+b}^{1-b}\iint_{\hat P_1}(\partial_t\phi-\hat\Delta\phi)(x_{n+1}-\lambda)\psi g'(g^{-1}(s)) d\hat xdtds+O(E(P_1)).
\end{align*}
Therefore we get, 
\begin{align}
\label{eqn_TermTwo}&\iint_{\{|u|\leq 1-b\}\cap P_1}(\partial_t\phi-\hat\Delta\phi)(x_{n+1}-\lambda)\psi\left(\frac{\e|\nabla u|^2}{2}+ \frac{W(u)}{\e}\right) dxdt\\ \nonumber 
&=\int_{-1+b}^{1-b}\iint_{\hat P_1}(\partial_t\phi-\hat\Delta\phi)(x_{n+1}-\lambda)\psi g'(g^{-1}(s)) d\hat xdtds +O\left(\iint_{P_1}(x_{n+1}-\lambda)^2\e|\nabla u^\e|^2dxdt\right)
\end{align}
Putting together \eqref{eqn_TermOne} and \eqref{eqn_TermTwo}, we have
\begin{align*}
&\iint_{\{|u|\leq 1-b\}\cap P_1}2\e\sum_{i=1}^n\frac{\partial\phi}{\partial x_i}\frac{\partial u}{\partial x_i}\frac{\partial u}{\partial x_{n+1}}\psi dxdt\\&+\iint_{\{|u|\leq 1-b\}\cap P_1}(\partial_t\phi-\hat\Delta\phi)\psi\cdot(x_{n+1}-\lambda)\left(\frac{\e|\nabla u|^2}{2}+ \frac{W(u)}{\e}\right) dxdt\\
&=\int_{-1+b}^{1-b}\iint_{\Pi(\hat P_1)}\left((\partial_t\phi-\hat\Delta\phi)(h^{\e,s}-\lambda)-2\langle\hat\nabla\phi,\hat\nabla (h^s-\lambda)\rangle \right)\psi g'(g^{-1}(s)) d\hat xdtds\\
&+O\left(\iint_{P_1}(1-\nu_{n+1}^2)\e|\nabla u^\e|^2dxdt\right)+O\left(\iint_{P_1}\left|\frac{\e|\nabla u^\e|^2}{2}-\frac{W(u^\e)}{\e}\right|dxdt\right)\\
&=\int_{-1+b}^{1-b}\iint_{\Pi(\hat P_1)}\left((\partial_t\phi-\hat\Delta\phi)(h^{\e,s}-\lambda)-2\langle\hat\nabla\phi,\hat\nabla (h^s-\lambda)\rangle \right)\psi g'(g^{-1}(s)) d\hat xdtds\\
&+O\left(\iint_{P_1}(x_{n+1}-\lambda)^2\e|\nabla u^\e|^2dxdt\right)
\end{align*}
where the last line follows by Proposition \ref{Comparable}. This gives the first part of the claim for the set$\{|u|\leq 1-b\}$.

For the second part of the claim, we estimate as follows
\begin{align*}
&\iint_{\{|u|\geq 1-b\}\cap P_1}2\e\sum_{i=1}^n\frac{\partial\phi}{\partial x_i}\frac{\partial u}{\partial x_i}\frac{\partial u}{\partial x_{n+1}}\psi-\hat\Delta\phi\psi(x_{n+1}-\lambda) d\mu^\e+\partial_t\phi\psi(x_{n+1}-\lambda) d\mu^\e\\
&\leq C(\phi,\psi)\iint_{\{|u|\geq 1-b\}\cap P_1}\sqrt{1-\nu_{n+1}^2}\e|\nabla u^\e|^2 dxdt+C(\phi,\psi)\iint_{\{|u|\geq 1-b\}\cap P_1}(x_{n+1}-\lambda)\e|\nabla u^\e|^2 dxdt\\
&\leq C(\phi,\psi)\left(\iint_{|u|\geq 1-b}(1-\nu_{n+1}^2)\e|\nabla u^\e|^2 dxdt\right)^\frac{1}{2}\left(\iint_{|u|\geq 1-b}\e|\nabla u^\e|^2 dxdt\right)^\frac{1}{2}\\
&+C(\phi,\psi)\left(\iint_{|u|\geq 1-b}(x_{n+1}-\lambda)^2\e|\nabla u^\e|^2 dxdt\right)^\frac{1}{2}\left(\iint_{|u|\geq 1-b}\e|\nabla u^\e|^2 dxdt\right)^\frac{1}{2}\\
&\leq O\left(\iint_{|u|\geq 1-b}(x_{n+1}-\lambda)^2\e|\nabla u^\e|^2 dxdt\right)^\frac{1}{2}o_b(1)
\end{align*}
\text{(by Lemma \ref{LevelSetCloseTo1} and Proposition \ref{Comparable}}).
\end{proof}

Now by the above two claims, we have
\begin{align*}
&\int\eta\left(\frac{\e|\nabla u(\cdot, t_2)|^2}{2}+ \frac{W(u(\cdot,t_2))}{\e}\right) dx-\int\eta\left(\frac{\e|\nabla u(\cdot, t_1)|^2}{2}+ \frac{W(u(\cdot,t_1))}{\e}\right) dx\\
&+\int_{t_1}^{t_2}\int\e\eta\left(\Delta u-\frac{W'(u)}{\e^2}\right)^2 dxdt\\
&=\int_{-1+b}^{1-b}\iint_{\Pi(\hat P_1)}\left((\partial_t\phi-\hat\Delta\phi)(h^{\e,s}-\lambda)-2\langle\hat\nabla\phi,\hat\nabla (h^s-\lambda)\rangle \right)\psi g'(g^{-1}(s)) d\hat xdtds\\
&+O\left(\iint_{P_1}(x_{n+1}-\lambda)^2\e|\nabla u^\e|^2dx\right)+o_b(1)O\left(\iint_{P_1}(x_{n+1}-\lambda)^2\e|\nabla u^\e|^2dx\right)^\frac{1}{2}.
\end{align*}
We divide both sides by $\left(\iint_{P_1}(x_{n+1}-\lambda)^2\e|\nabla u^\e|^2dx\right)^\frac{1}{2}$ and use Proposition \ref{Comparable}, we have
\begin{align*}
&\frac{\int_{-1+b}^{1-b}\iint_{\Pi(\hat P_1)}\left((\partial_t\phi-\hat\Delta\phi)(h^{\e,s}-\lambda)-2\langle\hat\nabla\phi,\hat\nabla (h^s-\lambda)\rangle \right)\psi g'(g^{-1}(s)) d\hat xdtds}{\left(\iint_{P_1}(x_{n+1}-\lambda)^2\e|\nabla u^\e|^2dx\right)^\frac{1}{2}}\\
&=O\left(\iint_{P_1}(x_{n+1}-\lambda)^2\e|\nabla u^\e|^2dx\right)^\frac{1}{2}+o_b(1),
\end{align*}
namely
\begin{align*}
&\int_{-1+b}^{1-b}\iint_{\Pi(\hat P_1)}\left((\partial_t\phi-\hat\Delta\phi)(\bar h^{\e,s}-\lambda)-2\langle\hat\nabla\phi,\hat\nabla (\bar h^s-\lambda)\rangle \right)\psi g'(g^{-1}(s)) d\hat xdtds\\
&=O\left(\iint_{P_1}(x_{n+1}-\lambda)^2\e|\nabla u^\e|^2dx\right)^\frac{1}{2}+o_b(1).
\end{align*}
We compute (using that $\bar h^{\e,s}\rightarrow\bar h^\infty$ in $L^2$ is independent of $s\in(-1+b,1-b)$)
\begin{align*}
&\lim_{\e\rightarrow 0}\int_{-1+b}^{1-b}\iint_{\Pi(\hat P_1)}\left((\partial_t\phi-\hat\Delta\phi)(\bar h^{\e,s}-\lambda)-2\langle\hat\nabla\phi,\hat\nabla (\bar h^s-\lambda)\rangle \right)\psi g'(g^{-1}(s)) d\hat xdtds\\
&=\iint_{\Pi(\hat P_1)}\int_{-1+b}^{1-b}\left((\partial_t\phi-\hat\Delta\phi) \bar h^\infty-2\langle\hat\nabla\phi,\hat\nabla \bar h^\infty\rangle \right)\psi g'(g^{-1}(s)) dsd\hat xdt\\
&=\iint_{\Pi(\hat P_1)}\left((\partial_t\phi-\hat\Delta\phi) \bar h^\infty-2\langle\hat\nabla\phi,\hat\nabla \bar h^\infty\rangle \right)\psi\int_{-1+b}^{1-b}g'(g^{-1}(s)) dsd\hat xdt\\
&=\alpha\iint_{\Pi(\hat P_1)}\left((\partial_t\phi-\hat\Delta\phi) \bar h^\infty-2\langle\hat\nabla\phi,\hat\nabla \bar h^\infty\rangle \right)\psi d\hat xdt+o_b(1)\\
&=\alpha\iint_{\Pi(\hat P_1)}\left((\partial_t\phi+\hat\Delta\phi) \bar h^\infty \right)\psi d\hat xdt+o_b(1).
\end{align*}
where the last line follows by integration by parts and where $\alpha$ is the total energy of the $1$-d static solution. This equation holds for any test function $\phi$. Furthermore the height excess term $O\left(\iint_{P_1}(x_{n+1}-\lambda)^2\e|\nabla u^\e|^2dx\right)^\frac{1}{2}\rightarrow0$ by \eqref{ExcessConvergence} and Proposition \ref{Comparable}. Also since the blowup limit $\bar h^\infty$ is independent of $b$, by letting $b\rightarrow0$ and $\e\rightarrow0$, this limit $\bar h^\infty$ satisfies the heat equation weakly  in $\hat P_{\frac{1}{2}}\subset\mathbb R^n\times\mathbb R$, standard parabolic regularity then shows that it satisfies the heat equation in a strong sense. 
\end{proof}

We recall the following $L^2$ decay property of the heat equation (see for example (8.58) of \cite{Kasai2014}).
\begin{proposition}\label{prop_HeatDecay}
There exists a constant $C_1>0$ such that if $\bar h$ satisfies the heat equation in $\mathbb R^n\times\mathbb R$, then for any $\theta\in(0,\frac{1}{4})$
\begin{align}\label{HeatDecay}
\iint_{\hat B_\theta\times(-\theta^2,\theta^2)}|\bar h(\hat x,t)-\bar h(0,0)-\hat x\cdot\hat\nabla\bar h(0,0)|^2 d\hat xdt
\leq C_1\theta^{n+6}.
\end{align}
\end{proposition}
We are now ready to prove the height excess decay property \eqref{ExcessDecay} by a contradiction argument.
\begin{proof}[Proof of Theorem \ref{thm_ExcessDecay}]
We will prove this theorem by contradiction. Hence we suppose as stated at the beginning of this section that \eqref{SequenceDensityGoingTo1}-\eqref{ExcessDecayNotHold} hold. (Notice by Lemma \ref{ExcessConvergence} and Proposition \ref{Comparable}, \eqref{SequenceDensityGoingTo1} implies \eqref{SequenceHeightGoingTo0}).

We define $T^m\in G(n+1,n)$ to be the tangent space at $x=0$ of the graph 
\begin{align*}
\left(\hat x,\left(\iint_{P_1}x_{n+1}^2\e|\nabla u^{\e_m}|^2dxdt\right)^\frac{1}{2}\cdot \bar h^\infty(\hat x,0)\right)=\left(\hat x,\mathbb H_m \bar h^\infty(\hat x,0)\right),
\end{align*}
where we denote $\mathbb H_m=\left(\iint_{P_1}x_{n+1}^2\e|\nabla u^{\e_m}|^2dxdt\right)^\frac{1}{2}$. Furthermore we define $A^m\in A(n+1,n)$ to be
\begin{align*}
A^m=T^m+\left(0,\mathbb H_m \bar h^\infty(\hat x,0)\right).
\end{align*}
By the definition above, Proposition \ref{TiltAndGradient} and Proposition \ref{Comparable}, we have
\begin{align}\label{PointwiseGradient}
|\nabla\bar h(0,0)|\leq \left(\iint_{P_1}|\nabla\bar h|^2+\bar h^2\right)^\frac{1}{2}\leq C\left(\iint_{P_1}x_{n+1}^2\e|\nabla u^{\e_m}|^2dxdt\right)^\frac{1}{2}.
\end{align}
We want to show a contradiction with \eqref{ExcessDecayNotHold}. By Theorem \ref{thm_L2Bound}, we have strong $L^2$ bounds on the sequence of $h^{\e_m,s}\rightarrow h^\infty$. After passing to a subsequence, we have for sufficiently large $m$
\begin{align*}
&(\mathbb H_m)^{-2}\iint_{P_{\theta}}\mathrm{dist}(x,A^m)^2\e_m | \nabla u^{\e_m}|^2dxdt\\
&=(\mathbb H_m)^{-2}\int_{-1}^1\iint_{\{u^{\e_m}=s\}\cap P_{\theta}}\mathrm{dist}(x,A^m)^2\e_m | \nabla u^{\e_m}|d\mathcal H_{\{u^{\e_m}=s\}}dtds\\
&=(\mathbb H_m)^{-2}\int_{-1}^1\iint_{\{u^{\e_m}=s\}\cap P_{\theta}}|(T^m)^\perp\left((\hat x, h^{\e_m,s})-(0,\mathbb H_m \bar h^\infty(0,0))\right)|^2\e_m | \nabla u^{\e_m}|d\mathcal H_{\{u^{\e_m}=s\}}dtds.
\end{align*}
We now estimate
\begin{align*}
&|(T^m)^\perp\left((\hat x, h^{\e_m,s})-(0,\mathbb H_m \bar h^\infty(0,0))\right)|^2 \\
&=|(T^m)^\perp\left((\hat x, h^{\e_m,s})- (\hat x,\mathbb H_m \bar h^\infty(\hat x,t))+(\hat x,\mathbb H_m \bar h^\infty(\hat x,t))-(0,\mathbb H_m \bar h^\infty(0,0))\right)|^2\\
&\leq 2|(T^m)^\perp\left((\hat x, h^{\e_m,s})-(\hat x,\mathbb H_m \bar h^\infty(\hat x,t))\right)|^2 \\
&+2|(T^m)^\perp\left( (\hat x,\mathbb H_m \bar h^\infty(\hat x,t))-(0,\mathbb H_m \bar h^\infty(0,0))-(\hat x, \langle\hat x,\mathbb H_m \nabla \bar h^\infty(0,0)\rangle)\right)|^2\\
&\leq 2|(T^m)^\perp\left((\hat x, h^{\e_m,s})-(\hat x,\mathbb H_m \bar h^\infty(\hat x,t))\right)|^2 +2|(T^m)^\perp\left( X\right)|^2,
\end{align*}
where the second last line follows because $T^m$ is the tangent space at $(0,0)$ of the graph $\left(\hat x,\mathbb H_m h^\infty(\hat x,0)\right)$ and in the last line we let 
\begin{align*}
X=\left( (\hat x,\mathbb H_m \bar h^\infty(\hat x,t))-(0,\mathbb H_m \bar h^\infty(0,0))-(\hat x, \langle\hat x,\mathbb H_m \nabla \bar h^\infty(0,0)\rangle)\right). 
\end{align*}
This then gives us 

\begin{align*}
&(\mathbb H_m)^{-2}\iint_{P_{\theta}}\mathrm{dist}(x,A^m)^2\e_m | \nabla u^{\e_m}|^2dxdt\\
&\leq 2(\mathbb H_m)^{-2}\int_{-1}^1\iint_{\{u^{\e_m}=s\}\cap P_{\theta}}|(T^m)^\perp\left((\hat x, h^{\e_m,s})-(\hat x,\mathbb H_m \bar h^\infty(\hat x,t))\right)|^2\e_m | \nabla u^{\e_m}|d\mathcal H_{\{u^{\e_m}=s\}}dtds\\
&+2(\mathbb H_m)^{-2}\int_{-1}^1\iint_{\{u^{\e_m}=s\}\cap P_{\theta}}
\left |(T^m)^\perp\left( X \right)\right |^2
\e_m | \nabla u^{\e_m}|d\mathcal H_{\{u^{\e_m}=s\}}dtds.
\end{align*}

Therefore using our Lipschitz parameterisation of the level sets to express the Haussdorff measure, we get
\begin{align*}
&(\mathbb H_m)^{-2}\iint_{P_{\theta}}\mathrm{dist}(x,A^m)^2\e_m|\nabla u^{\e_m}|^2dxdt\\
&\leq 2(\mathbb H_m)^{-2}\int_{-1}^1\iint_{\Pi(\{u=s\})\cap \hat P_{\theta}}|(T^m)^\perp\left((\hat x, h^{\e_m,s})-(\hat x,\mathbb H_m \bar h^\infty(\hat x,t))\right)|^2\sqrt{1+|\nabla h^{\e_m,s}|^2}\e_m|\nabla u^{\e_m}|d\hat{x}dtds\\
&+2(\mathbb H_m)^{-2}\int_{-1}^1\iint_{\Pi(\{u=s\})\cap \hat P_{\theta}}
\left |(T^m)^\perp\left( X \right)\right |^2
\sqrt{1+|\nabla h^{\e_m,s}|^2}\e_m|\nabla u^{\e_m}|d\hat{x}dtds.
\end{align*}
Hence from the $L^2$ convergence of $\frac{h^{\e_m,s}}{\mathbb H_m}$ to $\bar h^\infty(\hat x,t)$ from Theorem \ref{thm_L2Bound} and the uniform Lipschitz bound of $  h^{\e_m,s}$, together with the gradient bound $ \e_m |\nabla u^{\e_m}|\leq C$ we get
\begin{align*}
2(\mathbb H_m)^{-2}&\int_{-1}^1\iint_{\Pi(\{u=s\})\cap \hat P_{\theta}}|(T^m)^\perp\left((\hat x, h^{\e_m,s})-(\hat x,\mathbb H_m \bar h^\infty(\hat x,t))\right)|^2\sqrt{1+|\nabla h^{\e_m,s}|^2}\e_m|\nabla u^{\e_m}|d\hat{x}dtds\\
&=o(1)
\end{align*}
Similarly, the Lipschitz bound on $ h^{\e_m,s}$, the apriori gradient bound on $\e_m |\nabla u^{\e_m}|$ (from standard parabolic regularity) and the heat decay bounds \eqref{HeatDecay} gives us  
\begin{align*}
&2(\mathbb H_m)^{-2}\int_{-1}^1\iint_{\Pi(\{u=s\})\cap \hat P_{\theta}}\left |(T^m)^\perp\left( X \right)\right |^2\sqrt{1+|\nabla h^{\e_m,s}|^2}\e_m|\nabla u^{\e_m}|d\hat{x}dtds \\
&\leq C_1 \theta^{n+6}.
\end{align*}

Therefore we get 
\begin{align*}
(\mathbb H_m)^{-2}\iint_{P_{\theta}}\mathrm{dist}(x,A^m)^2\e_m | \nabla u^{\e_m}|^2dxdt \leq o(1) + C_1 \theta^{n+6}.
\end{align*}
This is a contradiction to \eqref{ExcessDecayNotHold}, 
\begin{align*}
(\mathbb H_m)^{-2}\iint_{P_{\theta}}\mathrm{dist}(x,A^m)^2\e_m | \nabla u^{\e_m}|^2dxdt\geq \theta^{n+5}
\end{align*}
which completes the proof the excess decay.
\end{proof}
\section{The gap theorem}\label{sec_Gap}
Here we prove the rigidity theorem for eternal solutions of the parabolic Allen--Cahen equation \eqref{eqn_ACF1} using the conditional excess decay theorem, Theorem \ref{thm_ExcessDecay} (under the technical condition \eqref{TechnicalAssumption}). The idea is analogous to \cite{wang2014new}: after finite many steps of iterations the technical condition will always be violated for entire eternal solutions. First, we make use the excess decay for solutions to the \eqref{eqn_ACFe} and an appropriate rescaling to obtain an excess growth estimate for solutions to the \eqref{eqn_ACF1}.

\begin{proposition}\label{ExcessGrowthControl}
There exists a constant $C(n)>0$ large and $\tau_n>0$ small such that the following holds: if $u$ is an entire eternal solution of \eqref{eqn_ACF1} in $\mathbb R^{n+1}\times\mathbb R$ such that $\frac{1}{2}R^{-n-2}\iint_{P_R}\left(\frac{|\nabla u|^2}{2}+W(u)\right) dxdt\leq \alpha(1+\tau_n)$ for any $R>0,(x_0,t_0)\in\mathbb R^{n+1}\times\mathbb R$, then for any $r\geq 1$, there exists a unit vector $e_r$, so that the tilt excess has the following growth bound
\begin{align*}
\iint_{P_r}(1-\langle\nu ,e_r\rangle^2)|\nabla u|^2 dxdt\leq C(n)r^n.
\end{align*}
\end{proposition}
\begin{proof}
Since our Theorem \ref{thm_ExcessDecay} is stated in terms of height excess, we will use that to prove a growth bound for height excess. And it will imply the corresponding growth bound for tilt excess.

First, we parabolically scale $u$ by $\e$ (to be chosen later) to get $u^\e(x,t)=u(\frac{x}{\e},\frac{t}{\e^2})$ so that $u^\e$ satisfies the equation \eqref{eqn_ACFe} and which allows us to apply the excess decay theorem, Theorem \ref{thm_ExcessDecay}, to $u^\e$. 

For the $\theta$ obtained from the excess decay theorem, Theorem \ref{thm_ExcessDecay}, we denote $r_k=\frac{1}{\theta^k}$ for $k\geq 0$ and define
\begin{align*}
\mathbb H_k:=\min_{T\in G(n+1,n),\lambda\in\mathbb R} r_k^{-n-2}\iint_{P_{r_k}}|T^\perp(x)-\lambda|^2|\nabla u|^2 dxdt.
\end{align*}
We have that if
\begin{align}\label{ExcessTooLarge}
\mathbb H_k\geq K_1, 
\end{align}
then after rescaling by $\frac{1}{r_k}$, for any $T_k\in G(n+1,n),\lambda_k\in\mathbb R$,
\begin{align*}
\iint_{P_{1}}|T_k^\perp(x)-\lambda_k|^2\frac{1}{r_k}|\nabla u^\frac{1}{r_k}|^2 dxdt\leq K_1\left(\frac{1}{r_k}\right)^2.
\end{align*}
The technical condition \eqref{LayerRepulsion} in Theorem \ref{thm_ExcessDecay} is satisfied for $\e=\frac{1}{r_k}$, and thus the excess decay theorem applies, and there exists $T_{k-1}\in G(n+1,n),\tilde\lambda_{k-1}\in\mathbb R$ such that
\begin{align*}
\theta^{-n-4}\iint_{P_{\theta}}|T_{k-1}^\perp(x)-\tilde\lambda_{k-1}|^2\frac{1}{r_k}|\nabla u^\frac{1}{r_k}|^2 dxdt&\leq \frac{\theta}{2}\iint_{P_{1}}|T_k^\perp(x)-\lambda_k|^2\frac{1}{r_k}|\nabla u^\frac{1}{r_k}|^2 dxdt.
\end{align*}
By rescaling back,
\begin{align*}
r_{k-1}^{-n-4}\iint_{P_{r_{k-1}}}|T_{k-1}^\perp(x)-r_k\tilde\lambda_{k-1}|^2|\nabla u|^2 dxdt&\leq \frac{\theta}{2}r_k^{-n-4}\iint_{P_{r_k}}|T_k^\perp(x)-r_k\lambda_k|^2|\nabla u|^2 dxdt,
\end{align*}
where $\lambda_{k-1}=r_k\tilde\lambda_{k-1}$. Dividing both sides by $r_{k-1}$, we have
\begin{align*}
r_{k-1}^{-n-2}\iint_{P_{r_{k-1}}}|T_{k-1}^\perp(x)-\lambda_{k-1}|^2|\nabla u|^2 dxdt&\leq \frac{\theta^3}{2}r_k^{-n-2}\iint_{P_{r_k}}|T_k^\perp(x)-r_k\lambda_k|^2|\nabla u|^2 dxdt,
\end{align*}
Since $T_k,\lambda_k$ are arbitrary, this gives
\begin{align*}
\mathbb H_{k-1}&\leq \frac{\theta^3}{2}\mathbb H_k.
\end{align*}
Moreover, by definition, for every $k$, we always have
\begin{align*}
\mathbb H_{k-1}\leq \theta^{-n-2}\mathbb H_k.
\end{align*}
We have
\begin{claim}
If for some finite integer $\tilde k$, there holds
\begin{align*}
\mathbb H_{\tilde k}>K_1\theta^{-n-2},
\end{align*}
then for all integers $k\geq \tilde k$, there holds
\begin{align*}
\mathbb H_k\geq \mathbb H_{\tilde k}>K_1\theta^{-n-2}.
\end{align*}
\end{claim}
\begin{proof}[Proof of Claim]
Since $\mathbb H_{\tilde k+1}\geq \theta^{n+2}\mathbb H_k>\theta^{n+2}K_1\theta^{-n-2}=K_1$, condition \eqref{ExcessTooLarge} is satisfied for $\tilde k+1$ and so
\begin{align}\label{ForcedExcessGrowth}
\mathbb H_{\tilde k+1}\geq \frac{2}{\theta^3}\mathbb H_k\geq \mathbb H_{\tilde k}> K_1\theta^{-n-2},
\end{align}
since $\theta\in(0,\frac{1}{4})$.

Now by induction, this holds for every $k\geq \tilde k$.
\end{proof}

Now suppose there is an integer $\tilde k$ such that $\mathbb H_{\tilde k}>K_1\theta^{-n-2}$, then from \eqref{ForcedExcessGrowth} in the proof of the above claim, for every $k\geq \tilde k$,
\begin{align*}
\mathbb H_k\geq \frac{2^{k-\tilde k}}{\theta^{3(k-\tilde k)}}\mathbb H_{\tilde k}=\left(\frac{2}{\theta^3}\right)^k\left[\frac{2^{-\tilde k}}{\theta^{-3\tilde k}}\mathbb H_{\tilde k}\right]\geq \tilde C\frac{2^k}{\theta^{3k}}\mathbb H_{\tilde k}.
\end{align*}
Therefore for each $k\geq \tilde k$, there exists $\bar T_k\in G(n+1,n),\bar\lambda_k\in\mathbb R$ such that
\begin{align*}
\mathbb H_k&\geq \tilde C\frac{2^k}{\theta^{3k}}\theta^{(n+2)\tilde k}\iint_{P_{\theta^{-\tilde k}}}|\bar T_{\tilde k}^\perp(x)-\bar\lambda_{\tilde k}|^2|\nabla u|^2 dxdt \\
&\geq \tilde C\left(\frac{2}{\theta}\right)^k\theta^{-2\tilde k}\theta^{(n+4)\tilde k}\iint_{P_{\theta^{-\tilde k}}}|\bar T_{\tilde k}^\perp(x)-\bar\lambda_{\tilde k}|^2|\nabla u|^2 dxdt \\
&\geq \tilde C_1\left(\frac{2}{\theta}\right)^k\theta^{(n+4)\tilde k}\iint_{P_{\theta^{-\tilde k}}}|\bar T_{\tilde k}^\perp(x)-\bar\lambda_{\tilde k}|^2|\nabla u|^2 dxdt \\
&\geq \tilde C_1\left(\frac{2}{\theta}\right)^k(\theta^{2\tilde k}\mathbb H_{\tilde k})\\
&\geq \tilde C_2\left(\frac{2}{\theta}\right)^k\rightarrow\infty.
\end{align*}
By definition of $\mathbb H_k$, for any $T_k\in G(n+1,n)$ and any $\lambda_k\in\mathbb R$,
\begin{align*}
\theta^{(n+2)k}\iint_{P_{\theta^{-k}}}|T_k^\perp(x)-\lambda_k|^2|\nabla u^\e|^2 dxdt &\geq \tilde C_2\left(\frac{2}{\theta}\right)^k\rightarrow\infty.
\end{align*}
Rescaling by $\e=\frac{1}{r_k}=\theta^k$, we get for any $T_k$
\begin{align*}
\iint_{P_{1}}| T_k^\perp(x)-\theta^k\lambda_k|^2\e|\nabla u^\e|^2 dxdt &\rightarrow\infty,
\end{align*}
as $k\rightarrow\infty$. 

Since $T_k,\lambda_k$ are arbitrary, by Proposition \eqref{Comparable} the tilt-excess satisfies
\begin{align*}
E(P_1)&\rightarrow\infty,
\end{align*}
where $\bar e_k$ is a unit normal to the plane $ T_k$. This is a contradiction to \eqref{ExcessConvergence}, which states that $E(P_1)\rightarrow0$. 

Thus the assumption that $\mathbb H_{\tilde k}>K_1\theta^{-n-2}$ for some $\tilde k$ is false and we must have for every $k>0$, $\mathbb H_k\leq K_1\theta^{-n-2}$, namely there exists  $T_k\in G(n+1,n),\lambda_k\in\mathbb R$ such that
\begin{align*}
\iint_{P_{r_k}}|T_k^\perp(x)-\lambda_k|^2|\nabla u|^2 dxdt\leq K_1\theta^{-n-2}r_k^{n+2}\leq c(n)r_k^{n+2}.
\end{align*}
By interpolating between $r_k$, we easily have for every $r\geq 1$, there exists $T_r\in G(n+1,n),\lambda_r\in\mathbb R$ such that
\begin{align*}
r^{-n-4}\iint_{P_{r}}|T_r^\perp(x)-\lambda_r|^2|\nabla u|^2 dxdt\leq c(n)r^{-2}.
\end{align*}

Scaling by $\frac{1}{r}$ and applying Proposition \eqref{Comparable} in $P_1$, we  have for some $e_r\in\mathbb S^n$
\begin{align*}
r^{-n-2}\iint_{P_{r}}(1-\langle\nu,e_r\rangle^2)|\nabla u|^2 dxdt&\leq C(n)r^{-2}\\
\iint_{P_{r}}(1-\langle\nu,e_r\rangle^2)|\nabla u|^2 dxdt&\leq C(n)r^n.
\end{align*}
which completes the proof.
\end{proof}

Furthermore, we show the choice of this unit vector can be made independent of $r$.
\begin{proposition}\label{ExcessGrowthBound}
There exists a constant $C(n)>0$ large and $\tau_n>0$ small such that the following holds : if $u$ is an entire eternal solution of \eqref{eqn_ACF1} in $\mathbb R^{n+1}\times\mathbb R$ with entropy below $1+\tau_n$, then there exists a unit vector $e_\infty$, so that the tilt excess has the following growth bound
\begin{align*}
\iint_{P_r}(1-\langle\nu ,e_\infty\rangle^2)|\nabla u|^2 dxdt\leq C(n)r^n
\end{align*}
\end{proposition}
\begin{remark}
Recall the tilt excess scales like $r^{n+2}$ for Allen--Cahn equations  so this proposition implies the scale invariant excess $r^{-n-2}\iint_{P_r}(1-\nu_{n+1}^2)|\nabla u|^2 dxdt\leq O(\frac{1}{r^2})$ is decaying.

This decay  implies $u$ is monotonic in any direction $e$ in the cone $\{e\in\mathbb S^n||\langle e,e_{n+1}\rangle|>\delta\}$ for any $\delta>0$. Then a sliding method (moving plane method) argument shows the solution must be the 1-d solution with flat level sets.
\end{remark}

\begin{proof}
We denote by $u^\e(x,t):=u(\frac{x}{\e},\frac{t}{\e^2})$. For any $R\geq 1$, 
\begin{align*}
&\iint_{P_{\frac{R}{\theta}}}(1-\langle\nu_{n+1},e_{\frac{R}{\theta}}\rangle^2)\e|\nabla u^\e|^2 dxdt+\iint_{P_{R}}(1-\langle\nu_{n+1},e_{R}\rangle^2)\e|\nabla u^\e|^2 dxdt\\
&\geq C\iint_{P_R}(\mathrm{dist}_{\mathbb{RP}^n}(\nu_{n+1},e_{R})^2+\mathrm{dist}_{\mathbb{RP}^n}(\nu_{n+1},e_{\frac{R}{\theta}})^2)\e|\nabla u^\e|^2dxdt\\
&\geq C\mathrm{dist}_{\mathbb{RP}^n}(e_{\frac{R}{\theta}},e_{R})^2\iint_{P_R}\e|\nabla u^\e|^2dxdt\\
&\geq C\mathrm{dist}_{\mathbb{RP}^n}(e_{\frac{R}{\theta}},e_{R})^2 R^{n+2},
\end{align*}
where the first inequality is by Cauchy-Schwarz, the second inequality is due to the triangle inequality and the last inequality uses the density lower bound Lemma \ref{DensityLowerBound}. By Proposition \ref{ExcessGrowthControl}, we have
\begin{align*}
\mathrm{dist}_{\mathbb{RP}^n}(e_{\frac{R}{\theta}},e_{R})^2&\leq \frac{C(n)R^n}{R^{n+2}}\leq C(n)R^{-2}.
\end{align*}
And thus
\begin{align*}
\mathrm{dist}_{\mathbb{RP}^n}(e_{\theta^{-k-1}},e_{\theta^{-k}})&\leq c(n)\theta^{k},
\end{align*}
which is a geometric series (since $\theta<\frac{1}{4}$). Summing from $k_0$ to $\infty$, we can choose a $e_\infty$ so that
\begin{align}\label{FixingUnitVector}
\mathrm{dist}_{\mathbb{RP}^n}(e_{\theta^{-k}},e_{\infty})^2\leq \tilde c(n)\theta^{2k_0},
\end{align}
for any $k\geq k_0$.

Interpolating for all $R\in[\theta^{-k},\theta^{-k-1}]$ and integrating we can thus choose a fixed unit vector $e_\infty$ such that
\begin{align*}
\iint_{P_R}(1-\langle\nu_{n+1},e_\infty\rangle^2)\e|\nabla u|^2&\leq \iint_{P_R}(1-\langle\nu_{n+1},e_R\rangle^2)\e|\nabla u|^2+\iint_{P_R}\mathrm{dist}_{\mathbb{RP}^n}(e_R,e_\infty)\e|\nabla u|^2\\
&\leq c(n)R^n+C(n)R^{-2}R^{n+2}\\
&\leq c(n)R^n.
\end{align*}
where the first term is estimated by Proposition \ref{ExcessGrowthControl} and the second term is estimated by interpolating \eqref{FixingUnitVector}.
\end{proof}

Now with the excess growth bound from Proposition \ref{ExcessGrowthBound} and applying the sliding method, we have the following gap (rigidity) theorem 
\begin{theorem}\label{GapTheoremTwo}
There exists $\tau_0>0$ such that if $u:\mathbb R^{n+1}\times\mathbb R$ is an entire eternal solution to \eqref{eqn_ACF1} satisfying the almost unit density condition
\begin{align*}
\frac{1}{2}R^{-n-2}\iint_{P_{R}(0,0)}\left(\frac{\e|\nabla u^\e|^2}{2}+\frac{W(u^\e)}{\e}\right) dxdt\leq (1+\tau(n))\alpha\omega_n, \forall R>0,
\end{align*}
then 
\begin{align*}
u(x,t)=u(x)=g(x\cdot e+s_0)=\tanh\left(x\cdot e+s_0\right)
\end{align*}
for some unit vector $e\in\mathbb S^n$ and $s\in\mathbb R$.
\end{theorem}
\begin{proof}
The proof follows from an application of the sliding method analagous to the elliptic case in \cite{wang2014new}. First, by the excess growth bound Proposition \ref{ExcessGrowthBound}, any blow down sequence 
\begin{align*}
u_{x_0,t_0}^\e(x,t)=u\left(\frac{x-x_0}{\e},\frac{t-t_0}{\e^2}\right)
\end{align*}
converges to the same limit $u^\infty=2\chi_{\{\langle x,e^\infty\rangle\geq 0\}}-1$, where $e^\infty\in\mathbb S^n$ is a fixed unit vector independent of the choice of subsequences $\e$ and translations $(x_0,t_0)$. Without loss of generality, we can assume $e^\infty=e_{n+1}$ and thus $u^\infty=2\chi_{\{x_{n+1}\geq 0\}}-1$.
 
By \eqref{ExcessConvergence}, Proposition \ref{Comparable} and Theorem \ref{Caccioppoli}, the discrepancy and tilt excess converges to $0$ for every time slice.  As a consequence, the proof of \cite[Proposition 4.6]{tonegawa2003integrality} shows the distance type function $z^\e$ defined in \eqref{DistanceTypeFunction} converges in $C^1$ to $\langle e_{n+1},x\rangle$. Namely, for every $\delta>0$, there exists $\e_0(\delta)>0$ such that for every $\e\leq \e_0$ there holds 
\begin{align*}
|\nabla z^\e-e_{n+1}|\leq \delta.
\end{align*}
After rescaling, this reads that for every $\delta>0$, there exists $L(\delta)>0$ such that for $(x,t)\in\{x_{n+1}\geq L\}$ there holds
\begin{align*}
|\nabla z-e_{n+1}|\leq \delta.
\end{align*}
And so $u$ is monotonic increasing in $\{x_{n+1}\geq L\}$ along any direction $e\in\mathbb S^n$ such that $\langle e_{n+1},e\rangle\geq \delta$.

By the sliding method, this monotonicity of $u$ can be extended from $\{x_{n+1}\geq L\}$ to the whole $\mathbb R^{n+1}\times\mathbb R$. (Otherwise, we can translate $u$ in the direction of $e$ for distance $L$ so that the translated graph is disjoint from the untranslated one. If $u$ is not increasing along this direction, we then slide back in $-e$ direction until we get a first interior contact point. This is a violation of the maximum principle for semilinear parabolic partial differential equations).

Since $\delta$ is arbitrary, this then gives $\nabla z=e_{n+1}$ and thus 
\begin{align*}
u=g\circ z=\tanh (x_{n+1}+s_0),
\end{align*}
for some $s_0\in\mathbb R$.
\end{proof}
\section{$C^{2,\alpha}$ regularity of transition layers} \label{sec_Holder}

We  will use a similar iteration argument in Proposition \ref{ExcessGrowthControl} in the proof of gap theorem above to prove a Morrey type bound up to scale $\e$.
\begin{proposition}\label{MorreyType}
There exists constants $K_3,K_4$ such that the following holds: if $u^\e$ is a solution to \eqref{eqn_ACFe} in $P_1\subset\mathbb R^{n+1}\times\mathbb R$ such that $|u^\e(0,0)|\leq 1-b$ and satisfies all the conditions in Theorem \ref{MainRegularityTheorem} then for any $r\in(K_3\e,\theta)$, we can find a unit vector $e_r$ such that 
\begin{align*}
r^{-n-2}\iint_{P_r}(1-\langle\nu_\e,e_r\rangle^2)\e|\nabla u^\e|^2 dxdt\leq K_4\max\left\{\frac{\e^2}{r^2},\delta_0^2r^\alpha\right \},
\end{align*}
where $\alpha=\frac{\log\frac{\theta}{2}}{\log\theta}\in(1,2)$.
\end{proposition}
Similar to Proposition \ref{ExcessGrowthBound}, we can fix the choice of unit vector and obtain
\begin{proposition}\label{ImprovedMorreyType}
For any $\sigma>0$, there exists constants $K_5(\sigma),K_6$ such that the following holds : let $u^\e$ be a solution to \eqref{eqn_ACFe} in $P_1\subset\mathbb R^{n+1}\times\mathbb R$ such that $|u^\e(0,0)|\leq 1-b$ and which satisfies all the conditions in Theorem \ref{MainRegularityTheorem}. Then we can find a unit vector $\bar e_0$ such for any $r\in(K_5\e,\theta)$ we have
\begin{align*}
r^{-n-2}\iint_{P_r}(1-\langle\nu_\e,\bar e_0\rangle^2)\e|\nabla u^\e|^2 dxdt\leq \sigma+K_6\delta_0r^\frac{\alpha}{2},
\end{align*}
where $\alpha=\frac{\log\frac{\theta}{2}}{\log\theta}\in(1,2)$. Here $e_0$ is independent of $r$. Moreover, the difference of $\bar e_0$ and $e_{n+1}$ can be estimated by
\begin{align*}
|\bar e_0-e_{n+1}|\leq C(\sigma^\frac{1}{2}+\delta_0^\frac{1}{2}).
\end{align*}
\end{proposition}
Since proof of Proposition \ref{MorreyType} and Proposition \ref{ImprovedMorreyType} are very similar to that of Proposition \ref{ExcessGrowthControl} and \ref{ExcessGrowthBound} (see also \cite[10.1 and 10.4]{wang2014new} for the proof in elliptic setting), we skip the proof the above two propositions here.

As a consequence of the gap theorem Theorem \ref{GapTheorem}, we have the following interior regularity theorem.
\begin{theorem}\label{InteriorLipschitz}
For any $b\in(0,1)$, $\sigma>0$ small and let $K_5$ be given by Proposition \ref{ImprovedMorreyType} then there exists a $K_7>K_5$ so that the following holds : let $u^\e$ be a solution to \eqref{eqn_ACFe} with $|u(0,0)|\leq 1-b$ and which satisfies the non-positive discrepancy condition \eqref{eqn_discrepancy}. Moreover suppose
\begin{align}\label{InteriorCond1}
\frac{1}{2}(K_7\e)^{-n-2}\iint_{P_{K_7\e}}\left(\frac{\e|\nabla u^\e|^2}{2}+\frac{W(u^\e)}{2}\right) dxdt\leq \alpha(1+\tau(n))\omega_n,
\end{align}
where $\tau(n)$ is as in Theorem \ref{MainRegularityTheorem} and that there exists $\bar e_0$ (obtained from Proposition \ref{ImprovedMorreyType}) such that for any $r\in(K_5\e,K_7\e)$, there holds
\begin{align}\label{InteriorCond2}
\iint_{P_{r}}\left(1-\langle\nu_\e,\bar e_0\rangle^2\right)\e|\nabla u^\e|^2 dxdt\leq \sigma^2r^{n+2}.
\end{align}
Assume that $u^\e>0$ when $x_{n+1}\gg0$. Then we have $\nabla u^\e\neq0$ and
\begin{align*}
|\nu_\e-\bar e_0|\leq \frac{1}{4}
\end{align*}
in  $P_{K_5\e}$. In particular, the level sets of $u$ are uniformly Lipschitz in $B_{K_5\e}$.
\end{theorem}
\begin{proof}
Suppose not, there exists a sequence $\bar K_i\rightarrow\infty$ and a sequence of solutions $u^{\e_i}$ to \eqref{eqn_ACFe} with $\e=\e_i$ satisfying $|u^{\e_i}(0,0)|\leq 1-b$ and non-positive discrepancy condition \eqref{eqn_discrepancy}. Moreover
\begin{align*}
\frac{1}{2}(\bar K_i\e_i)^{-n-2}\iint_{P_{\bar K_i\e_i}}\left(\frac{\e|\nabla u^{\e_i}|^2}{2}+\frac{W(u^{\e_i})}{2}\right) dxdt\leq \alpha(1+\tau_0)\omega_n,
\end{align*}
and that for any $r\in(K_5\e_i,\bar K_i\e_i)$, there holds
\begin{align*}
\iint_{P_{r}}\left(1-\langle\nu_{\e_i},\bar e_0\rangle^2\right)\e_i|\nabla u^{\e_i}|^2 dxdt\leq \sigma^2r^{n+2}.
\end{align*}
But either $\nabla u^{\e_i}=0$ in $P_{\bar K_i\e_i}$ or
\begin{align*}
\sup_{P_{\bar K_i\e_i}}|\nu_{\e_i}-\bar e_0|>\frac{1}{4}.
\end{align*}
Rescaling by $\frac{1}{\e_i}$, we have $\bar u_i(x,t)=u^{\e_i}(\e_ix,\e_i^2t)$ satisfies equation \eqref{eqn_ACF1} with $|u_i(0,0)|\leq 1-b$ and that
\begin{align*}
\frac{1}{2}(\bar K_i)^{-n-2}\iint_{P_{\bar K_i}}\left(\frac{\e|\nabla u_i|^2}{2}+\frac{W(u_i)}{2}\right) dxdt\leq \alpha(1+\tau(n))\omega_n,
\end{align*}
and that for any $r\in(K_5,\bar K_i)$, there holds
\begin{align*}
\iint_{P_{r}}\left(1-\langle\nu,\bar e_0\rangle^2\right)|\nabla u_i|^2 dxdt\leq \sigma^2r^{n+2}.
\end{align*}
Therefore either $\nabla u_i=0$ in $P_{\bar K_i}$ or
\begin{align*}
\sup_{P_{\bar K_i}}|\nu_i-\bar e_0|>\frac{1}{4}.
\end{align*}
Since $\bar K_i\rightarrow\infty$, after passing to a subsequence, we obtain a limit $u_\infty:\mathbb R^{n+1}\times\mathbb R$ which is a solution to \eqref{eqn_ACF1} and satisfies the conditions of Theorem \ref{GapTheorem}. Thus $u_\infty(x,t)=g(e\cdot x+g^{-1}(u_\infty(0,0)))$ for some $e\in\mathbb S^n$. And moreover, by choosing sufficiently small $\sigma$, we have $|e-\bar e_0|\leq \frac{1}{8}$.

For sufficiently large $i$, we then have 
\begin{align*}
|\nu_i-e|\leq \frac{1}{8},
\end{align*}
which gives a contradiction by triangle inequality.
\end{proof}

By Proposition \ref{ImprovedMorreyType}, the conditions in Proposition \ref{InteriorLipschitz} are satisfied and consequently we have uniform Lipschitz regularity.
\begin{corollary}\label{UniformLipschitzRegularity}
For $\tau_0$ as in Theorem \ref{MainRegularityTheorem} and for any $b\in(0,1)$, suppose that $u^\e$ is a solution to \eqref{eqn_ACFe} in the parabolic ball $P_1\subset\mathbb R^{n+1}\times\mathbb R$, with $\e\leq \frac{1}{2K_5}$ ($K_5$ as in Theorem \ref{InteriorLipschitz}), and furthermore let us assume we have $|u^\e(0,0)|\leq 1-b$ and that \eqref{SingleLayer} holds with $\tau\leq \tau_0$. Then for any $(x_0,t_0)\in\{u^\e(x_0,t_0)=u^\e(0,0)\}\cap P_{\frac{1}{2}}$, we have $\nabla u^\e\neq0$ and
\begin{align*}
|\nu_\e- e_{n+1}|\leq \frac{1}{2}
\end{align*}
in $P_{K_5\e}(x_0,t_0)$. In particular we have uniform Lipschitz regularity of the level sets in $P_{\frac{1}{2}}$.
\end{corollary}

\begin{proof}[Proof of Theorem \ref{MainRegularityTheorem}]

Now we apply the improvement of regularity result for intermediate level sets in \cite{Nguyen2020} here, which improves the uniform Lipschitz regularity in Corollary \ref{UniformLipschitzRegularity} to a uniform $C^{2,\alpha}$ regularity. 
\end{proof}

\section{Strong Convergence Theorem}\label{PfNoCancelation}
Let us recall the construction in \cite[1.4]{Ilmanen1993} for $M_0^n\subset\mathbb R^{n+1}$ a smooth strictly mean convex closed hypersurface. This shows that there is a sequence of smooth functions $u_0^\e:\mathbb R^{n+1}\rightarrow[-1,1]$ such that they approximate $M_0$ in the following sense:
\begin{itemize}
\item The discrepancy $\frac{\e|\nabla u^\e|^2}{2}-\frac{W(u^\e)}{\e}\leq 0$;
\item The energy measure $d\mu^\e=\frac{\e|\nabla u^\e|^2}{2}+\frac{W(u^\e)}{\e}\rightarrow\alpha\mathcal H^{n}\llcorner M_0$;
\item $u_0^\e\rightarrow2\chi_{E_0}-1$ in $\mathbf{BV}_{loc}$, where $E_0$ is the region enclosed by $M_0$;
\item There is a uniform density bound $\frac{\mu_o^\e(B_r(x))}{r^n}\leq C$ for any $x\in\mathbb R^{n+
1}, r>0$;
\item $\|u_0^\e\|_{C^2}\leq C(\e)$.
\end{itemize}
After passing to subsequence of solutions $u^{\e_i}:\mathbb R^{n+1}\times\mathbb R^+\rightarrow[-1,1]$ to \eqref{eqn_ACFe} with initial condition $u^\e(\cdot,0)=u^\e_0$ satisfies
\begin{align*}
u^{\e_i}\rightarrow2\chi_{E}-1
\end{align*}
in $\mathbf{BV}_{loc}(\mathbb R^{n+1}\times\mathbb R^+)$ for some open subset $E\subset \mathbb R^{n+1}\times\mathbb R^+, E\cap\mathbb R^{n+1}\times\{0\}=E_0\times\{0\}$. Moreover $d\mu=\lim_{\e_i\rightarrow0} d\mu^{\e_i}=\alpha\mathcal H^{n+1}\llcorner E$ is a Brakke flow in $\mathbb R^{n+1}\times\mathbb R^+$.

For such a sequence of solutions $u^{\e_i}$ satisfying conditions in Theorem \ref{NoCancelation}, we denote the defect measure (cf \cite{Metzger2008})
\begin{align*}
\gamma(t):&=\lim_{\e_i\rightarrow 0}|\nabla u^{\e_i}(\cdot,t)|-|\nabla u(\cdot,t)|\\
\gamma:&=\gamma(t)dt
\end{align*}
Here we have $u=2\chi_{E}-1$ as the limit of $u^{\e_i}$ under BV convergence. Furthermore $|\nabla u|=\mathcal H^{n}\llcorner\partial^*E$ by the theory of BV functions and Caccioppoli set. We want to show $\gamma(t)=0$ a.e. $t$.

First, using White's partial regularity theory \cite{White2000}, we show
\begin{proposition}\label{RegularNoCancelation}
For every $t$, $\spt\gamma(t)\subset S(t)\subset\mathbb R^{n+1}$ where $S$ is a set of Hausdorff dimension at most $n-1$, and thus $\mathcal H^n(\spt(\gamma(t))=0$.
\end{proposition}
\begin{proof}
By the main result of  \cite{Ilmanen1994} and \cite{tonegawa2003integrality}, the sequence of energy measures converges as follows 
\begin{align*}
d\mu^{\e_i}(\cdot,t)=\left(\frac{\e_i|\nabla u^{\e_i}(\cdot,t)|^2}{2}+\frac{W(u^{\e_i}(\cdot,t))}{\e_i}\right) dx\rightarrow \alpha d\mu_t
\end{align*}
where $d\mu_t$ is an integral Brakke flow starting from $d\mu_0=\alpha \mathcal H^n\llcorner M_0$.

Since this limit flow is mean convex, it is non-fattening. Summing up the conclusion in \cite[12.2]{Ilmanen1994} in the non-fattening case, we have 
\begin{align*}
u^{\e_i}\rightarrow u&=2\chi_{E}-1,   \text{in $\mathbf{BV}_{loc}(\mathbb R^{n+1}\times[0,\infty))$}\\
\lim d\mu^{\e_i}&=\alpha d\mu_t=\alpha\mathcal H^n\llcorner\partial^* E_t,   \text{where $E_t=E\cap(\mathbb R^{n+1}\times\{t\})$}\\
\spt d\mu&=\spt d\mu_tdt=\overline{\partial^*E}=\partial^*E   (\mathcal H^{n+1} a.e.)
\end{align*}
Applying the partial regularity theory of \cite[Theorem 1.1]{White2000} to the limit flow $d\mu$, we further conclude
\begin{claim}\label{WhitePartialRegularity}
There is a set $S\subset\spt d\mu$ such that the parabolic Hausdorff measure satisifies $\mathcal H_p^{n-1}(S)=0$ and that $(\spt d\mu)\setminus S$ is a regular $n+1$ dimensional hypersurface in $\mathbb R^{n+1}\times[0,\infty)$.

Let us denote $S(t)=S\cap(\mathbb R^{n+1}\times\{t\})$, then $S(t)$ has Hausdorff dimension at most $n-1$ for every $t$, and have as an immediate consequence, $\mathcal H^{n}(S(t))=0$ for every $t$.
\end{claim} 
Now for any $(y,t)\in\spt d\mu\setminus S$, there exists $r_{y,t}$ such that $\spt d\mu\cap \left(B_{r_{y,t}}\times(t-r^2_{y,t},t+r^2_{y,t})\right)$ is a regular hypersurface. 

By the combination of Theorem \ref{MainRegularityTheorem} and the improvement of regularity results in \cite{Nguyen2020}, we have: up to a choice of coordinate, the level sets of the Allen--Cahn solutions $u^{\e_i}$ in the ball $B_{r_{y,t}}$ can be represented by uniform $C^{2,\theta}$ graphs $\{u^{\e_i}=s\}\cap B_{r_{y,t}}=\{(\hat x,x_{n+1})|x_{n+1}=h^{\e_i,s}(\hat x)\}$. Furthermore $h^{\e_i,s}$ converges in $C^{2,\theta}$ to $h^\infty$ where $\spt d\mu_t\cap B_{r_{y,t}}=\{(\hat x,x_{n+1})|x_{n+1}=h^\infty(\hat x)\}$. And so
\begin{align*}
\int_{B_{r_{y,t}}(y,t)}\phi|\nabla  u^\e(\cdot,t)| dx&=\int_{-1}^1\int_{\{u^\e=s\}\cap B_{r_{y,t}}(y,t)}\phi d\mu_{\{u^\e=s\}}ds\\
&=\int_{-1}^1\int_{\hat B_{r_{y,t}}(\hat y)}\phi\sqrt{1+|\hat\nabla h^{\e,s}|^2} d\mathcal H^nds\\
\rightarrow&\int_{-1}^1\int_{\hat B_{r_{y,t}}(y,t)(\hat y)}\phi\sqrt{1+|\hat\nabla h^\infty|^2} d\mathcal H^nds\\
&=\int_{-1}^1\left(\mathcal H^n(\Gamma\cap B_{r_{y,t}}(y,t))\right)ds\\
&=2\mathcal H^n(\partial^* E\cap B_{r_{y,t}}(y,t))\\
&=\int_{B_{r_{y,t}}(y,t)}\phi \nabla (2\chi_{E}-1)| dx\\
&=\int_{B_{r_{y,t}}(y,t)}\phi|\nabla u(\cdot,t)| dx.
\end{align*}
Thus $\gamma(t)(B_{r_{y,t}}(y,t))=0$ and $\spt\gamma(t)\subset S(t)$.
\end{proof}
Second, we show $\lim|\nabla u^{\e_i}|$, and thus $\gamma$, does not concentrate near $S(t)$.
\begin{proposition}\label{SingularNoCancelation}
For any $\delta>0$, there exists $\overline{\e}_0$ such that, if $\e_i\leq \overline{\e}_0$, then
\begin{align*}
|\nabla u^{\e_i}|\left(S(t)\right)\leq \delta.
\end{align*}
Since $\delta$ is arbitrary and $\gamma(t)\left(S(t)\right)\leq \lim|\nabla u^{\e_i}|\left(S(t)\right)$, we have
\begin{align*}
\gamma(t)\left(S(t)\right)=0.
\end{align*}
\end{proposition}
\begin{proof}
First, by a priori parabolic estimates, there exists $C_1(n)$ such that $\e_i|\nabla u^{\e_i}|\leq C_1(n)$, and thus $|\nabla u^{\e_i}|\leq C_1d\mu^{\e_i}$ as measures. 

By Claim \ref{WhitePartialRegularity} and the definition of Hausdorff measure, for every $\bar\delta>0$, we can cover $S(t)$ by $S\subset\cup_{i=1}^\infty B_{r_{i,\bar\delta}}(x_{i,\bar\delta})$ and that $\sum_i r_{i,\bar\delta}^n<\bar\delta$.

Since the Allen--Cahn solutions under consideration have Euclidean volume growth (see \cite[1.5]{Ilmanen1994}), that is $d\mu_t^{\e_i}(B_{r_{i,\bar\delta}}(x_{i,\bar\delta}))\leq C_2r^n_{i,\bar\delta}$, we have
\begin{align*}
|\nabla u^{\e_i}|(\cup_{i=1}^\infty B_{r_{i,\bar\delta}}(x_{i,\bar\delta}))&\leq C_1d\mu^{\e_i}(\cup_{i=1}^\infty B_{r_{i,\bar\delta}}(x_{i,\bar\delta}))\\
&\leq C_1\sum_id\mu^{\e_i}(B_{r_{i,\bar\delta}}(x_{i,\bar\delta}))\\
&\leq C_1C_2\sum_i r_{i,\bar\delta}^n\\
&\leq C_1C_2\bar\delta.
\end{align*}
Thus by choosing $\bar\delta\leq \frac{\delta}{C_1C_2}$ we complete the proof.
\end{proof}

Proposition \ref{RegularNoCancelation} and Proposition \ref{SingularNoCancelation} together gives $\gamma(t)=0$, which completes the proof of Theorem \ref{NoCancelation}.

\section{Appendix}

\begin{lemma}\label{DensityLowerBound}
For any $b\in(0,1),\Omega_1$, there exists $K_0(b),R_0(b)>0$ so that the following holds: let $u^\e$ be a solution to \eqref{eqn_ACFe} in the parabolic ball $P_{R}$ with $R\geq 2\e R_0$. Suppose that $|u(0,0)|\leq 1-b$ and let us assume 
\begin{align*}
\iint_{P_r}\left(  \frac{\e|\nabla u^{\e}|^2}{2}+\frac{W(u^{\e})}{\e}  \right) dxdt\leq \Omega_1r^{n+2} 
\end{align*}
for any $r\geq 0$.

Then for any $r\in\left [\e,\frac{R}{2}\right]$, we have
\begin{align*}
r^{-n-2}\iint_{P_r}\left(  \frac{\e|\nabla u^{\e}|^2}{2}+\frac{W(u^{\e})}{\e}  \right) dxdt\geq K_0.
\end{align*}
\end{lemma}
\begin{proof}
For the sake of contradiction, let us suppose the statement is false. That is, there exists a sequence $u^{\e_i}$ of solutions to equation \eqref{eqn_ACFe} with $\e_i$ in place of $\e$ in $P_{R_i}\subset\mathbb R^{n+1}\times\mathbb R$ where $\frac{R_i}{\e_i}\rightarrow\infty$, such that $|u^{\e_i}(0,0)|\leq 1-b$ but for some $r_i\in[\e_i,\frac{R_i}{2}]$ we have
\begin{align*}
\sigma_i:=r_i^{-n-2}\iint_{P_{r_i}}\left(  \frac{\e_i|\nabla u^{\e_i}|^2}{2}+\frac{W(u^{\e_i})}{\e_i}  \right) dxdt\rightarrow0.
\end{align*}
After parabolically rescaling by $r_i$, we have a sequence $v^i(x,t)=u^{\e_i}(\frac{x}{r_i},\frac{t}{r_i^2})$ satisfying equation \eqref{eqn_ACFe} with $\e=\hat \e_i=\e_ir_i$. This new sequence satisfies 
\begin{align*}
\sigma_i:=\iint_{P_{1}}\left(  \frac{\hat\e_i|\nabla v^{i}|^2}{2}+\frac{W(v^{i})}{\hat\e_i}  \right) dxdt\rightarrow0.
\end{align*}
If $\hat \e_i=\e_ir_i\rightarrow0$, by \cite{tonegawa2003integrality} the sequence of energy measures converges to an integral Brakke flow and $\sigma_i\rightarrow0$ implies that $(0,0)$ is not in the support of the limit flow. Then by the results in \cite{trumper2008relaxation}, we have
\begin{align*}
|u^{\e_i}(0,0)|=|v^i(0,0)|\rightarrow\pm1,
\end{align*}
which contradicts the assumption that $|u^{\e_i}(0,0)|\leq 1-b$.

If $\hat \e_i=\e_ir_i\rightarrow\e_0>0$, without loss of generality, we can assume $\e_0=1$. In this case $\sigma_i\rightarrow0$ in $P_1$ implies that $v^i\equiv v^i(0,0)$. However, constant solutions satisfy 
\begin{align*}
\iint_{P_r}\left(  \frac{\e|\nabla u^{\e}|^2}{2}+\frac{W(u^{\e})}{\e}  \right) dxdt\geq O(r^{n+3})
\end{align*}
as $r\rightarrow\infty$, which contradicts the assumption $\iint_{P_r}\left(  \frac{\e|\nabla u^{\e}|^2}{2}+\frac{W(u^{\e})}{\e}  \right) dxdt\leq \Omega_1r^{n+2}$.
\end{proof}

\printaddress

\end{document}